\pdfoutput=1
\RequirePackage{ifpdf}
\ifpdf 
\documentclass[pdftex]{sigma}
\else
\documentclass{sigma}
\fi

\usepackage[mathscr]{euscript}
\usepackage{bookmark}
\usepackage{mathabx}
\usepackage{bm}
\usepackage{tikz}
\usetikzlibrary{positioning}

\numberwithin{equation}{section}

\newtheorem{Theorem}{Theorem}[section]
\newtheorem*{Theorem*}{Theorem}

\newtheorem{Lemma}[Theorem]{Lemma}
\newtheorem{Proposition}[Theorem]{Proposition}
 { \theoremstyle{definition}
\newtheorem{Definition}[Theorem]{Definition}

 }

\DeclareMathOperator{\diag}{diag}

\DeclareMathOperator{\gr}{gr}
\DeclareMathOperator{\KR}{KR}
\DeclareMathOperator{\Hom}{Hom}
\DeclareMathOperator{\ch}{ch}
\DeclareMathOperator{\res}{res}

\DeclareMathOperator{\Sym}{Sym}

\begin{document}

\allowdisplaybreaks

\newcommand{\arXivNumber}{2410.08657}

\renewcommand{\PaperNumber}{041}

\FirstPageHeading

\ShortArticleName{Twisted Fusion Products and Quantum Twisted $Q$-Systems}

\ArticleName{Twisted Fusion Products\\ and Quantum Twisted $\boldsymbol{Q}$-Systems}

\Author{Mingyan Simon LIN}

\AuthorNameForHeading{M.S.~Lin}

\Address{Singapore Institute of Manufacturing Technology (SIMTech), Agency for Science,\\ Technology and Research (A*STAR), 5 Cleantech Loop, \#01-01 CleanTech Two Block~B,\\ Singapore 636732, Republic of Singapore}
\Email{\href{mailto:simon_lin@simtech.a-star.edu.sg}{simon\_lin@simtech.a-star.edu.sg}}

\ArticleDates{Received October 14, 2024, in final form May 21, 2025; Published online June 10, 2025}

\Abstract{We obtain a complete characterization of the space of matrix elements dual to the graded multiplicity space arising from fusion products of Kirillov--Reshetikhin modules over special twisted current algebras defined by Kus and Venkatesh, which generalizes the result of Ardonne and Kedem to the special twisted current algebras. We also prove the conjectural identity of $q$-graded fermionic sums by Hatayama et al.\ for the special twisted current algebras, from which we deduce that the graded tensor product multiplicities of the fusion products of Kirillov--Reshetikhin modules over special twisted current algebras are both given by the $q$-graded fermionic sums, and constant term evaluations of products of solutions of the quantum twisted $Q$-systems obtained by Di Francesco and Kedem.}

\Keywords{twisted $Q$-systems; quantum $Q$-systems; Kirillov--Reshetikhin modules; fusion products}

\Classification{17B37; 13F60}

\section{Introduction}\label{Section 1}

\subsection{Overview}\label{Section 1.1}

The Kirillov--Reshetikhin (KR-) modules were first introduced in \cite{KR87} in the context of Bethe ansatz of generalized inhomogeneous Heisenberg spin chains. These KR-modules are irreducible, finite-dimensional modules over the Yangian $Y(\mathfrak{g})$ of a simple classical Lie algebra $\mathfrak{g}$, and these modules satisfy two key properties. Firstly, the $\mathfrak{g}$-characters of these KR-modules satisfy the $Q$-system relations, which is a family of nonlinear recurrence relations~\cite{Kirillov89,KR87}. Secondly, the multiplicities of irreducible, finite-dimensional $\mathfrak{g}$-modules in a tensor product of KR-modules are given by fermionic formulas \cite{KR87}.

Subsequently, the untwisted $Q$-systems arose in the fusion procedure for the transfer matrices of the vertex and the restricted solid-on-solid (RSOS) models associated to Yangian $Y(\mathfrak{g})$ \cite{KNS94}, or equivalently, an untwisted quantum affine algebra $U_q(\widehat{\mathfrak{g}})$. The twisted $Q$-systems then appeared in a subsequent sequel \cite{KS95}, where Kuniba and Suzuki generalized the fusion procedure to the twisted quantum affine algebras.

In \cite{HKOTY99}, Hatayama et al.\ gave combinatorial definitions of $q$-deformations of the fermionic sums, and defined KR-modules over $U_q(\widehat{\mathfrak{g}})$ in terms of Drinfeld polynomials \cite{Chari01} in the untwisted cases. These definitions were then extended to the twisted cases in \cite{HKOTT02}. In addition, they showed in \cite[Theorem~8.1]{HKOTY99}, \cite[Theorem 6.3]{HKOTT02} that if the $U_q(\mathfrak{g})$-characters of the KR-modules over $U_q(\widehat{\mathfrak{g}})$ satisfy the $Q$-system relations, together with some extra asymptotic conditions, then the multiplicity of an irreducible $U_q(\mathfrak{g})$-module in a tensor product of KR-modules over $U_q(\widehat{\mathfrak{g}})$ is given by the extended fermionic sum defined by \cite[equation~(4.16)]{HKOTY99}, \cite[equation~(4.20)]{HKOTT02} at $q=1$. As a first step towards proving the claims advanced by Hatayama et al., Nakajima showed that the $q$-characters of the KR-modules over $U_q(\widehat{\mathfrak{g}})$ satisfy the $T$-system relations in the simply-laced untwisted cases~\cite{Nakajima03}. Subsequently, Hernandez extended the results to the non-simply laced untwisted cases \cite{Hernandez06}, and the twisted cases \cite{Hernandez10}, using a different approach from Nakajima. As the $Q$-system relations are specializations of the $T$-system relations, this implies the fermionic formulas for the tensor product multiplicities.

In the same papers, Hatayama et al.\ also conjectured via \cite[Conjecture~3.1]{HKOTY99}, \cite[Conjecture~3.10]{HKOTT02} that the $q$-grading of the fermionic sums also appears in the context of crystals of tensor products of KR-modules over a quantum affine algebra. Shortly after \cite{HKOTY99}, Kirillov et al.\ proved \cite[Conjecture~3.1]{HKOTY99} for the \smash{$A_r^{(1)}$} case, by establishing a bijection between rigged configurations and crystal paths in type $A$. This was then extended to the \smash{$D_r^{(1)}$} case by Naoi~\cite{Naoi12} via a representation theoretic approach.

Subsequently, further interpretations of the $q$-grading in the $q$-graded fermionic sums were given in the untwisted cases. In \cite{AK07, DFK08}, it was shown that the $q$-grading corresponds to the $\mathfrak{g}$-equivariant grading on the Feigin--Loktev fusion product \cite{FL99} of localized KR-modules over the untwisted current algebra $\mathfrak{g}[t]$, by showing that the graded tensor product multiplicities are given by defined by the $q$-graded fermionic formula given in \cite[equation~(4.3)]{HKOTY99}. A subsequent interpretation was given in \cite{DFK14, Lin21}, where the $q$-grading corresponds to the $q$-grading of quantum $Q$-systems. These quantum $Q$-systems arise naturally as the quantum deformation of the $Q$-system cluster algebras defined in \cite{DFK09, Kedem08}, and were used to yield a complete characterization of the fusion product of KR-modules over current algebras in terms of the quantum $Q$-systems.\looseness=1

A natural question to ask at this point is whether there are analogues of the various interpretations of the $q$-grading of the fermionic sums \cite[equations~(4.5) and (4.20)]{HKOTT02} in the twisted case. In more recent work, Okado et al.\ \cite{OSS18} showed that $q$-grading of the fermionic sums arises in the context of crystals of tensor products of KR-modules in the non-exceptional twisted cases, and Scrimshaw did the same in the exceptional twisted cases described in \cite{Scrimshaw16, Scrimshaw20}, by proving~\cite[Conjecture~3.10]{HKOTT02} in the aforementioned twisted cases.

Our goal in this paper is to extend the results in \cite{AK07, DFK14, Lin21}, and show that the $q$-grading corresponds to the equivariant grading on the fusion product~\cite{KV16} of localized KR-modules over special twisted current algebras, as well as that of the quantum twisted $Q$-systems of type~\smash{$\neq A_{2r}^{(2)}$} defined by Di Francesco and Kedem~\cite{DFK24}. More precisely, we will use the quantum twisted $Q$-system relations to prove the identity \cite[Conjecture~4.3]{HKOTT02} of $q$-graded fermionic sums for all twisted affine types \smash{$\neq A_{2r}^{(2)}$}. We will also show that the graded multiplicities in the fusion product of KR-modules over the special twisted current algebras are given by the $q$-graded fermionic sums defined in \cite[equation~(4.5)]{HKOTT02}. These two results together will then yield a complete characterization of the fusion product of KR-modules over the special twisted current algebras in terms of the quantum twisted $Q$-systems, thereby complementing the characterizations obtained in the untwisted cases \cite{DFK14, Lin21}.

\subsection{Main results}\label{Section 1.2}

To begin, we let $\mathfrak{g}$ be a simply-laced simple Lie algebra of Dynkin type $X_m\neq A_{2r}$, $\sigma$ be a~nontrivial Dynkin diagram automorphism of $\mathfrak{g}$, and $\kappa$ be the order of $\sigma$. We denote the subalgebra of $\sigma$-fixed points of the untwisted current algebra $\mathfrak{g}[t]$ by $\mathfrak{g}[t]^{\sigma}$, and we call $\mathfrak{g}[t]^{\sigma}$ the (special) \cite[Section 1.5]{KV16} twisted current algebra of affine Dynkin type \smash{$X_m^{(\kappa)}$}. We also denote the subalgebra of $\sigma$-fixed points of $\mathfrak{g}$ by $\mathfrak{g}^{\sigma}$, the Dynkin type of $\mathfrak{g}^{\sigma}$ by $Y_r$, and the Cartan matrix of $\mathfrak{g}^{\sigma}$ by $\overline{C}$.

The KR-modules over $\mathfrak{g}[t]^{\sigma}$ are parameterized by $a\in I_r$, where $I_r=[1,r]$ is the set of simple root labels of $\mathfrak{g}^{\sigma}$, $m\in\mathbb{Z}_+$, along with a non-zero localization parameter $z\in\mathbb{C}^*$, and are denoted by $\KR_{a,m}^{\sigma}(z)$. The $\mathfrak{g}^{\sigma}$-characters \smash{$Q_{a,k}=\ch\res_{\mathfrak{g}^{\sigma}}^{\mathfrak{g}[t]^{\sigma}}\KR_{a,k}^{\sigma}(z)$} of these KR-modules over~$\mathfrak{g}[t]^{\sigma}$ satisfy the $X_m^{(\kappa)}$ $Q$-system relations \cite{KS95}, that is, we have
\begin{equation}\label{eq:1.1}
Q_{a,k+1}Q_{a,k-1}=Q_{a,k}^2-\prod_{b\sim a}Q_{b,k}^{-\overline{C}_{ba}},\qquad a\in I_r,\quad k\in\mathbb{N},
\end{equation}
where $b\sim a$ if and only if $\overline{C}_{ba}<0$.

Our first result involves the identity $M_{\overline{\lambda},\mathbf{n}}\bigl(q^{-1}\bigr)=\tilde{M}_{\overline{\lambda},\mathbf{n}}\bigl(q^{-1}\bigr)$ \cite[Conjecture~4.3]{HKOTT02}, which was previously proved for the untwisted simply-laced and untwisted non-simply laced cases in~\cite[Theorem~5.1]{DFK14} and \cite[Theorem~1.2]{Lin21} respectively. To begin, we let $t_1^{\vee},\dots,t_r^{\vee}$ be integers that satisfy~${\min_{a\in I_r}t_a^{\vee}=1}$ and $t_a^{\vee}\overline{C}_{ab}= t_b^{\vee}\overline{C}_{ba}$ for all $a,b\in I_r$. We let $\mathbf{n}=(n_{a,i})_{a\in I_r,i\in\mathbb{N}}$ be a vector that parameterizes a finite set of KR-modules over $\mathfrak{g}[t]^{\sigma}$, where $n_{a,i}$ is the number of KR-modules of type~$\KR_{a,i}^{\sigma}$. In addition, let us fix any dominant $\mathfrak{g}^{\sigma}$-weight $\overline{\lambda}$, and write~\smash{$\overline{\lambda} =\sum_{a\in I_r}\ell_a\overline{\omega}_a$}, where $\overline{\omega}_a$ is the fundamental $\mathfrak{g}^{\sigma}$-weight corresponding to the root label $a$. For any vector~${\mathbf{m}=(m_{a,i})_{a\in I_r,i\in\mathbb{N}}}$ of nonnegative integers with a finite number of nonzero entries, we define the total spin $q_{a,0}$ and the vacancy numbers $p_{a,i}$ as follows:
\begin{gather*}
q_{a,0}=\ell_a+\sum_{j\in\mathbb{N}}\sum_{b\in I_r}j\bigl(\overline{C}_{ab}m_{b,j}-\delta_{ab}n_{b,j}\bigr),\qquad
p_{a,i}=\sum_{j\in\mathbb{N}}\sum_{b\in I_r}\min(i,j)\bigl(\delta_{ab}n_{b,j}-\overline{C}_{ab}m_{b,j}\bigr).
\end{gather*}
Next, we define the quadratic form $Q(\mathbf{m},\mathbf{n})$ by
\[
Q(\mathbf{m},\mathbf{n})=\frac{1}{2}
\sum_{i,j\in\mathbb{N}}\sum_{a,b\in I_r}
t_a^{\vee}\min(i,j)m_{a,i}\bigl(\overline{C}_{ab}m_{b,j}-2\delta_{ab}n_{b,j}\bigr).
\]
The $M$-sum $M_{\overline{\lambda},\mathbf{n}}\bigl(q^{-1}\bigr)$ is given by \cite[equation~(4.5)]{HKOTT02}
\begin{equation}\label{eq:1.5}
M_{\overline{\lambda},\mathbf{n}}\bigl(q^{-1}\bigr)
=\sum_{\substack{\mathbf{m}\geq\mathbf{0}\\q_{a,0}=0,p_{a,i}\geq0}}q^{Q(\mathbf{m},\mathbf{n})}\prod_{i\in\mathbb{N}}\prod_{a\in I_r}
\begin{bmatrix}m_{a,i}+p_{a,i}\\m_{a,i}\end{bmatrix}_{q_a},
\end{equation}
where
\begin{equation*}
\begin{bmatrix}m+p\\m\end{bmatrix}_v=\frac{\bigl(v^{p+1};v\bigr)_{\infty}\bigl(v^{m+1};v\bigr)_{\infty}}{(v;v)_{\infty}\bigl(v^{p+m+1};v\bigr)_{\infty}},\qquad (a;v)_{\infty}=\prod_{j=0}^{\infty}\bigl(1-av^j\bigr),
\end{equation*}
and $q_a=q^{t_a^{\vee}}$ for all $a\in I_r$. Similarly, the $\tilde{M}$-sum $\tilde{M}_{\overline{\lambda},\mathbf{n}}\bigl(q^{-1}\bigr)$, defined without the constraint~${p_{a,i}\geq0}$, is given by \cite[equation~(4.20)]{HKOTT02}
\begin{equation*}
\tilde{M}_{\overline{\lambda},\mathbf{n}}\bigl(q^{-1}\bigr)
=\sum_{\substack{\mathbf{m}\geq\mathbf{0}\\q_{a,0}=0}}q^{Q(\mathbf{m},\mathbf{n})}\prod_{i\in\mathbb{N}}\prod_{a\in I_r}
\begin{bmatrix}m_{a,i}+p_{a,i}\\m_{a,i}\end{bmatrix}_{q_a}.
\end{equation*}
The following theorem implies that the identity $M_{\overline{\lambda},\mathbf{n}}\bigl(q^{-1}\bigr)=\tilde{M}_{\overline{\lambda},\mathbf{n}}\bigl(q^{-1}\bigr)$ \cite[Conjecture 4.3]{HKOTT02} holds for all twisted types not of type \smash{$A_{2r}^{(2)}$}:

\begin{Theorem}\label{1.1}
Let $\mathfrak{g}[t]^{\sigma}$ be a twisted current algebra of type \smash{$X_m^{(\kappa)}\neq A_{2r}^{(2)}$}, $\overline{\lambda}$ be a dominant $\mathfrak{g}^{\sigma}$-weight, and $\mathbf{n}=(n_{a,i})_{a\in I_r,i\in\mathbb{N}}$ be a vector that parameterizes a finite set of KR-modules over~$\mathfrak{g}[t]^{\sigma}$. Then we have
\smash{$
M_{\overline{\lambda},\mathbf{n}}\bigl(q^{-1}\bigr)=\tilde{M}_{\overline{\lambda},\mathbf{n}}\bigl(q^{-1}\bigr)$}.
\end{Theorem}

Here, we would like to remark that Okado et al.\ \cite{OSS18} proved \cite[Conjecture 3.10]{HKOTT02} in the non-exceptional twisted cases, and Scrimshaw did the same in the exceptional twisted cases described in \cite{Scrimshaw16, Scrimshaw20}. These results, along with earlier results by Hernandez \cite{Hernandez10}, show that the conjectural identity \cite[Conjecture~4.3]{HKOTT02} of the $q$-graded fermionic sums $M_{\overline{\lambda},\mathbf{n}}\bigl(q^{-1}\bigr)=\tilde{M}_{\overline{\lambda},\mathbf{n}}\bigl(q^{-1}\bigr)$ holds at $q=1$ in the twisted cases described above.

Our next result involves the description of the graded multiplicities of the irreducible $\mathfrak{g}^{\sigma}$-module $V\bigl(\overline{\lambda}\bigr)$ of highest $\mathfrak{g}^{\sigma}$-weight $\overline{\lambda}$ in a graded tensor product of KR-modules over $\mathfrak{g}[t]^{\sigma}$ in terms of the $M$-sum $M_{\overline{\lambda},\mathbf{n}}(q)$. For each vector $\mathbf{n}=(n_{a,i})_{a\in I_r,i\in\mathbb{N}}$, we let $\mathcal{F}_{\mathbf{n}}^*$ denote the corresponding graded tensor product of twisted KR-modules parameterized by $\mathbf{n}$, equipped with a $\mathfrak{g}^{\sigma}$-equivariant grading, which is called the fusion product of twisted KR-modules \cite{KV16} parameterized by $\mathbf{n}$. Then the graded components $\mathcal{F}_{\mathbf{n}}^*[m]$ of $\mathcal{F}_{\mathbf{n}}^*$ are $\mathfrak{g}^{\sigma}$-modules for all $m\in\mathbb{Z}_+$.

Let us define the generating function $\mathcal{M}_{\overline{\lambda},\mathbf{n}}(q)$ for the graded multiplicities of $V\bigl(\overline{\lambda}\bigr)$ in $\mathcal{F}_{\mathbf{n}}^*$ by
\begin{gather}
\mathcal{M}_{\overline{\lambda},\mathbf{n}}(q)=\sum_{m=0}^{\infty}\dim\Hom_{\mathfrak{g}^{\sigma}}\bigl(\mathcal{F}_{\mathbf{n}}^*[m],V\bigl(\overline{\lambda}\bigr)\bigr)q^m.\label{eq:1.8}
\end{gather}
Here, $\Hom_{\mathfrak{g}^{\sigma}}(\mathcal{F}_{\mathbf{n}}^*[m],V\bigl(\overline{\lambda}\bigr))$ denotes the multiplicity space of $\mathfrak{g}^{\sigma}$-equivariant maps from $\mathcal{F}_{\mathbf{n}}^*[m]$ to $V\bigl(\overline{\lambda}\bigr)$. The graded $\mathfrak{g}^{\sigma}$-character $\ch_q \mathcal{F}_{\mathbf{n}}^*$ of $\mathcal{F}_{\mathbf{n}}^*$ is then defined by
\begin{gather*}
\ch_q \mathcal{F}_{\mathbf{n}}^*= \sum_{\overline{\lambda}}\mathcal{M}_{\overline{\lambda},\mathbf{n}}(q)\ch_{\mathfrak{g}^{\sigma}}V\bigl(\overline{\lambda}\bigr).
\end{gather*}

By extending the tools developed in \cite{AK07} to derive the graded dimension of the space of matrix elements dual to the multiplicity space $\Hom_{\mathfrak{g}^{\sigma}}(\mathcal{F}_{\mathbf{n}}^*,V\bigl(\overline{\lambda}\bigr))$, and using Theorem~\ref{1.1} as well, we arrive at the following fermionic formula for the graded multiplicities $\mathcal{M}_{\overline{\lambda},\mathbf{n}}(q)$, which extends the results obtained in \cite{AK07,DFK08} for the untwisted types to the twisted types not of type \smash{$A_{2r}^{(2)}$}:

\begin{Theorem}\label{1.2}
Let us keep the assumptions as in Theorem {\rm\ref{1.1}}. Then we have
$
\mathcal{M}_{\overline{\lambda},\mathbf{n}}\bigl(q^{-1}\bigr)=M_{\overline{\lambda},\mathbf{n}}\bigl(q^{-1}\bigr)$.
\end{Theorem}

Our last result involves a $q$-graded version of the $X_m^{(\kappa)}$ $Q$-system relations \eqref{eq:1.1}, which generalizes the $q$-graded version of the $Q$-system relations of untwisted type obtained in \cite[Theorem~1.3]{Lin21} to the twisted types not of type \smash{$A_{2r}^{(2)}$}:

\begin{Theorem}\label{1.3}
For all $a\in I_r$ and $m\in\mathbb{N}$, we let
\[
K_{a,m}^{\sigma}= \bigotimes_{b\sim a}(\KR_{b,m}^{\sigma})^{\otimes |\overline{C}_{ba}|},
\]
and we let~$(K_{a,m}^{\sigma})^*$ denote the fusion product corresponding to the tensor product $K_{a,m}^{\sigma}$ of twisted KR-modules. Likewise, we let $\KR_{a,m+1}^{\sigma}*\KR_{a,m-1}^{\sigma}$ and $\KR_{a,m}^{\sigma}*$ $\KR_{a,m}^{\sigma}$ denote the fusion products corresponding to $\KR_{a,m+1}^{\sigma}\otimes\KR_{a,m-1}^{\sigma}$ and $\KR_{a,m}^{\sigma}\otimes\KR_{a,m}^{\sigma}$ respectively. Then the graded $\mathfrak{g}^{\sigma}$-characters of the fusion products of the twisted KR-modules satisfy the following identity:
\[
\ch_q\KR_{a,m+1}^{\sigma}*\KR_{a,m-1}^{\sigma}=
\ch_q \KR_{a,m}^{\sigma}*\KR_{a,m}^{\sigma}
-q^{t_a^{\vee}m}\ch_q (K_{a,m}^{\sigma})^*.
\]
\end{Theorem}

Here, we briefly remark that Kus and Venkatesh obtained a short exact sequence of fusion product of KR-modules in \cite{KV16} that extends the \smash{$X_m^{(\kappa)}$} $Q$-system relations \eqref{eq:1.1}. We will explain the connection between Theorem~\ref{1.3} and their short exact sequences in Section \ref{Section 5}.

\subsection{Outline of the paper}

The paper is organized as follows. In Section \ref{Section 2}, we will review the notion of Kirillov--Reshetikhin modules and fusion products of cyclic modules over both untwisted and twisted current algebras. In Section \ref{Section 3}, we will extend the tools and techniques in \cite{AK07} to the twisted setting, where we will describe the decomposition of fusion products of KR-modules over $\mathfrak{g}[t]^{\sigma}$ into irreducible $\mathfrak{g}^{\sigma}$-modules. Specifically, we will explicitly describe the space of matrix elements dual to the multiplicity space $\Hom_{\mathfrak{g}^{\sigma}}\bigl(\mathcal{F}_{\mathbf{n}}^*,V\bigl(\overline{\lambda}\bigr)\bigr)$ to first deduce that an upper bound for the graded multiplicity~$\mathcal{M}_{\overline{ \lambda},\mathbf{n}}\bigl(q^{-1}\bigr)$ is given by the $M$-sum $M_{\overline{\lambda},\mathbf{n}}\bigl(q^{-1}\bigr)$.

In Section \ref{Section 4}, we will first review the definition and properties of the quantum twisted $Q$-systems given by Di Francesco and Kedem \cite{DFK24}, which are quantum deformations of the cluster transformations corresponding to the twisted $Q$-system relations given by Williams \cite{Williams15}. Subsequently, we will introduce quantum generating functions in the general twisted case whose constant term evaluation is a scalar multiple of the $\tilde{M}$-sum, and derive factorization properties of these generating functions analogous to those in \cite{DFK08, DFK14, Lin21}. Using these factorization properties, along with the Laurent polynomiality property of the solutions of the quantum twisted $Q$-systems, we will then show that $M_{\overline{ \lambda},\mathbf{n}}\bigl(q^{-1}\bigr)=\tilde{M}_{\overline{\lambda},\mathbf{n}}\bigl(q^{-1}\bigr)$ in the twisted cases. Together with the upper bound on the graded multiplicity \smash{$\mathcal{M}_{\overline{ \lambda},\mathbf{n}}\bigl(q^{-1}\bigr)$} derived in Section \ref{Section 3}, we will then show that $\mathcal{M}_{\overline{\lambda},\mathbf{n}}\bigl(q^{-1}\bigr)=M_{\overline{\lambda},\mathbf{n}}\bigl(q^{-1}\bigr)$. Finally, in Section \ref{Section 5}, we will use the results in Section \ref{Section 4} to derive a $q$-graded version of the twisted $Q$-system relations.

\section{Preliminaries}\label{Section 2}

Throughout this section and beyond, we let $\mathfrak{g}$ be a finite-dimensional simply-laced simple Lie algebra of Dynkin type $X_m\neq A_{2r}$. We will also let $\mathfrak{h}$ be a Cartan subalgebra of $\mathfrak{g}$, $\mathfrak{g}=\mathfrak{n}_-\oplus\mathfrak{h}\oplus\mathfrak{n}_+$ be the triangular decomposition of $\mathfrak{g}$ with respect to $\mathfrak{h}$, and $C$ be the Cartan matrix of $\mathfrak{g}$.

In addition, we let $\Delta$ be the set of roots of $\mathfrak{g}$ with respect to $\mathfrak{h}$, and $\Delta^+\subseteq\Delta$ the subset of positive roots. We also fix a basis of simple roots $\Pi=\{\alpha_1,\dots,\alpha_m\}\subseteq\mathfrak{h}^*$ for $\Delta$ and a basis of corresponding simple coroots $\{\alpha_1^{\vee},\dots,\alpha_m^{\vee}\}$ for $\mathfrak{h}^*$ satisfying $\alpha_j(\alpha_i^{\vee})=C_{i,j}$ for all $i\in[1,m]$.
Finally, we let $\{e_{\pm\alpha},\alpha_i^{\vee}\mid \alpha\in\Delta^+,\,i\in[1,m]\}$ be a Chevalley basis for $\mathfrak{g}$, and we denote the positive and negative Chevalley generators for $\mathfrak{g}$ by $e_i:=e_{\alpha_i}$ and $f_i:=e_{-\alpha_i}$ for all $i\in[1,m]$.

Next, we let $P$ denote the weight lattice of $\mathfrak{g}$, and $P^+\subseteq P$ the set of dominant integral weights of $\mathfrak{g}$. The weight lattice $P$ has a basis given by the set $\{\omega_1,\dots,\omega_m\}\subseteq\mathfrak{h}^*$ of fundamental weights of $\mathfrak{g}$, defined by $\omega_i(\alpha_j^{\vee})=\delta_{i,j}$ for all $i,j\in[1,m]$. The irreducible highest weight $\mathfrak{g}$-modules are parameterized by $\lambda\in P^+$, and are denoted by $V({\lambda})$.

\subsection{Untwisted affine and current algebras}\label{Section 2.1}

Let $\mathfrak{g}\bigl[t^{\pm1}\bigr]:=\mathfrak{g}\otimes\mathbb{C}\bigl[t^{\pm1}\bigr]$ denote the untwisted loop algebra of $\mathfrak{g}$, and we denote the current generators of $\mathfrak{g}\bigl[t^{\pm1}\bigr]$ by $x[n]:=x\otimes t^n$ for all $x\in\mathfrak{g}$ and $n\in\mathbb{Z}$. The Lie bracket on $\mathfrak{g}\bigl[t^{\pm1}\bigr]$ is given~by
\smash{$
\bigl[x\otimes t^m,y\otimes t^n\bigr]
=[x,y]\otimes t^{m+n}
$}
for all $x,y\in\mathfrak{g}$ and $m,n\in\mathbb{Z}$. The untwisted loop algebra $\mathfrak{g}\bigl[t^{\pm1}\bigr]$ contains the untwisted current algebra $\mathfrak{g}[t]:=\mathfrak{g}\otimes\mathbb{C}[t]$ of positive currents as a~subalgebra, which in turn contains $\mathfrak{g}$ as a~subalgebra, where we identify $x$ with $x[0]$ for all~${x\in\mathfrak{g}}$.

The untwisted affine algebra \smash{$\widehat{\mathfrak{g}}$} associated with the simple Lie algebra $\mathfrak{g}$ is the central extension of $\mathfrak{g}\bigl[t^{\pm1}\bigr]$ by the central element $K$ associated to the cocycle $\langle\cdot,\cdot\rangle$, where the cocycle $\langle\cdot,\cdot\rangle$ is defined by
$
\big\langle x\otimes t^m,y\otimes t^n\big\rangle = m\delta_{m,-n}(x|y)$
for all $x,y\in\mathfrak{g}$ and $m,n\in\mathbb{Z}$. Here, $(\cdot|\cdot)$ is the symmetric, nondegenerate, invariant bilinear form on $\mathfrak{g}$.

The triangular decomposition of $\widehat{\mathfrak{g}}$ is given by $\widehat{\mathfrak{g}}=\widehat{\mathfrak{n}}_-\oplus\widehat{\mathfrak{h}}\oplus\widehat{\mathfrak{n}}_+$, where $\widehat{\mathfrak{h}}=\mathbb{C}K\oplus\mathfrak{h}$, and~${\widehat{\mathfrak{n}}_{\pm}=\mathfrak{n}_{\pm}\oplus \bigl(\mathfrak{g}\otimes t^{\pm1}\mathbb{C}\bigl[t^{\pm1}\bigr]\bigr)}$.

The irreducible highest weight $\widehat{\mathfrak{g}}$-modules are parameterized by a positive integer $k$, and $\lambda\in P_k^+$, and are denoted by $\widehat{V}_{k,\lambda}$. Here, the integer $k$ is called the level of $\widehat{V}_{k,\lambda}$: the central element $K$ acts on \smash{$\widehat{V}_{k,\lambda}$} by the constant $k$. The set $P_k^+$ is defined by
\begin{equation*}
P_k^+=\Bigg\{\lambda=\sum_{i=1}^m\ell_i\omega_i\in P^+\mid \sum_{i=1}^m\ell_ia_i^{\vee}\leq k\Bigg\},
\end{equation*}
where $a_1^{\vee},\dots,a_m^{\vee}$ are the co-marks of $\widehat{\mathfrak{g}}$.

\subsection{Untwisted current algebra modules}\label{Section 2.2}

\subsubsection{Localization}\label{Section 2.2.1}

Let $V$ be a $\mathfrak{g}[t]$-module on which $\mathfrak{g}[t]$ acts via some representation $\pi_0$. For any $z\in\mathbb{C}^*$, we define the localization of $V$ at $z$ to be the $\mathfrak{g}[t]$-module $V(z)$ whose underlying vector space is $V$, and on which $\mathfrak{g}[t]$ acts by expansion in the local parameter $t_z:=t-z$. More precisely, if we define the Lie algebra map $\varphi_z\colon\mathfrak{g}[t]\to\mathfrak{g}[t]$ by $x\otimes t^n\mapsto x\otimes(t+z)^n$ for all $x\in\mathfrak{g}$ and $n\in\mathbb{Z}_+$, then~$V(z)$ is the pullback of $V$ under $\varphi_z$. By denoting the ``translated'' action of $\mathfrak{g}[t]$ on $V(z)$ by $\pi_z$, it follows that the actions $\pi_0$ and $\pi_z$ of $\mathfrak{g}[t]$ on the vector space $V$ are related to each other by the following equation for all $x\in\mathfrak{g}$, $n\in\mathbb{Z}_+$, and $v\in V$,
\begin{equation}\label{eq:2.1}
\pi_z(x[n])v=\pi_0(x\otimes (t+z)^n)v=\sum_{j=0}^n\binom{n}{j}z^{n-j}\pi_0(x[j])v.
\end{equation}

\subsubsection{Associated graded space of cyclic untwisted current algebra modules}\label{Section 2.2.2}

The degree $d=t\frac{\rm d}{{\rm d}t}$ grading in $t$ on $\mathbb{C}[t]$ induces a natural $\mathbb{Z}_+$-grading on the current algebra $\mathfrak{g}[t]$, and hence also on its universal enveloping algebra $U(\mathfrak{g}[t])$. In particular, for any $j\in\mathbb{Z}_+$, the $j$-th graded component $U(\mathfrak{g}[t])^{(j)}$ of $U(\mathfrak{g}[t])$ is spanned by monomials of the form $x_1[n_1]\cdots x_k[n_k]$, where $k\in\mathbb{Z}_+$, $x_1,\dots,x_k\in\mathfrak{g}$, $n_1,\dots,n_k\in\mathbb{Z}_+$, and $\sum_{i=1}^k n_i=j$. Consequently, the $\mathbb{Z}_+$-grading on $U(\mathfrak{g}[t])$ naturally induces a filtration of $U(\mathfrak{g}[t])$,
\begin{equation*}
U(\mathfrak{g})= U(\mathfrak{g}[t])^{(0)}\subseteq U(\mathfrak{g}[t])^{(\leq 1)}\subseteq U(\mathfrak{g}[t])^{(\leq 2)}\subseteq\cdots,
\end{equation*}
where
\[
U(\mathfrak{g}[t])^{(\leq n)}=\bigoplus_{j=0}^n U(\mathfrak{g}[t])^{(j)}
\]
 for all $n\in\mathbb{Z}_+$.

Let $V$ be a cyclic $\mathfrak{g}[t]$-module with cyclic vector $v$. Then $V$ inherits a filtration (depending on the choice of $v$) from the filtration on $U(\mathfrak{g}[t])$ as follows:
\begin{equation*}
\mathcal{F}(0)(V)\subseteq\mathcal{F}(1)(V)\subseteq\mathcal{F}(2)(V)\subseteq\cdots,
\end{equation*}
where $\mathcal{F}(i)(V)=(U(\mathfrak{g}[t])^{(\leq i)})v$ for all $i\in\mathbb{Z}_+$. Consequently, the associated graded space $\gr V=\bigoplus_{i=0}^{\infty}\mathcal{F}(i)(V)/\mathcal{F}(i-1)(V)$ (where $\mathcal{F}(-1)(V):=\{0\}$) of the above filtration of $V$ inherits a canonical structure of a cyclic graded $\mathfrak{g}[t]$-module, where the $\mathfrak{g}[t]$-action on $\gr V$ is given by~${x[n]\cdot\overline{w}=\overline{x[n]\cdot w}}$ for all $x\in\mathfrak{g}$, $n\in\mathbb{Z}_+$ and $\overline{w}\in\mathcal{F}(i)(V)/\mathcal{F}(i-1)(V)$. As the $\mathbb{Z}_+$-grading on the filtration is $\mathfrak{g}$-equivariant, it follows that the graded components $\mathcal{F}(i)(V)/\mathcal{F}(i-1)(V)$ of~$\gr V$ are $\mathfrak{g}$-modules for all $i\in\mathbb{Z}_+$.

\subsubsection{Fusion product of untwisted current algebra modules}\label{Section 2.2.3}

Let $V_1,\dots,V_N$ be cyclic $\mathfrak{g}[t]$-modules with cyclic vectors $v_1,\dots,v_N$ respectively, and let $z_1,\dots,\allowbreak z_N\in\mathbb{C}$ be pairwise distinct nonzero localization parameters. Feigin and Loktev \cite[Proposition 1.4]{FL99} showed that the tensor product $V_1(z_1)\otimes\cdots\otimes V_N(z_N)$ is a cyclic $\mathfrak{g}[t]$-module with cyclic vector $v_1\otimes\cdots\otimes v_N$. Thus, the tensor product $V_1(z_1)\otimes\cdots\otimes V_N(z_N)$ of localized $\mathfrak{g}[t]$-modules~${V_1(z_1),\dots,V_N(z_N)}$ can be endowed with a $\mathfrak{g}$-equivariant grading, and the resulting graded tensor product, which we denote by $V_1(z_1)*\cdots*V_N(z_N)$, is called the Feigin--Loktev fusion product of untwisted current algebra modules.

\subsection{Kirillov--Reshetikhin modules over untwisted current algebras}\label{Section 2.3}

In this subsection, we will review the construction of Kirillov--Reshetikhin modules over untwisted current algebras, in order to motivate the definition of Kirillov--Reshetikhin modules over twisted current algebras, which we will do in Section \ref{Section 2.5}. While the KR-modules over the untwisted Yangian or the untwisted quantum affine algebra are defined in terms of their Drinfeld polynomials, the KR-modules over the untwisted current algebra $\mathfrak{g}[t]$ are defined in terms of current generators $e_i[n]$, $f_i[n]$, $\alpha_i^{\vee}[n]$ ($i\in[1,m]$, $n\in\mathbb{Z}_+$) and relations, and are the classical limits of the KR-modules over the untwisted quantum affine algebra \cite{CM06, Kedem11}.

\begin{Definition}[{\cite[Definition 2.1]{CM06}}]\label{2.1}
Let $i\in[1,m]$ and $k\in\mathbb{Z}_+$. The KR-module $\KR_{i,k}$ over $\mathfrak{g}[t]$ is the graded $\mathfrak{g}[t]$-module generated by a vector $v$, with relations given by
\begin{gather}
\pi_0(\mathfrak{n}_+[t])v=0,\nonumber\\
\pi_0(f_j[n])v=0,\quad n\geq\delta_{i,j},\nonumber\\
\pi_0(f_i)^{k+1}v=0,\qquad
\pi_0(\alpha_j^{\vee}[n])v=k\delta_{i,j}\delta_{n,0}v.\label{eq:2.2}
\end{gather}
\end{Definition}

Using \eqref{eq:2.1} and \eqref{eq:2.2}, the KR-module $\KR_{i,k}(z)$ is then defined to be the localization of the graded $\mathfrak{g}[t]$-module $\KR_{i,k}$ at $z$, with the relations given by
\begin{subequations}\label{eq:2.3}
\begin{gather}
\pi_z(\mathfrak{n}_+[t]) v=0,\\
\pi_z(f_j[n])v=\delta_{i,j}z^n\pi_0(f_i)v,\label{eq:2.3b}\\
\pi_z(f_i)^{k+1}v=0,\\
\pi_z(\alpha_j^{\vee}[n])v=kz^n\delta_{i,j}v.
\end{gather}
\end{subequations}
For any nonzero $z\in\mathbb{C}^*$, we have that the associated graded space $\gr \KR_{i,k}(z)$ of $\KR_{i,k}(z)$ is isomorphic to $\KR_{i,k}$ as graded $\mathfrak{g}[t]$-modules, and $\KR_{i,k}$ and $\KR_{i,k}(z)$ are isomorphic as $\mathfrak{g}$-modules, but not as $\mathfrak{g}[t]$-modules \cite{CM06}.

In type A, we have $\KR_{i,k}(z)\cong V(k\omega_i)$ as $\mathfrak{g}$-modules, so KR-modules over $\mathfrak{g}[t]$ are irreducible as $\mathfrak{g}$-modules. In general, the KR-module $\KR_{i,k}(z)$ decomposes into irreducible $\mathfrak{g}$-modules as follows \cite{Chari01}
\[
\KR_{i,k}(z)\cong V(k\omega_i)\oplus\Bigl(\bigoplus_{\mu\prec m\omega_i}V(\mu)^{\oplus m_{\mu}}\Bigr),
\]
where $\prec$ is the usual dominance partial ordering on $P$. This decomposition immediately implies that under the restriction of the action to $\mathfrak{g}$, $\KR_{i,k}(z)$ has a highest weight component isomorphic to $V(k\omega_i)$.

\subsection{Twisted affine and current algebras}\label{Section 2.4}

In this subsection, we will review the definition of twisted affine and current algebras. To begin, we let $\overline{\sigma}$ be a nontrivial automorphism of the Dynkin diagram of $\mathfrak{g}$ of order $\kappa>1$, and we let~$r$ denote the number of orbits of $\overline{\sigma}$. The diagram automorphism $\overline{\sigma}$ is described explicitly, as follows:{\samepage
\begin{enumerate}\itemsep=0pt
\item[(1)] If $\mathfrak{g}$ is of type $A_{2r-1}$, then we have $\overline{\sigma}(i)=2r-i$ for all $i\in[1,2r-1]$.
\item[(2)] If $\mathfrak{g}$ is of type $D_{r+1}$, then we have $\overline{\sigma}(i)=i$ for all $i\in[1,r-1]$, $\overline{\sigma}(r)=r+1$ and $\overline{\sigma}(r+1)=r$.
\item[(3)] If $\mathfrak{g}$ is of type $E_6$, then we have $\overline{\sigma}(i)=6-i$ for all $i\in[1,5]$, and $\overline{\sigma}(6)=6$.
\item[(4)] If $\mathfrak{g}$ is of type $D_4$, then we have $\overline{\sigma}(1)=3$, $\overline{\sigma}(2)=2$, $\overline{\sigma}(3)=4$ and $\overline{\sigma}(4)=1$.
\end{enumerate}
The diagram automorphism $\overline{\sigma}$ naturally induces an automorphism $\sigma$ of $\Delta$, and hence of $\mathfrak{g}$.}

{\samepage The following table lists all $\widehat{\mathfrak{g}}^{\sigma}$ and their corresponding $\mathfrak{g}^{\sigma}$:
\begin{center}\renewcommand{\arraystretch}{1.2}
 \begin{tabular}{|c|c|c|c|c|c|}
 \hline
 $\widehat{\mathfrak{g}}^{\sigma}$ & $A_{2r-1}^{(2)}$ & $D_{r+1}^{(2)}$ & $E_{6}^{(2)}$ & $D_4^{(3)}$ \\ \hline
 $\mathfrak{g}^{\sigma}$ & $C_r$ & $B_r$ & $F_4$ & $G_2$ \\
 \hline
 \end{tabular}
\end{center}

}

Next, we let $\xi$ be a $\kappa$-th primitive root of unity. Then $\sigma$ can be extended to an automorphism of $\widehat{\mathfrak{g}}$ by defining $\sigma(K)=K$ and $\sigma(x[n])=\xi^{-n}\sigma(x)[n]$ for all $x\in\mathfrak{g}$ and $n\in\mathbb{Z}$. By \cite[Theorem~8.5]{Kac90}, the affine Dynkin type of $\widehat{\mathfrak{g}}^{\sigma}$ is given by \smash{$X_m^{(\kappa)}$}. We denote the $\sigma$-fixed points of $\mathfrak{g}$, $\mathfrak{g}[t]$ and~$\widehat{\mathfrak{g}}$ by $\mathfrak{g}^{\sigma}$, $\mathfrak{g}[t]^{\sigma}$ and~$\widehat{\mathfrak{g}}^{\sigma}$, respectively.

Our next step is to describe the twisted affine and current algebras in further detail, and review some basic properties concerning the twisted affine and current algebras.

We let $\overline{\Delta}$ be the set of roots of $\mathfrak{g}^{\sigma}$ with respect to the Cartan subalgebra $\mathfrak{h}^{\sigma}$ of $\mathfrak{g}^{\sigma}$, and $\overline{\Delta}^+$ be the set of positive roots of $\mathfrak{g}^{\sigma}$. Then we have
\smash{$
\overline{\Delta}^+=\bigl\{\alpha|_{\mathfrak{h}^{\sigma}}\mid \alpha\in\Delta^+\bigr\}$}.
In particular, a set $\overline{\Pi}$ of simple roots for $\overline{\Delta}$ is given by $\Pi=\{\overline{\alpha}_1,\dots,\overline{\alpha}_r\}\subseteq(\mathfrak{h}^{\sigma})^*$, where $\overline{\alpha}_j= \alpha_j|_{\mathfrak{h}^{\sigma}}$ for all $j\in[1,r]$.

Let $\{\overline{\alpha}_1^{\vee},\dots, \overline{\alpha}_r^{\vee}\}$ be the corresponding basis of simple coroots for $(\mathfrak{h}^{\sigma})^*$ satisfying $\overline{\alpha}_j(\overline{\alpha}_i^{\vee})=\overline{C}_{i,j}$ for all $i\in I_r$. Then for all $j\in I_r$, $\overline{\alpha}_j^{\vee}$ is given by
\begin{gather*}
\overline{\alpha}_j^{\vee}=
\begin{cases}
\alpha_j^{\vee} & \text{if }\overline{\sigma}(j)=j,\\
\displaystyle \sum_{k=0}^{\kappa-1}\alpha_{\overline{\sigma}^k(j)}^{\vee} & \text{if }\overline{\sigma}(j)\neq j.
\end{cases}
\end{gather*}
Next, we will describe a Chevalley basis of $\widehat{\mathfrak{g}}^{\sigma}$ (and hence $\mathfrak{g}[t]^{\sigma}$). To this end, we will first need to describe the set $\widehat{\Delta}^{\sigma}$ of roots of $\widehat{\mathfrak{g}}^{\sigma}$ with respect to $\widehat{\mathfrak{h}}^{\sigma}$. We let $\overline{\Delta}_{\ell}$ and $\overline{\Delta}_s$ denote the set of long and short roots of $\mathfrak{g}^{\sigma}$, respectively. In addition, we also let $\overline{\Delta}_{\ell}^+:= \overline{\Delta}^+ \cap \overline{\Delta}_{\ell}$ and $\overline{\Delta}_s^+= \overline{\Delta}^+\cap\overline{\Delta}_s$. By letting $\delta$ denote the unique non-divisible positive imaginary root of $\widehat{\mathfrak{g}}^{\sigma}$, it follows that the set $\widehat{\Delta}^{\sigma}$ is given by
\[
\widehat{\Delta}^{\sigma}=\{\pm n\delta\mid n\in\mathbb{N}\}\cup\bigl\{\pm\overline{\alpha}+ n\delta\mid \overline{\alpha}\in\overline{\Delta}_s^+,\,n\in\mathbb{Z}\bigr\}\cup\bigl\{\pm\overline{\alpha}+\kappa n\delta\mid \overline{\alpha}\in\overline{\Delta}_{\ell}^+,\,n\in\mathbb{Z}\bigr\}.
\]
We are now ready to describe a Chevalley basis of $\widehat{\mathfrak{g}}^{\sigma}$. As $\widehat{\mathfrak{g}}^{\sigma}=\mathbb{C}K\oplus\mathfrak{g}\bigl[t^{\pm1}\bigr]^{\sigma}$, it suffices to describe a Chevalley basis of $\mathfrak{g}\bigl[t^{\pm1}\bigr]^{\sigma}$. A basis for $\mathfrak{g}\bigl[t^{\pm1}\bigr]^{\sigma}$ is given by the following elements:
\begin{gather*}
\overline{e}_{\pm\overline{\alpha}}[n]=\sum_{i=0}^{\kappa-1}\xi^{-in}e_{\pm\sigma^i(\alpha)}\otimes t^n,\qquad\overline{\alpha}\in\overline{\Delta}_s^+,\\
\overline{e}_{\pm\overline{\alpha}}[\kappa n]=e_{\pm\alpha}\otimes t^{\kappa n},\qquad\overline{\alpha}\in \overline{\Delta}_{\ell}^+,\\
\overline{\alpha}_j^{\vee}[n]=\sum_{i=0}^{\kappa-1}\xi^{-in}\alpha_{\overline{\sigma}^i(j)}^{\vee}\otimes t^n,\qquad i\in I_r\quad\text{with}\ \overline{\sigma}(j)\neq j,\\
\overline{\alpha}_j^{\vee}[\kappa n]=\alpha_j^{\vee}\otimes t^{\kappa n},\qquad i\in I_r\quad\text{with}\ \overline{\sigma}(j)= j,
\end{gather*}
for all $n\in\mathbb{Z}_+$, where $\alpha$ is a root of $\mathfrak{g}$ satisfying $\alpha|_{\mathfrak{h}^{\sigma}} =\overline{\alpha}$.

The roots of $\widehat{\Delta}^{\sigma}$ and the basis elements of the corresponding root space of $\widehat{\mathfrak{g}}^{\sigma}$, are described in the following table:
\begin{center}\renewcommand{\arraystretch}{1.2}
 \begin{tabular}{|c|c|}
 \hline
 root & basis elements \\ \hline
 $\pm\overline{\alpha}+n\delta$, $\overline{\alpha}\in\overline{\Delta}_s^+$, $n\in\mathbb{Z}$ & $\overline{e}_{\pm\overline{\alpha}}[n]$ \\ \hline
 $\pm\overline{\alpha}+\kappa n\delta$, $\overline{\alpha}\in\overline{\Delta}_{\ell}^+$, $n\in\mathbb{Z}$ & $\overline{e}_{\pm\overline{\alpha}}[\kappa n]$ \\ \hline
 $n\delta$, $n\in\mathbb{Z}$, $\kappa\divides n$ & $\overline{\alpha}_j^{\vee}[n]$, $j\in I_r$ \\ \hline
 $n\delta$, $n\in\mathbb{Z}$, $\kappa\not\divides n$ & $\overline{\alpha}_j^{\vee}[n]$, $j\in I_r$ with $\overline{\sigma}(j)\neq j$ \\
 \hline
 \end{tabular}
\end{center}
For later convenience, we will write $\overline{f}_i[t_i^{\vee}n]$ in lieu of $\overline{e}_{-\overline{\alpha}_i}[t_i^{\vee}n]$ for any $n\in\mathbb{Z}$ and $i\in I_r$ \big(note that $t_i^{\vee}=1$ if \smash{$\overline{\alpha}_i\in\overline{\Delta}_s^+$}, and $t_i^{\vee}=\kappa$ if \smash{$\overline{\alpha}_i\in \overline{\Delta}_{\ell}^+$}\big).

Next, we will review a few basic facts and properties concerning twisted affine and current algebras. To begin, we recall that the central element $\overline{K}$ of $\widehat{\mathfrak{g}}^{\sigma}$ is given by $\overline{K}=\kappa K$, and that the triangular decomposition of $\widehat{\mathfrak{g}}$ restricts to the triangular decomposition of $\widehat{\mathfrak{g}}^{\sigma}$, given by~${\widehat{\mathfrak{g}}^{\sigma}=\widehat{\mathfrak{n}}_-^{\sigma}\oplus \widehat{\mathfrak{h}}^{\sigma}\oplus \widehat{\mathfrak{n}}_+^{\sigma}}$, where $\widehat{\mathfrak{h}}^{\sigma}= \mathbb{C}K\oplus\mathfrak{h}^{\sigma}$, and $\widehat{\mathfrak{n}}_{\pm}^{\sigma}=\mathfrak{n}_{\pm}^{\sigma} \oplus(\mathfrak{g}\otimes t^{\pm1}\mathbb{C}\bigl[t^{\pm1}\bigr]) ^{\sigma}$.

We let $\overline{P}$ denote the weight lattice of $\mathfrak{g}^{\sigma}$, and \smash{$\overline{P}^+\subseteq\overline{P}$} the set of dominant integral weights of $\mathfrak{g}^{\sigma}$. The weight lattice $\overline{P}$ has a basis given by the set $\{\overline{\omega}_1,\dots, \overline{\omega}_r\}\subseteq(\mathfrak{h}^{\sigma})^*$ of fundamental weights of $\mathfrak{g}^{\sigma}$, where $\overline{\omega}_j=\omega_j|_{\mathfrak{h}^{\sigma}}$ for all $j\in I_r$. The irreducible highest weight $\mathfrak{g}^{\sigma}$-modules are parameterized by \smash{$\overline{\lambda}\in\overline{P}^+$}, and by an abuse of notation, are denoted by $V\bigl(\overline{\lambda}\bigr)$ as well.

Similar to the untwisted case, the irreducible highest weight $\widehat{\mathfrak{g}}^{\sigma}$-modules are parameterized by a positive integer $k$, and $\overline{\lambda}\in \overline{P}_k^+$, and are denoted by $\widehat{V}_{k,\overline{\lambda}}$. Here, the central element $\overline{K}$ acts on \smash{$\widehat{V}_{k,\overline{\lambda}}$} by the constant $k$, and the set \smash{$\overline{P}_k^+$} is defined by
\begin{equation*}
\overline{P}_k^+=\left\{\overline{\lambda}=\sum_{i=1}^r\ell_i\overline{\omega}_i\in \overline{P}^+\mid \sum_{i=1}^r\ell_i\overline{a}_i^{\vee}\leq k\right\},
\end{equation*}
where $\overline{a}_1^{\vee},\dots,\overline{a}_r^{\vee}$ are the co-marks of $\widehat{\mathfrak{g}}^{\sigma}$.

\subsubsection{Fusion product of twisted current algebra modules}\label{Section 2.4.1}

Our final goal of this subsection is to recall the construction of fusion product of modules over twisted current algebras by Kus and Venkatesh \cite{KV16}. Unlike the untwisted case, the construction of fusion product of modules over twisted current algebras is more involved, as the Lie algebra map $\varphi_z\colon\mathfrak{g}[t]\to\mathfrak{g}[t]$ does not restrict to a Lie algebra map $\mathfrak{g}[t]^{\sigma}\to\mathfrak{g}[t]^{\sigma}$, and hence the localization of twisted current algebra modules cannot be defined using pullbacks via the restriction of $\varphi_z$ to $\mathfrak{g}[t]^{\sigma}$. Nevertheless, the fusion product of twisted current algebra modules can still be defined in the case where the constituent $\mathfrak{g}[t]^{\sigma}$-modules are also $\mathfrak{g}[t]$-modules, which we will describe in detail below the fold.

To begin, we first observe that the $\mathbb{Z}_+$-grading on the universal enveloping algebra $U(\mathfrak{g}[t])$ of $\mathfrak{g}[t]$ naturally induces a $\mathbb{Z}_+$-grading on the universal enveloping algebra $U(\mathfrak{g}[t]^{\sigma})$ of $\mathfrak{g}[t]^{\sigma}$ via restriction, and hence a filtration of $U(\mathfrak{g}[t]^{\sigma})$
\begin{equation*}
U(\mathfrak{g}[t]^{\sigma})= U(\mathfrak{g}[t]^{\sigma})^{(0)}\subseteq U(\mathfrak{g}[t]^{\sigma})^{(\leq 1)}\subseteq U(\mathfrak{g}[t]^{\sigma})^{(\leq 2)}\subseteq\cdots,
\end{equation*}
where $U(\mathfrak{g}[t]^{\sigma})^{(\leq n)}=\bigoplus_{j=0}^n U(\mathfrak{g}[t]^{\sigma})^{(j)}$ for all $n\in\mathbb{Z}_+$. Similar to the untwisted case, this implies that for any cyclic $\mathfrak{g}[t]^{\sigma}$-module $V$ with cyclic vector $v$, $V$ inherits a filtration from the filtration on $U(\mathfrak{g}[t]^{\sigma})$ as follows:
\[
\mathcal{F}(0)(V)\subseteq\mathcal{F}(1)(V)\subseteq\mathcal{F}(2)(V)\subseteq\cdots,
\]
where $\mathcal{F}(i)(V)=\bigl(U(\mathfrak{g}[t]^{\sigma})^{(\leq i)}\bigr)v$ for all $i\in\mathbb{Z}_+$. Consequently, the associated graded space $\gr V=\bigoplus_{i=0}^{\infty}\mathcal{F}(i)(V)/\mathcal{F}(i-1)(V)$ (where $\mathcal{F}(-1)(V):=\{0\}$) of the above filtration of $V$ inherits a canonical structure of a cyclic graded $\mathfrak{g}[t]^{\sigma}$-module, with the $\mathfrak{g}[t]^{\sigma}$-action on $\gr V$ given by~${x[n]\cdot\overline{w}=\overline{x[n]\cdot w}}$
for all $x[n]\in\mathfrak{g}[t]^{\sigma}$ and $\overline{w}\in \mathcal{F}(i)(V)/\mathcal{F}(i-1)(V)$. As the $\mathbb{Z}_+$-grading on the filtration is $\mathfrak{g}^{\sigma}$-equivariant, it follows that the graded components $\mathcal{F}(i)(V)/\mathcal{F}(i-1)(V)$ of~$\gr V$ are $\mathfrak{g}^{\sigma}$-modules for all $i\in\mathbb{Z}_+$.

We are now ready to define the notion of fusion products of twisted current algebra modules. Let us first recall the following twisted analogue of \cite[Proposition 1.4]{FL99}:

\begin{Proposition}[{\cite[Proposition 6.3]{KV16}}]\label{2.2}
Let $V_1,\dots,V_N$ be finite-dimensional cyclic $\mathfrak{g}[t]$-modules with cyclic vectors $v_1,\dots,v_N$ respectively, and $z_1,\dots, z_N\in\mathbb{C}$ be nonzero localization parameters satisfying $z_i^{\kappa}\neq z_j^{\kappa}$ for all distinct $i,j\in[1,N]$. Then $V_1(z_1)\otimes\cdots\otimes V_N(z_N)$ is a cyclic $\mathfrak{g}[t]^{\sigma}$-module with cyclic vector $v_1\otimes\cdots\otimes v_N$.
\end{Proposition}

Thus, by Proposition \ref{2.2}, the tensor product $V_1(z_1)\otimes\cdots\otimes V_N(z_N)$ of localized $\mathfrak{g}[t]^{\sigma}$-modules~${V_1(z_1),\dots,V_N(z_N)}$ can be endowed with a $\mathfrak{g}^{\sigma}$-equivariant grading, and we call the resulting graded tensor product, which we denote similarly by $V_1(z_1)*\cdots*V_N(z_N)$, the fusion product of twisted current algebra modules.

\subsection{Kirillov--Reshetikhin modules over twisted current algebras}\label{Section 2.5}

In this subsection, we will recall the construction of Kirillov--Reshetikhin modules over $\mathfrak{g}[t]^{\sigma}$ and their associated properties given in \cite{CM06, CM07, KV16}. The graded Kirillov--Reshetikhin modules over~$\mathfrak{g}[t]^{\sigma}$ was first defined by Chari and Moura in \cite{CM06, CM07} in an analogous fashion as their untwisted counterparts in terms of the current generators $\overline{e}_i[t_i^{\vee}n]$, $\overline{f}_i[t_i^{\vee}n]$, $\overline{\alpha}_i^{\vee}[t_i^{\vee}n]$ ($i\in I_r$, $n\in\mathbb{Z}_+$) of~$\mathfrak{g}[t]^{\sigma}$ and relations. While localized KR-modules over $\mathfrak{g}[t]^{\sigma}$ cannot be defined directly using pullbacks via the restriction of the Lie algebra map $\varphi_z\colon \mathfrak{g}[t]\to\mathfrak{g}[t]$ to $\mathfrak{g}[t]^{\sigma}$, they can still be defined by restricting the action of localized KR-modules over $\mathfrak{g}[t]$ to $\mathfrak{g}[t]^{\sigma}$ \cite[Section 6.5]{KV16}, and the associated graded space of the localized KR-modules over $\mathfrak{g}[t]^{\sigma}$ are precisely the graded KR-modules over $\mathfrak{g}[t]^{\sigma}$, which we will explain below the fold.

\begin{Definition}[{\cite[Definition 3.3]{CM06} and \cite[Definition 2.2]{CM07}}]\label{2.3}
Let $i\in I_r$ and $k\in\mathbb{Z}_+$. The KR-module $\KR_{i,k}^{\sigma}$ over $\mathfrak{g}[t]^{\sigma}$ is the graded $\mathfrak{g}[t]^{\sigma}$-module generated by a vector $v$, where $\mathfrak{g}[t]^{\sigma}$ acts on~$V$ via a representation $\psi_0$, with relations given by
\begin{subequations}\label{eq:2.4}
\begin{gather}
\psi_0\bigl(\mathfrak{n}_+\bigl[t^{\pm1}\bigr]^{\sigma}\bigr)v =0,\\
\psi_0\bigl(\overline{f}_j[t_j^{\vee}n]\bigr)v =\delta_{i,j}\delta_{n,0}\psi_0\bigl(\overline{f}_i[0]\bigr)v,\label{eq:2.4b}\\
\psi_0\bigl(\overline{f}_i[0]\bigr)^{k+1}v =0,\\
\psi_0\bigl(\overline{\alpha}_j^{\vee}[t_j^{\vee}n]\bigr)v =\delta_{i,j}\delta_{n,0}kv.
\end{gather}
\end{subequations}
\end{Definition}

\begin{Definition}\label{2.4}
Let $i\in [1,m]$, $k\in\mathbb{Z}_+$ and $z\in\mathbb{C}^*$. Let us denote the restriction of the action~$\pi_z$ of $\mathfrak{g}[t]$ on $\KR_{i,k}(z)$ to $\mathfrak{g}[t]^{\sigma}$ by $\psi_z$. We denote the resulting $\mathfrak{g}[t]^{\sigma}$-module by $\KR_{i,k}^{\sigma}(z)$, and we call $\KR_{i,k}^{\sigma}(z)$ a localized KR-module over $\mathfrak{g}[t]^{\sigma}$.
\end{Definition}

The following proposition justifies the notation and definition of $\KR_{i,k}^{\sigma}(z)$ as a localized KR-module over $\mathfrak{g}[t]^{\sigma}$.

\begin{Proposition}[{\cite[Theorem 4]{FL07} and \cite[Propositions 6.6, 6.7 and 7.2]{KV16}}]\label{2.5}
Let $j\in [1,m]$, $k\in\mathbb{Z}_+$, $z\in\mathbb{C}^*$, and $i\in I_r$ be the unique index that satisfies $\overline{\sigma}(i)=\overline{\sigma}(j)$. Then
\begin{equation*}
\gr\KR_{j,k}^{\sigma}(z)\cong \gr\KR_{i,k}^{\sigma}(z)\cong \KR_{i,k}^{\sigma}
\end{equation*}
as graded $\mathfrak{g}[t]^{\sigma}$-modules.
\end{Proposition}

Thus, by Proposition \ref{2.5}, we may restrict our attention to localized KR-modules $\KR_{i,k}^{\sigma}(z)$ over $\mathfrak{g}[t]^{\sigma}$ whose root index $i$ lies in $I_r$.

It remains to describe the relations of $\KR_{i,k}^{\sigma}(z)$ for any $z\in\mathbb{C}^*$, $k\in\mathbb{Z}_+$ and $i\in I_r$. It follows from the relations \eqref{eq:2.3} of the localized KR-module $\KR_{i,k}(z)$ over $\mathfrak{g}[t]$ that the relations of~$\KR_{i,k}^{\sigma}(z)$ are given by{\samepage
\begin{subequations}\label{eq:2.5}
\begin{gather}
\psi_z\bigl(\mathfrak{n}_+\bigl[t^{\pm1}\bigr]^{\sigma}\bigr)v =0,\\
\psi_z\bigl(\overline{f}_j[t_j^{\vee}n]\bigr)v =\delta_{i,j}z^{t_j^{\vee}n}\pi_0(f_i)v,\label{eq:2.5b}\\
\psi_z\bigl(\overline{f}_i[0]\bigr)^{k+1}v =0, \label{eq:2.5c}\\
\psi_z\bigl(\overline{\alpha}_j^{\vee}[t_j^{\vee}n]\bigr)v =\delta_{i,j}z^{t_j^{\vee}n}kv
\end{gather}
\end{subequations}
for all $n\geq0$ and $j\in I_r$.}

Following \eqref{eq:2.3b}, it follows that we can rewrite \eqref{eq:2.5b}, using \eqref{eq:2.4b}, as
\begin{equation}\label{eq:2.6}
\psi_z\bigl(\overline{f}_j[t_j^{\vee}n]\bigr)v =\delta_{i,j}z^{t_j^{\vee}n}\psi_0\bigl(\overline{f}_i[0]\bigr)v
\end{equation}
for all $n\geq0$ and $j\in I_r$.

\section[Fusion products of Kirillov--Reshetikhin modules over twisted current algebras]{Fusion products of Kirillov--Reshetikhin modules\\ over twisted current algebras}\label{Section 3}

Our goal in this section is to establish an upper bound on the graded multiplicity $\mathcal{M}_{\overline{ \lambda},\mathbf{n}}(q)$.

\begin{Theorem}\label{3.1}
Let us keep the assumptions as in Theorem {\rm\ref{1.1}}. Then we have
$\smash{
\mathcal{M}_{\overline{\lambda},\mathbf{n}}}\bigl(q^{-1}\bigr)\leq \smash{M_{\overline{\lambda},\mathbf{n}}}\bigl(q^{-1}\bigr)$,
where the inequality refers to an inequality in the respective coefficients of each power of $q$.
\end{Theorem}

Our strategy in proving Theorem~\ref{3.1} is largely similar to that employed by Ardonne et al.\ in~\cite{AK07,AKS06} for the untwisted cases, where they gave an upper bound for the graded multiplicities of irreducible $\mathfrak{g}$-modules in a fusion product of KR-modules over $\mathfrak{g}[t]$ in terms of $q$-graded fermionic sums. Before we recall the definitions and results needed to prove Theorem~\ref{3.1}, we will first start by recalling some basic results concerning the space of generating functions of matrix elements corresponding to the fusion products of finite-dimensional cyclic modules over untwisted current algebras in the following subsection, following the treatment given in \cite[Section 3.2]{AK07}, before extending the tools and techniques developed in \cite{AK07,AKS06} to the twisted case.

\subsection{Fusion products of current algebra modules and matrix elements}\label{Section 3.1}

\subsubsection{The untwisted case}\label{Section 3.1.1}

Let us first recall the relation between fusion products of finite-dimensional cyclic $\mathfrak{g}[t]$-modules and the fusion product of $\widehat{\mathfrak{g}}$-modules. To begin, we let $V_1,\dots,V_N$ be graded, finite-dimensional cyclic $\mathfrak{g}[t]$-modules with cyclic vectors $v_1,\dots,v_N$ respectively, where the cyclic vectors $v_1,\dots,v_N$ are also highest weight vectors with respect to the $\mathfrak{g}$-action, and $z_1,\dots,z_N\in \mathbb{C}$ be nonzero localization parameters satisfying $z_i\neq z_j$ for all distinct $i,j\in[1,N]$. For convenience, we let~${V:=V_1(z_1)\otimes\cdots\otimes V_N(z_N)}$.

Next, for any positive integer $k$, we let $\widehat{V}_1^k(z_1),\dots,\widehat{V}_N^k(z_N)$ denote the $\widehat{\mathfrak{g}}$-modules induced from the localized $\mathfrak{g}[t]$-modules $V_1(z_1),\dots,V_N(z_N)$ at level $k$ respectively. The fusion product of the $\widehat{\mathfrak{g}}$-modules $\widehat{V}_1^k(z_1),\dots ,\widehat{V}_N^k(z_N)$, denoted $\widehat{V}_1^k(z_1)\boxtimes \cdots\boxtimes\widehat{V}_N^k(z_N)$ \cite{FJKLM04}, is an integrable $\widehat{\mathfrak{g}}$-module of level $k$ (compared to the usual tensor product \smash{$\widehat{V}_1^k(z_1) \otimes\cdots\otimes\widehat{V}_N^k(z_N)$} of \smash{$\widehat{V}_1^k(z_1),\dots,\widehat{V}_N^k(z_N)$}, which is of level $Nk$), where for any $x\otimes f(t)\in \widehat{\mathfrak{g}}$ and
$
w=w_1\otimes\cdots\otimes w_N\in\widehat{V}_1^k(z_1)\boxtimes\cdots\boxtimes\widehat{V}_N^k(z_N)$,
 the element $x\otimes f(t)$ acts on $w$ by the usual coproduct formula, but similar to the setting of localized $\mathfrak{g}[t]$-modules, the action of $x\otimes f(t)$ on the $i$-th component $w_i$ is given by the expansion in the local parameter $t_{z_i}=t-z_i$. For later convenience, we will write $\widehat{V}$ in lieu of~${\widehat{V}_1^k(z_1) \boxtimes\cdots\boxtimes\widehat{V}_N^k(z_N)}$.

As \smash{$\widehat{V}$} is an integrable $\widehat{\mathfrak{g}}$-module of level $k$, \smash{$\widehat{V}$} is completely reducible, and thus \smash{$\widehat{V}$} admits a~decomposition into irreducible $\widehat{\mathfrak{g}}$-modules of level $k$ (we refer the reader to the Appendix of~\cite{FKLMM01} and the introduction of \cite{FJKLM04} for further details). When $k$ is sufficiently large,\footnote{In the case where the modules $V_1,\dots,V_N$ are graded KR-modules over $\mathfrak{g}[t]$, one can explicitly define a lower bound for the level $k$; we refer the reader to the footnote in \cite[Section 5.1]{AK07} for further details.} it follows that the multiplicity of the $\widehat{\mathfrak{g}}$-module $\widehat{V}_{k,\mu}$ in the fusion product $\widehat{V}$ of $\widehat{\mathfrak{g}}$-modules $\widehat{V}_1^k(z_1),\dots,\widehat{V}_N^k(z_N)$ is equal to the multiplicity of the $\mathfrak{g}$-module $V({\mu})$ in the tensor product $V$ of localized $\mathfrak{g}[t]$-modules~${V_1(z_1),\dots,V_N(z_N)}$ for any dominant weight $\mu$ of $\mathfrak{g}$, that is, we have \cite[equation~(3.6)]{AK07}
\begin{gather}\label{eq:3.2}
\dim\Hom_{\widehat{\mathfrak{g}}}\bigl(\widehat{V},\widehat{V}_{k,\mu}\bigr)=\dim\Hom_{\mathfrak{g}}(V, V(\mu)).
\end{gather}
Next, we will describe the dual space \smash{$\mathcal{C}_{ \mu,\widehat{V}}$} of \smash{$\Hom_{\widehat{\mathfrak{g}}} \bigl(\widehat{V},\widehat{V}_{k,\mu}\bigr)$} in terms of generating functions of matrix elements (here, we suppress the integer $k$ from the notation \smash{$\mathcal{C}_{ \mu,\widehat{V}}$}, as it follows from~\eqref{eq:3.2} that we may take $k$ to be sufficiently large in all of our subsequent calculations). We let $v_{\mu^*}$ be a~lowest weight vector of the irreducible lowest weight $\widehat{\mathfrak{g}}$-module \smash{$\widehat{V}_{k,\mu}^*$}, where \smash{$\widehat{V}_{k,\mu}^*$} is the graded dual of \smash{$\widehat{V}_{k,\mu}$}. Let us define the following generating functions of current generators for any~${i\in[1,m]}$:
\begin{gather*}
e_i^{<0}(z)=\sum_{n=-\infty}^{-1}e_i[n]z^{-n-1},\qquad
f_i(z)=\sum_{n\in\mathbb{Z}}f_i[n]z^{-n-1},\qquad
h_i^{<0}(z)=\sum_{n=-\infty}^{-1}\alpha_i^{\vee}[n]z^{-n-1}.
\end{gather*}
Next, we observe that the $\mathfrak{g}[t]$-module $V$ is generated by the action of $U(\mathfrak{n}_-[t])$ on the cyclic vector $v=v_1 \otimes\cdots\otimes v_N$, as the cyclic vectors $v_1,\dots,v_N$ of $V_1(z_1),\dots,V_N(z_N)$ respectively are highest weight vectors with respect to the $\mathfrak{g}$-action. Together with the Poincar\'e--Birkhoff--Witt theorem and the triangular decomposition $\widehat{\mathfrak{g}}= \widehat{\mathfrak{n}}_-\oplus\widehat{\mathfrak{h}}\oplus \widehat{\mathfrak{n}}_+$, it follows that we have
\begin{gather*}
\widehat{V}=U\bigl(\mathfrak{n}_+\otimes t^{-1}\mathbb{C}\bigl[t^{-1}\bigr]\bigr)U\bigl(\mathfrak{h}\otimes t^{-1}\mathbb{C}\bigl[t^{-1}\bigr]\bigr)U\bigl(\mathfrak{n}_-\bigl[t^{\pm1}\bigr]\bigr)v.
\end{gather*}
As the Borel subalgebras $\mathfrak{n}_+$ and $\mathfrak{n}_-$ are generated by $e_1,\dots,e_m$ and $f_1,\dots,f_m$, respectively, it follows that $\mathcal{C}_{ \mu,\widehat{V}}$ consists of generating functions of matrix elements of the following form:
\begin{equation}
\big\langle v_{\mu^*}\big|e_{j_1}^{<0}(y_1)\cdots e_{j_{\ell}}^{<0}(y_{\ell})h_{k_1}^{<0}(u_1)\cdots h_{k_p}^{<0}(u_p)f_{i_1}(x_1)\cdots f_{i_n}(x_n)\big|v\big\rangle,\label{eq:3.3}
\end{equation}
where $\ell,p,n\in\mathbb{Z}_+$, $j_1,\dots,j_{\ell},k_1, \dots,k_p,i_1,\dots,i_n\in[1,m]$, and $y_1,\dots,y_{\ell}, u_1,\dots,u_p,x_1,\dots,x_n$ are formal variables. Moreover, as $v_{\mu^*}$ is a lowest weight vector of \smash{$\widehat{V}_{k,\mu}^*$}, it follows that we have~${\widehat{\mathfrak{n}}_-\cdot v_{\mu^*}=0}$. In particular, we have $e_i[n]\cdot v_{\mu^*}=0=\alpha_i^{\vee}[n]\cdot v_{\mu^*}$ for all $i\in[1,m]$ and $n\in\mathbb{Z}_{<0}$. Thus, it follows from \eqref{eq:3.3} that the space \smash{$\mathcal{C}_{\mu,\widehat{V}}$} only contains generating functions of matrix elements of the following form
$
\langle v_{\mu^*}\big|f_{i_1}(x_1)\cdots f_{i_n}(x_n)\big|v\rangle$,
where $n\in\mathbb{Z}_+$, $i_1,\dots,i_n\in[1,m]$, and~${x_1, \dots,x_n}$ are formal variables, which implies that \smash{$\mathcal{C}_{\mu,\widehat{V}}$} consists of polynomials in the formal variables $x_1, \dots,x_n$. Subsequently, the graded multiplicity of $V({\mu})$ in $V$ is then equal to the graded dimension of the associated graded space \smash{$\gr\mathcal{C}_{\mu,\widehat{V}}$}, where the filtration on \smash{$\mathcal{C}_{\mu,\widehat{V}}$} is inherited from the $\mathbb{Z}$-filtration on the universal enveloping algebra $U\bigl(\mathfrak{n}_-\bigl[t^{\pm1}\bigr]\bigr)$ of $\mathfrak{n}_-\bigl[t^{\pm1}\bigr]$.

\subsubsection{The twisted case}\label{Section 3.1.2}

The approach described above in expressing the graded multiplicity in the fusion product of localized current algebra modules in terms of generating functions of matrix elements for the untwisted case can be extended to the twisted case as well, which we will describe below the fold. Let us keep the notations as above, with the further assumption that we have \smash{$z_i^{\kappa}\neq z_j^{\kappa}$} for all distinct $i,j\in[1,N]$. We recall that for any dominant weights $\mu$ and $\overline{\lambda}$ of $\mathfrak{g}$ and $\mathfrak{g}^{\sigma}$ respectively, the multiplicity of the $\widehat{\mathfrak{g}}^{\sigma}$-module \smash{$\widehat{V}_{\kappa k,\overline{\lambda}}$} in $\widehat{V}_{k,\mu}$ (regarded as a $\widehat{\mathfrak{g}}^{\sigma}$-module of level~$\kappa k$ via restriction, as the central element $\overline{K}$ of $\widehat{\mathfrak{g}}^{\sigma}$ is related to the central element $K$ of $\widehat{\mathfrak{g}}$ by~${\overline{K}= \kappa K}$) is equal to the multiplicity of the $\mathfrak{g}^{\sigma}$-module $V\bigl(\overline{\lambda}\bigr)$ in $V(\mu)$, that is, we have
\begin{equation}\label{eq:3.5}
\dim\Hom_{\widehat{\mathfrak{g}}^{\sigma}}\bigl(\widehat{V}_{k,\mu},\widehat{V}_{\kappa k,\overline{\lambda}}\bigr)=\dim \Hom_{\mathfrak{g}^{\sigma}}\bigl(V(\mu),V\bigl(\overline{\lambda}\bigr)\bigr).
\end{equation}
We claim that when $k$ is sufficiently large, the multiplicity of the $\widehat{\mathfrak{g}}^{\sigma}$-module \smash{$\widehat{V}_{\kappa k,\overline{\lambda}}$} in $\widehat{V}$ is equal to the multiplicity of the $\mathfrak{g}^{\sigma}$-module $V\bigl(\overline{\lambda}\bigr)$ in $V$, that is, we have
\begin{equation*}
\dim\Hom_{\widehat{\mathfrak{g}}^{\sigma}}\bigl(\widehat{V},\widehat{V}_{\kappa k,\overline{\lambda}}\bigr)=\dim \Hom_{\mathfrak{g}^{\sigma}}\bigl(V, V\bigl(\overline{\lambda}\bigr)\bigr).
\end{equation*}
Indeed, as $V$ is finite-dimensional, there are only finitely many dominant $\mathfrak{g}$-weights $\mu$ for which~${\dim\Hom_{\mathfrak{g}}(V, V(\mu))\neq 0}$. Thus, by equations \eqref{eq:3.2} and \eqref{eq:3.5}, we have
\begin{align*}
\dim \Hom_{\mathfrak{g}^{\sigma}}\bigl(V, V\bigl(\overline{\lambda}\bigr)\bigr)
&=\sum_{\mu\in P^+}\dim\Hom_{\mathfrak{g}}(V, V(\mu))\dim \Hom_{\mathfrak{g}^{\sigma}}(V(\mu), V\bigl(\overline{\lambda}\bigr))\\
&=\sum_{\mu\in P^+}\dim\Hom_{\widehat{\mathfrak{g}}}\bigl(\widehat{V},\widehat{V}_{k,\mu}\bigr)\dim\Hom_{\widehat{\mathfrak{g}}^{\sigma}}\bigl(\widehat{V}_{k,\mu},\widehat{V}_{\kappa k,\overline{\lambda}}\bigr)\\
&=\dim\Hom_{\widehat{\mathfrak{g}}^{\sigma}}\bigl(\widehat{V},\widehat{V}_{\kappa k,\overline{\lambda}}\bigr).
\end{align*}
Similar to the untwisted case, the dual space \smash{$\overline{\mathcal{C}}_{\overline{\lambda},\widehat{V}}$} of \smash{$\Hom_{\widehat{\mathfrak{g}}^{\sigma}}\bigl(\widehat{V},\widehat{V}_{\kappa k,\overline{\lambda}}\bigr)$} can be described in terms of generating functions of matrix elements as well. We let \smash{$v_{\overline{\lambda}^*}$} be a lowest weight vector of $\widehat{\mathfrak{g}}^{\sigma}$-module \smash{$\widehat{V}_{\kappa k,\overline{\lambda}}^*$}, where \smash{$\widehat{V}_{\kappa k,\overline{\lambda}}^*$} is the graded dual of \smash{$\widehat{V}_{\kappa k,\overline{\lambda}}$}. Let us define the following generating functions of current generators for any $i\in I_r$:
\begin{equation*}
\overline{f}_i(z)=
\sum_{n\in\mathbb{Z}}\overline{f}_i[t_i^{\vee}n]z^{-t_i^{\vee}(n+1)}.
\end{equation*}
Next, we observe that the $\mathfrak{g}[t]^{\sigma}$-module $V$ is generated by the action of $U(\mathfrak{n}_-[t]^{\sigma})$ on the cyclic vector $v=v_1 \otimes\cdots\otimes v_N$, as the cyclic vectors $v_1,\dots,v_N$ of $V_1(z_1),\dots,V_N(z_N)$ respectively are highest weight vectors with respect to the $\mathfrak{g}^{\sigma}$-action. As the twisted loop algebra $\mathfrak{n}_-\bigl[t^{\pm1}\bigr]^{\sigma}$ is generated by the coefficients of \smash{$\overline{f}_1(x_1),\dots, \overline{f}_r(x_r)$}, it follows from a similar argument as in the untwisted case that $\overline{\mathcal{C}}_{ \overline{\lambda},\widehat{V}}$ consists of generating functions of matrix elements of the following form:
\begin{equation}\label{eq:3.8}
\big\langle v_{\overline{\lambda}^*}\big|\overline{f}_{i_1}(x_1)\cdots \overline{f}_{i_n}(x_n)\big|v\big\rangle,
\end{equation}
where $n\in\mathbb{Z}_+$, $i_1,\dots,i_n\in I_r$, and $x_1, \dots,x_n$ are formal variables. The graded multiplicity of~$V\bigl(\overline{\lambda}\bigr)$ in $V$ is then equal to the graded dimension of the associated graded space \smash{$\gr\overline{ \mathcal{C}}_{\overline{\lambda},\widehat{V}}$}, where the filtration on \smash{$\overline{\mathcal{C}}_{\overline{\lambda}, \widehat{V}}$} is inherited from the $\mathbb{Z}$-filtration on $U\bigl(\mathfrak{n}_-\bigl[t^{\pm1}\bigr]^{\sigma}\bigr)$.

The remainder of this section is devoted to describing the structure of the dual space \smash{$\overline{\mathcal{C}} _{\overline{\lambda},\widehat{V}}$} in the case where $V_1,\dots,V_N$ are graded KR-modules over $\mathfrak{g}[t]$, in which case the associated graded space $\gr V$ of $V$ (with respect to the $\mathbb{Z}_+$-filtration on $U\bigl(\mathfrak{g}\bigl[t^{\pm1}\bigr]^{\sigma}\bigr)$) is given by the fusion product of localized twisted KR-modules. In this case, we let $n_{a,i}$ denote the number of graded $\mathfrak{g}[t]$-modules $V_j$ that are isomorphic to the graded KR-module $\KR_{a,i}$ over $\mathfrak{g}[t]$ for all $a\in I_r$ and~${i\in\mathbb{N}}$, and we will write $\overline{\mathcal{C}}_{\overline{\lambda}, \mathbf{n}}$ in lieu of $\overline{\mathcal{C}}_ {\overline{\lambda},\widehat{V}}$, where $\mathbf{n}=(n_{a,i})_{a\in I_r,i\in\mathbb{N}}$.

The approach that we will take in describing the structure of \smash{$\overline{\mathcal{C}}_{\overline{\lambda},\mathbf{n}}$} in the subsequent subsections will follow that of \cite{AKS06} and \cite{AK07} for the \smash{$A_r^{(1)}$} and the general untwisted cases respectively, which we will describe briefly here. We will first describe the dual space $\mathcal{U}$ of functions to the universal enveloping algebra $U:=U\bigl(\mathfrak{n}_-\bigl[t^{\pm1}\bigr]^{\sigma}\bigr)$ of the twisted loop algebra $\mathfrak{n}_-\bigl[t^{\pm1}\bigr]^{\sigma}$, and introduce a filtration on $\mathcal{U}$. We will then specialize to the
subspace $\overline{\mathcal{C}}_{\overline{\lambda}, \mathbf{n}}$ of $\mathcal{U}$ and its corresponding filtration, from which we will derive Theorem~\ref{3.1}.

\subsection[The dual space of functions to the universal enveloping algebra of the twisted loop algebra]{The dual space of functions to the universal enveloping algebra\\ of the twisted loop algebra}\label{Section 3.2}

Let us take any $\overline{\alpha}\in\overline{\Delta}^+$. When $\overline{\alpha}\in\overline{\Delta}_{\ell}^+$, we define the following generating function $\overline{f}_{\overline{\alpha}}(z)$ of elements in $\mathfrak{n}_-\bigl[t^{\pm1}\bigr]^{\sigma}$ as follows:
\begin{equation*}
\overline{f}_{\overline{\alpha}}(z)=
\sum_{n\in\mathbb{Z}}\overline{e}_{-\overline{\alpha}}[\kappa n]z^{-\kappa(n+1)}.
\end{equation*}
When $\overline{\alpha}\in\overline{\Delta}_s^+$, we define the following generating functions $\overline{f}_{\overline{\alpha},j}(z)$ and $\overline{f}_{\overline{\alpha}}(z)$ of elements in~$\mathfrak{n}_-\bigl[t^{\pm1}\bigr]^{\sigma}$ for all $j\in[0,\kappa-1]$ as follows:{\samepage
\begin{align*}
\overline{f}_{\overline{\alpha},j}(z)
=\sum_{n\in\mathbb{Z}}\overline{e}_{-\overline{\alpha}}[\kappa n+j]z^{-\kappa n-j-1},\qquad
\overline{f}_{\overline{\alpha}}(z)
=\sum_{j=0}^{\kappa-1}\overline{f}_{\overline{\alpha},j}(z)=\sum_{n\in\mathbb{Z}}\overline{e}_{-\overline{\alpha}}[n]z^{-n-1}.
\end{align*}
In particular, we have $\overline{f}_i(z)= \overline{f}_{\overline{\alpha}_i}(z)$ for all $i\in I_r$.}

Similar to the untwisted case, the generating currents $\overline{f}_i(z)$, $i\in I_r$ (or equivalently, $\overline{f}_{\overline{\alpha}}(z)$, $\overline{\alpha}\in \overline{\Pi}$) satisfy two types of operator product expansion (OPE) relations. We will first describe the first type of OPE relations, which arise as a result of the commutation relations between the generating currents of $\mathfrak{n}_-\bigl[t^{\pm1}\bigr]^{\sigma}$.

\begin{Lemma}\label{3.2}
Let $\overline{\alpha},\overline{\beta}\in\overline{\Pi}$ be simple roots that satisfy $\overline{\alpha}+\overline{\beta} \in\overline{\Delta}$. Then up to a sign, the generating currents $\overline{f}_{\overline{\alpha}}(w)$ and $\overline{f}_{\overline{\beta}}(z)$ satisfy the following OPE relation:
\begin{gather}
\overline{f}_{\overline{\alpha}}(w) \overline{f}_{\overline{\beta}}(z)=
\begin{cases}
\dfrac{\overline{f}_{\overline{\alpha}+\overline{\beta}}(z)}{w^{\kappa}-z^{\kappa}}+\text{regular terms} & \text{ if }\overline{\alpha}\in \overline{\Delta}_{\ell},\vspace{1mm}\\
\dfrac{\overline{f}_{\overline{\alpha}+\overline{\beta}}(z)}{w-z}+\text{regular terms} & \text{ if } \overline{\alpha},\overline{\beta}\in\overline{\Delta}_s,\vspace{1mm}\\
\displaystyle \dfrac{1}{w^{\kappa}-z^{\kappa}}\sum_{j=0}^{\kappa-1}\left(\frac{w}{z}\right)^{\kappa-1-j} \overline{f}_{\overline{\alpha}
+\overline{\beta},j}(z)&\\
\qquad {}+\text{regular terms} & \text{ if }\overline{\alpha}\in \overline{\Delta}_s\text{ and }\overline{\beta}\in \overline{\Delta}_{\ell}.
\end{cases}\label{eq:3.12}
\end{gather}
Here, ``regular terms'' refer to terms which have no pole at $w=z$, and the expansion of the denominator is taken in the region $|w|>|z|$.
\end{Lemma}

\begin{proof}
Let us first show that \eqref{eq:3.12} holds when $\overline{\alpha},\overline{\beta}\in\overline{\Delta}_{\ell}$. In this case, we have $\overline{\alpha}+\overline{\beta}\in \overline{\Delta}_{\ell}$, and up to a sign, we have
\begin{align*}
\bigl[\overline{f}_{\overline{\alpha}}(w), \overline{f}_{\overline{\beta}}(z)\bigr]
&=\sum_{m,n\in\mathbb{Z}}[\overline{e}_{-\overline{\alpha}}[\kappa m],\overline{e}_{-\overline{\beta}}[\kappa n]]w^{-\kappa(m+1)}z^{-\kappa(n+1)}\\
&=\sum_{m,n\in\mathbb{Z}}\overline{e}_{-\left( \overline{\alpha}+\overline{\beta}\right)}[\kappa(m+n)]w^{-\kappa(m+1)}z^{-\kappa (n+1)}\\
&=\frac{1}{w^{\kappa}}\sum_{m,j\in\mathbb{Z}}\overline{e}_{-\left(\overline{\alpha}+\overline{\beta}\right)}[\kappa j]z^{-\kappa (j+1)}\left(\frac{z^{\kappa}}{w^{\kappa}}\right)^m\\
&=\frac{1}{w^{\kappa}}\delta(z^{\kappa}/w^{\kappa}) \overline{f}_{\overline{\alpha}+ \overline{\beta}}(z),
\end{align*}
where $\delta(z)=\sum_{m\in\mathbb{Z}}z^m$. This implies that \eqref{eq:3.12} holds when $\overline{\alpha}, \overline{\beta}\in\overline{\Delta}_{\ell}$. By a similar argument as above, it follows that \eqref{eq:3.12} holds when $\overline{\alpha},\overline{\beta}\in \overline{\Delta}_s$.

Next, let us show that \eqref{eq:3.12} holds when $\overline{\alpha}\in\overline{\Delta}_{\ell}$ and $\overline{\beta}\in\overline{\Delta}_s$. In this case, we have $\overline{\alpha}+\overline{\beta}\in \overline{\Delta}_s$, and up to a sign, we have
\begin{align*}
\bigl[\overline{f}_{\overline{\alpha}}(w), \overline{f}_{\overline{\beta}}(z)\bigr]
&=\sum_{m,n\in\mathbb{Z}}[\overline{e}_{-\overline{\alpha}}[\kappa m],\overline{e}_{-\overline{\beta}}[n]]w^{-\kappa(m+1)}z^{-n-1}\\
&=\sum_{m,n\in\mathbb{Z}}\overline{e}_{-\left( \overline{\alpha}+\overline{\beta}\right)}[\kappa m+n]w^{-\kappa(m+1)}z^{-n-1}\\
&=\frac{1}{w^{\kappa}}\sum_{m,j\in\mathbb{Z}}\overline{e}_{-\left(\overline{\alpha}+\overline{\beta}\right)}[j]z^{-j-1}\left(\frac{z^{\kappa}}{w^{\kappa}}\right)^m\\
&=\frac{1}{w^{\kappa}}\delta(z^{\kappa}/w^{\kappa}) \overline{f}_{\overline{\alpha}+ \overline{\beta}}(z).
\end{align*}
Consequently, it follows that \eqref{eq:3.12} holds when $\overline{\alpha}\in\overline{\Delta}_{\ell}$ and $\overline{\beta}\in\overline{\Delta}_s$.

Finally, let us show that \eqref{eq:3.12} holds when $\overline{\alpha}\in\overline{\Delta}_s$ and $\overline{\beta}\in\overline{\Delta}_{\ell}$. Then up to a sign, we have
\begin{align*}
\bigl[\overline{f}_{\overline{\alpha}}(w), \overline{f}_{\overline{\beta}}(z)\bigr]
&=\sum_{m,n\in\mathbb{Z}}[\overline{e}_{-\overline{\alpha}}[m],\overline{e}_{-\overline{\beta}}[\kappa n]]w^{-m-1}z^{-\kappa(n+1)}\\
&=\sum_{k=0}^{\kappa-1}\sum_{m,n\in\mathbb{Z}}[\overline{e}_{-\overline{\alpha}}[\kappa m+k],\overline{e}_{-\overline{\beta}}[\kappa n]]w^{-\kappa m-k-1}z^{-\kappa(n+1)}\\
&=\sum_{k=0}^{\kappa-1}\sum_{m,n\in\mathbb{Z}}\overline{e}_{-\left( \overline{\alpha}+\overline{\beta}\right)}[\kappa(m+n)+k]w^{-\kappa m-k-1}z^{-\kappa(n+1)}\\
&=\frac{1}{w^{\kappa}}\sum_{k=0}^{\kappa-1}\left(\frac{w}{z}\right)^{\kappa-1-k}\sum_{m,j\in\mathbb{Z}}\overline{e}_{-\left(\overline{\alpha}+\overline{\beta} \right)}[\kappa j+k] z^{-\kappa j-k-1}\left(\frac{z^{\kappa}}{w^{\kappa}}\right)^m\\
&=\frac{1}{w^{\kappa}}\delta(z^{\kappa}/w^{\kappa}) \sum_{k=0}^{\kappa-1}\left(\frac{w}{z}\right)^{\kappa-1-k} \overline{f}_{\overline{\alpha}+\overline{\beta},k}(z).
\end{align*}
Consequently, it follows that \eqref{eq:3.12} holds when $\overline{\alpha}\in\overline{\Delta}_s$ and $\overline{\beta}\in\overline{\Delta}_{\ell}$, and this completes the proof.\looseness=1
\end{proof}

Next, we will describe the second type of OPE relations, which arise mostly as a result of the Serre relations for $\mathfrak{g}^{\sigma}$.

\begin{Lemma}\label{3.3}
Let $\overline{\alpha},\overline{\beta}\in\overline{\Pi}$ be simple roots that satisfy $\overline{\alpha}+\overline{\beta} \in\overline{\Delta}$. Then we have the following OPE relation:
\begin{equation}\label{eq:3.13}
(w_1-z)(w_2-z)\overline{f}_{\overline{\alpha}}(w_1) \overline{f}_{\overline{\alpha}}(w_2) \overline{f}_{\overline{\beta}}(z)\big|_{w_1=w_2=z}=0.
\end{equation}
\end{Lemma}

\begin{proof}
When $\overline{\alpha}\in\overline{\Delta}_{\ell}$ or $\overline{\alpha},\overline{\beta}\in\overline{\Delta}_s$, it follows that we have $2\overline{\alpha}+\overline{\beta} \notin\overline{\Delta}$, or equivalently,
\begin{equation*}
\bigl[\overline{f}_{\overline{\alpha}}(w_1),\bigl[\overline{f}_{\overline{\alpha}}(w_2), \overline{f}_{\overline{\beta}}(z)\bigr]\bigr]=0.
\end{equation*}
On the other hand, it follows from \eqref{eq:3.12} that $(w_2-z)\overline{f}_{\overline{\alpha}}(w_2) \overline{f}_{\overline{\beta}}(z)$ consists solely of regular terms. The OPE relation \eqref{eq:3.13} then follows by combining the above two observations.

Next, let us show that the OPE relation \eqref{eq:3.13} holds when $\kappa=2$, $\overline{\alpha}\in\overline{\Delta}_s$ and $\overline{\beta}\in\overline{\Delta}_{\ell}$. Then $\overline{\alpha}+\overline{\beta}\in\overline{\Delta}_s$ and $2\overline{\alpha}+\overline{\beta}\in\overline{\Delta}_{\ell}$, and up to a sign, we have
\begin{equation*}
[\overline{e}_{-\overline{\alpha}}[2m-j],\overline{e}_{-(\overline{\alpha}+\overline{\beta})}[2n+j]]
=2(-1)^j\overline{e}_{-(2\overline{\alpha}+\overline{\beta})}[2(m+n)]
\end{equation*}
for $j\in\{0,1\}$. This implies that for $j\in\{0,1\}$, we have
\begin{align*}
\bigl[\overline{f}_{\overline{\alpha},j}(w_1), \overline{f}_{\overline{\alpha}+\overline{\beta},j}(z)\bigr]
&=\sum_{m,n\in\mathbb{Z}}[\overline{e}_{-\overline{\alpha}}[2m-j],\overline{e}_{-(\overline{\alpha}+\overline{\beta})}[2n+j]] w_1^{-2m-1+j}z^{-2n-1-j}\\
&=(-1)^j\sum_{m,n\in\mathbb{Z}}\overline{e}_{-(2\overline{\alpha}+\overline{\beta})}[2(m+n)]w_1^{-2m-1+j} z^{-2n-1-j}\\
&=(-1)^j\left(\frac{z}{w_1}\right)^{1-j} \sum_{m,k\in\mathbb{Z}}\overline{e}_{-(2\overline{\alpha}+ \overline{\beta})}[2k]z^{-2k-2}\left(\frac{z^2}{w_1^2} \right)^m\\
&=(-1)^j\left(\frac{z}{w_1}\right)^{1-j}\delta\big(z^2/w_1^2\big) \overline{f}_{2\overline{\alpha}+\overline{\beta}}(z),
\end{align*}
from which we deduce that
\begin{equation}\label{eq:3.14}
\overline{f}_{\overline{\alpha},j}(w_1)\overline{f}_{\overline{\alpha}+\overline{\beta},j}(z)
=\frac{(-1)^jz^{1-j}w_1^{1+j}\overline{f}_{2\overline{\alpha}+ \overline{\beta}}(z)}{z^2-w_1^2}+\text{regular terms}.
\end{equation}
Consequently, we deduce from \eqref{eq:3.12} and \eqref{eq:3.14} that we have
\begin{equation}\label{eq:3.15}
\overline{f}_{\overline{\alpha}}(w_1) \overline{f}_{\overline{\alpha}}(w_2)
\overline{f}_{\overline{\beta}}(z)
=\frac{w_1(w_2-w_1)\overline{f}_{2\overline{\alpha}+ \overline{\beta}}(z)}{\bigl(z^2-w_1^2\bigr)\bigl(z^2-w_2^2\bigr)}+\text{regular terms},
\end{equation}
and thus $(w_1-z)(w_2-z)\overline{f}_{\overline{\alpha}}(w_1) \overline{f}_{\overline{\alpha}}(w_2) \overline{f}_{\overline{\beta}}(z)$ consists solely of regular terms. As the terms $w_1(w_2-w_1)$ and $(w_1-z)(w_2-z)$ both vanish when $w_1=w_2=z$, it follows from \eqref{eq:3.15} that the OPE relation \eqref{eq:3.13} holds when $\kappa=2$, $\overline{\alpha}\in\overline{\Delta}_s$ and $\overline{\beta}\in\overline{\Delta}_{\ell}$.

Finally, let us show that the OPE relation \eqref{eq:3.13} holds when $\kappa=3$, $\overline{\alpha}\in\overline{\Delta}_s$ and $\overline{\beta}\in\overline{\Delta}_{\ell}$. Then $\overline{\alpha}+\overline{\beta},2\overline{\alpha}+\overline{\beta}\in\overline{\Delta}_s$, and up to a sign, we have
\begin{equation*}
[\overline{e}_{-\overline{\alpha}}[3m+i-j],\overline{e}_{-(\overline{\alpha}+\overline{\beta})}[3n+j]]
=\bigl(\xi^{j-i}+\xi^{-j}\bigr)\overline{e}_{-(2\overline{\alpha}+ \overline{\beta})}[3(m+n)+i]
\end{equation*}
for all $i,j\in\{0,1,2\}$. By letting $\overline{i-j}$ be the unique element in $\{0,1,2\}$ that satisfies $i-j\equiv \overline{i-j}\mod 3$ for all $i,j\in\{0,1,2\}$, it follows that we have
\begin{align*}
\bigl[\overline{f}_{\overline{\alpha},\overline{i-j}}(w_1), \overline{f}_{\overline{\alpha}+\overline{\beta},j}(z)\bigr]
&=\sum_{m,n\in\mathbb{Z}}[\overline{e}_{-\overline{\alpha}}[3m+i-j],\overline{e}_{-(\overline{\alpha}+\overline{\beta})}[3n+j]]w_1^{-3m-1-i+j}z^{-3n-1-j}\\
&=\bigl(\xi^{j-i}+\xi^{-j}\bigr)\sum_{m,n\in\mathbb{Z}}\overline{e}_{-(2\overline{\alpha}+\overline{\beta})}[3(m+n)+i]w_1^{-3m-1-i+j}z^{-3n-1-j}\\
&=\bigl(\xi^{j-i}+\xi^{-j}\bigr)z^{i-j}w_1^{j-i-1} \sum_{m,k\in\mathbb{Z}}\overline{e}_{-(2\overline{\alpha}+ \overline{\beta})}[3k+i]z^{-3k-i-1}\left(\frac{z^3}{w_1^3} \right)^m\\
&=\bigl(\xi^{j-i}+\xi^{-j}\bigr)z^{i-j}w_1^{j-i-1}\delta\bigl(z^3/w_1^3\bigr) \overline{f}_{2\overline{\alpha}+\overline{\beta},i}(z),
\end{align*}
from which we deduce that
\begin{equation}\label{eq:3.16}
\overline{f}_{\overline{\alpha},\overline{i-j}}(w_1)\overline{f}_{\overline{\alpha}+\overline{\beta},j}(z)
=\frac{\bigl(\xi^{j-i}+\xi^{-j}\bigr)z^{i-j}w_1^{j-i+2} \overline{f}_{2\overline{\alpha}+\overline{\beta},i}(z)}{z^3-w_1^3}+\text{regular terms}.
\end{equation}
Consequently, we deduce from \eqref{eq:3.12} and \eqref{eq:3.16} that we have
\begin{align}
\overline{f}_{\overline{\alpha}}(w_1) \overline{f}_{\overline{\alpha}}(w_2)
\overline{f}_{\overline{\beta}}(z)
={}&\frac{w_1^2(2w_2+w_1)(w_2-w_1)\overline{f}_{ 2\overline{\alpha}+\overline{\beta},0}(z)}{z^2\bigl(z^2-w_1^2\bigr)\bigl(z^2-w_2^2\bigr)}\nonumber\\
&-\frac{\xi w_1(w_2+2w_1)(w_2-w_1)\overline{f}_{ 2\overline{\alpha}+\overline{\beta},1}(z)}{z\bigl(z^2-w_1^2\bigr)\bigl(z^2-w_2^2\bigr)}\nonumber\\
&-\frac{\xi^2(w_2-w_1)^2\overline{f}_{ 2\overline{\alpha}+\overline{\beta},2}(z)}{\bigl(z^2-w_1^2\bigr)\bigl(z^2-w_2^2\bigr)}+\text{regular terms},\label{eq:3.17}
\end{align}
and thus \smash{$(w_1-z)(w_2-z)\overline{f}_{\overline{\alpha}}(w_1) \overline{f}_{\overline{\alpha}}(w_2) \overline{f}_{\overline{\beta}}(z)$} consists solely of regular terms. As the terms $w_1^2(2w_2+w_1)(w_2-w_1)$, $w_1(w_2+2w_1)(w_2-w_1)$, $(w_2-w_1)^2$ and $(w_1-z)(w_2-z)$ all vanish when $w_1=w_2=z$, it follows from \eqref{eq:3.17} that the OPE relation \eqref{eq:3.13} holds when $\kappa=3$, $\overline{\alpha}\in\overline{\Delta}_s$ and $\overline{\beta}\in\overline{\Delta}_{\ell}$, and this completes the proof.
\end{proof}

Next, we will describe the dual space $\mathcal{U}$ of functions to $U=U\bigl(\mathfrak{n}_-\bigl[t^{\pm1}\bigr]^{\sigma}\bigr)$. By the Poincar\'e--Birkhoff--Witt theorem, we have
\begin{equation*}
U=\bigoplus_{\overline{\mathbf{m}}=(m^{(1)},\dots,m^{(r)})\in\mathbb{Z}_+^r}U[\overline{\mathbf{m}}],
\end{equation*}
where
\begin{equation*}
U[\overline{\mathbf{m}}]=\left\{\overline{f}_{i_1}[t_{i_1}^{\vee}n_1]\cdots\overline{f}_{i_m}[t_{i_m}^{\vee}n_m] \mid i_1,\dots,i_m\in I_r,\,n_1,\dots,n_m\in\mathbb{Z},\, \sum_{j=1}^m\overline{\alpha}_{i_j}=\sum_{i=1}^rm^{(i)} \overline{\alpha}_i\right\}.
\end{equation*}
We define the dual space $\mathcal{U}$ to $U$ by
\begin{equation*}
\mathcal{U}=\bigoplus_{\overline{\mathbf{m}}=(m^{(1)},\dots,m^{(r)})\in\mathbb{Z}_+^r}\mathcal{U}[\overline{\mathbf{m}}],
\end{equation*}
where $\mathcal{U}[\overline{\mathbf{m}}]$ is the dual space to $U[\overline{\mathbf{m}}]$ defined as follows: it is a space of functions in
\begin{equation*}
\mathbf{x}_{\overline{\mathbf{m}}}
=\bigl\{\bigl(x_i^{(b)}\bigr)^{t_b^{\vee}}\mid b\in I_r,\,i\in \bigl[1,m^{(b)}\bigr]\bigr\},
\end{equation*}
where \smash{$x_i^{(b)}$} is the variable corresponding to a generator of the form \smash{$\overline{f}_{\overline{\alpha}_b}\bigl(x_i^{(b)}\bigr)$}. We define a~pairing~${\mathcal{U}[\overline{\mathbf{m}}]\times U[\overline{\mathbf{n}}]\to\mathbb{C}}$ between $\mathcal{U}[\overline{\mathbf{m}}]$ and $U[\overline{\mathbf{n}}]$ for all $\overline{\mathbf{m}},\overline{\mathbf{n}}\in \mathbb{Z}^r$ as follows: the pairing is $0$ if~${\overline{\mathbf{m}}\neq\overline{\mathbf{n}}}$, and otherwise it is defined inductively by the relations
\begin{gather*}
\langle 1,1\rangle =1,\\
\big\langle g(\mathbf{x}),W\overline{f}_b[t_b^{\vee}n]\big\rangle
=\bigg\langle\oint\nolimits_{x_1^{(b)}=0}\bigl(x_1^{(b)}\bigr) ^{t_b^{\vee}(n+1)-1} g(\mathbf{x}_{\overline{\mathbf{m}}}) {\rm d}x_1^{(b)},W\bigg\rangle, \qquad W\in U[\overline{\mathbf{m}}_b'],
\end{gather*}
where
\begin{equation*}
\overline{\mathbf{m}}_b'=\bigl(m^{(1)},\dots,m^{(b-1)}, m^{(b)}-1,m^{(b+1)},\dots,m^{(r)}\bigr),
\end{equation*}
and we may take any counterclockwise oriented contour around the point \smash{$x_1^{(b)}=0$} that does not contain the points \smash{$x_1^{(b)}=x_i^{(a)}$} for all $(a,i)\neq(b,1)$ in our above definition of $\big\langle g(\mathbf{x}),W\overline{f}_b[t_b^{\vee}n]\big\rangle$. Similarly,
\begin{equation*}
\big\langle g(\mathbf{x}),\overline{f}_b[t_b^{\vee}n]W\big\rangle
=\bigg\langle\oint\nolimits_{x_1^{(b)}=\infty}\bigl(x_1^{(b)}\bigr) ^{t_b^{\vee}(n+1)-1} g(\mathbf{x}_{\overline{\mathbf{m}}}){\rm d}x_1^{(b)},W\bigg\rangle, \qquad W\in U[\overline{\mathbf{m}}_b'],
\end{equation*}
where we may take any clockwise oriented contour around the point $x_1^{(b)}=\infty$ that does not contain the points \smash{$x_1^{(b)}=x_i^{(a)}$} for all $(a,i)\neq(b,1)$ in our above definition of $\big\langle g(\mathbf{x}),\overline{f}_b[t_b^{\vee}n]W\big\rangle$.

With this pairing, the first OPE relation \eqref{eq:3.12} implies that the functions in the dual space $\mathcal{U}[\overline{\mathbf{m}}]$ may have a simple pole whenever \smash{$x_i^{(a)}=x_j^{(b)}$} for any $a,b\in I_r$ satisfying $\overline{C}_{ab}<0$. In addition, for any $c\in I_r$ satisfying $t_c^{\vee}=\kappa$, it follows from the definition of $\mathbf{x}_{\overline{\mathbf{m}}}$ that any function~${f\in\mathcal{U}[\overline{\mathbf{m}}]}$ is a function in \smash{$\bigl(x_{\ell}^{(c)}\bigr)^{\kappa}$}, which implies in particular that we have $\smash{f\bigl(x_{\ell}^{(c)}\bigr)}=\smash{f\bigl(\xi^mx_{\ell}^{(c)}\bigr)}$ for all $m\in[1,\kappa-1]$. These two observations together imply that in the case where $\overline{C}_{ab}<0$ and $\max\bigl(t_a^{\vee},t_b^{\vee}\bigr)= \kappa$, the functions in $\mathcal{U}[\overline{\mathbf{m}}]$ may have a simple pole whenever \smash{$x_i^{(a)}=\xi^m x_j^{(b)}$} for all $m\in[1,\kappa-1]$. Thus, the dual space $\mathcal{U}[\overline{\mathbf{m}}]$ is a space of rational functions $g(\mathbf{x}_{\overline{\mathbf{m}}})$ in $\mathbf{x}_{\overline{\mathbf{m}}}$ that is of the form
\begin{equation}\label{eq:3.18}
g(\mathbf{x}_{\overline{\mathbf{m}}})
=\frac{g_1(\mathbf{x}_{\overline{\mathbf{m}}})}{ \prod_{\substack{a<b\\\overline{C}_{ab}<0}} \prod_{i,j}\bigl(\bigl(x_i^{(a)}\bigr)^{k_{ab}}-\bigl(x_j^{(b)}\bigr)^{k_{ab}}\bigr)},
\end{equation}
where $k_{ab}=\max(t_a^{\vee},t_b^{\vee})$ for all $a,b\in I_r$ satisfying $\overline{C}_{ab}<0$. The function $g_1(\mathbf{x}_{\overline{\mathbf{m}}})$ is a Laurent polynomial in $\mathbf{x}_{\overline{\mathbf{m}}}$, and is symmetric in each subset \smash{$\bigl\{x_i^{(a)}\mid i\in \bigl[1,m^{(a)}\bigr]\bigr\}$} of $\mathbf{x}_{\overline{\mathbf{m}}}$ for all $a\in I_r$, as we have $\bigl[\overline{f}_a(w), \overline{f}_a(z)\bigr]=0$. Moreover, the second OPE relation \eqref{eq:3.13} implies that the function~$g_1(\mathbf{x}_{\overline{\mathbf{m}}})$ satisfies the following vanishing condition:
\begin{equation}\label{eq:3.19}
g_1(\mathbf{x}_{\overline{\mathbf{m}}})\big|_{x_i^{(a)}=x_j^{(a)}=x_k^{(b)}}=0
\end{equation}
for all $a,b\in I_r$ satisfying $\overline{C}_{ab}<0$, distinct $i,j\in \bigl[1,m^{(a)}\bigr]$, and $k\in\bigl[1,m^{(b)}\bigr]$. Moreover, in the case where $\overline{C}_{ab}<0$ and $\max(t_a^{\vee},t_b^{\vee})= \kappa$, the function $g_1(\mathbf{x}_{\overline{\mathbf{m}}})$ also satisfies the following vanishing condition:
\begin{equation}\label{eq:3.20}
g_1(\mathbf{x}_{\overline{\mathbf{m}}})\big|_{x_i^{(a)}=x_j^{(a)}=\xi^m x_k^{(b)}}=0
\end{equation}
for all distinct $i,j\in \bigl[1,m^{(a)}\bigr]$, $k\in[1,m^{(b)}]$ and $m\in[1,\kappa-1]$.

\subsection[Filtration of the dual space Um]{Filtration of the dual space $\boldsymbol{\mathcal{U}[\overline{\mathbf{m}}]}$}\label{Section 3.3}

Our next step is to introduce a filtration on $\mathcal{U}[\overline{\mathbf{m}}]$, following the approaches outlined in \cite[Section 3.4]{AKS06} and \cite[Section 4.1]{AK07}. To begin, we let \smash{$\bm{\mu}=\bigl(\mu^{(1)},\dots,\mu^{(r)}\bigr)$} be a multipartition of~${\overline{\mathbf{m}}=\bigl(m^{(1)},\dots, m^{(r)}\bigr)}$, that is, $\mu^{(a)}$ is a partition of $m^{(a)}$ for all $a\in I_r$. For each $a\in I_r$, $i\in\mathbb{N}$, we let \smash{$m_{a,i}\bigl(\mu^{(a)}\bigr)$} denote the number of parts of length $i$ in the partition \smash{$\mu^{(a)}$}. When the context is clear, we will write $m_{a,i}$ in lieu of $m_{a,i}\bigl(\mu^{(a)}\bigr)$. We note that there is a bijection between the set of multipartitions \smash{$\bm{\mu}=\bigl(\mu^{(1)},\dots,\mu^{(r)}\bigr)$} of \smash{$\overline{\mathbf{m}}=\bigl(m^{(1)},\dots, m^{(r)}\bigr)$}, and the set of vectors~${\mathbf{m}=(m_{a,i})_{a\in I_r,i\in\mathbb{N}}}$ of nonnegative integers that satisfy \smash{$\sum_{i\in\mathbb{N}} im_{a,i}=m^{(a)}$} for all $a\in I_r$, given by
\begin{equation}\label{eq:3.21}
\bm{\mu}\mapsto\bigl(m_{a,i}\bigl(\mu^{(a)}\bigr)\bigr)_{a\in I_r,i\in\mathbb{N}}
\end{equation}
for all multipartitions $\bm{\mu}$ of $\overline{\mathbf{m}}$.

Next, for any multipartition $\bm{\mu}$ of $\overline{\mathbf{m}}$, denoted $\bm{\mu}\vdash \overline{\mathbf{m}}$, we let $\mathcal{H}[\bm{\mu}]$ be a space of functions in the variables
\smash{$
\mathbf{y}_{\bm{\mu}}=\bigl\{\bigl(y_{i,r}^{(b)}\bigr)^{t_b^{\vee}}\mid b\in I_r,\,i\in\bigl[1,m^{(b)}\bigr],\,r\in[1,m_{b,i}]\bigr\}$}.
For any $a\in I_r$ and any partition~$\mu^{(a)}$ of $m^{(a)}$, we note that there is a bijective correspondence between the parts of $\mu^{(a)}$ and the pairs $(i,r)$, where $i\in \bigl[1,m^{(a)}\bigr]$ and $r\in[1,m_{a,i}]$. We define an ordering on the parts of $\mu^{(a)}$, or equivalently, the pairs $\bigl\{(i,r)\mid i\in \bigl[1,m^{(a)}\bigr],\, r\in[1,m_{a,i}]\bigr\}$, as follows:
$
(i,r)>(j,s)$ if ${i>j}$, or~${ i=j}$ and $r<s$.
Let us pick a collection $\mathbf{T}=(T_1,\dots,T_r)$ of tableaux, where $T_a$ is a tableau of shape~$\mu^{(a)}$ on the letters $1,\dots,m^{(a)}$ for all $a\in I_r$. For each~${n\in\bigl[1,m^{(a)}\bigr]}$, we let~${i_{T_a}(n)}$ be the length of the row in~$T_a$ in which $n$ appears, and $r_{T_a}(n)-1$ be the number of rows above the row in which $n$ appears of the same length $i_{T_a}(n)$. We define the evaluation map ${\varphi_{\bm{\mu},\mathbf{T}}\colon \mathcal{U}[\overline{\mathbf{m}}]\to \mathcal{H}[\bm{\mu}]}$ by~\smash{$
\varphi_{\bm{\mu},\mathbf{T}}\bigl(x_n^{(a)}\bigr)=y_{i_{T_a}(n),r_{T_a}(n)}^{(a)}$}, $ a\in I_r$, $n\in\bigl[1,m^{(a)}\bigr]$,
and extend the map $\varphi_{\bm{\mu},\mathbf{T}}$ to the whole of~$\mathcal{U}[\overline{\mathbf{m}}]$ by linearity. As the functions in $\mathcal{U}[\overline{ \mathbf{m}}]$ are symmetric in~\smash{$\bigl\{x_i^{(a)}\mid i\in \bigl[1,m^{(a)}\bigr]\bigr\}$} for each~${a\in I_r}$, it follows that any two collections $\mathbf{T}$, $\mathbf{T}'$ of tableaux give rise to the same map $\mathcal{U}[\overline{\mathbf{m}}] \to \mathcal{H}[\bm{\mu}]$. Thus, we may write $\varphi_{ \bm{\mu}}$ in lieu of $\varphi_{\bm{\mu},\mathbf{T}}$.

We define a lexicographical ordering on the set of multipartitions $\bm{\mu}\vdash\overline{\mathbf{m}}$ as follows: $\bm{\mu}>\bm{\nu}$ if there exists some index $a\in I_r$ satisfying $\mu^{(b)}=\nu^{(b)}$ for all $b\in[1,a-1]$ and $\mu^{(a)}>\nu^{(a)}$, where we take the lexicographical ordering on partitions here. We define
\begin{equation*}
\Gamma_{\bm{\mu}}=\bigcap_{\bm{\nu}>\bm{\mu}}\ker\varphi_{\bm{\nu}},\qquad\Gamma_{\bm{\mu}}'=\bigcap_{\bm{\nu}\geq\bm{\mu}}\ker\varphi_{\bm{\nu}}\subseteq\Gamma_{\bm{\mu}}
\end{equation*}
for all multipartitions $\bm{\mu}\vdash\overline{\mathbf{m}}$. By enumerating the set of multipartitions $\bm{\mu}\vdash \overline{\mathbf{m}}$ as $\bm{\mu}_1<\bm{\mu}_2<\cdots< \bm{\mu}_N$, where $\bm{\mu}_1=\bigl(\bigl(1^{m^{(1)}}\bigr),\bigl(1^{m^{(2)}}\bigr), \dots,\bigl(1^{m^{(r)}}\bigr)\bigr)$ and $\bm{\mu}_N=\bigl(\bigl(m^{(1)}\bigr),\bigl(m^{(2)}\bigr), \dots,\bigl(m^{(r)}\bigr)\bigr)$, and setting $\Gamma_{\bm{\mu}_0}:=\{0\}$, it follows that we have $\Gamma_{\bm{\mu}_N}=\mathcal{U}[\overline{\mathbf{m}}]$, and $\Gamma_{\bm{\mu}_{i-1}}= \Gamma_{\bm{\mu}_i}'$ for all $i\in [1,N]$. Thus, we have a filtration $\mathcal{F}_{\overline{\mathbf{m}}}$ on $\mathcal{U}[\overline{\mathbf{m}}]$ parameterized by all multipartitions $\bm{\mu}\vdash \overline{\mathbf{m}}$
\begin{equation*}
\{0\}=\Gamma_{\bm{\mu}_0}\subseteq\Gamma_{\bm{\mu}_1} \subseteq \Gamma_{\bm{\mu}_2}\subseteq\cdots\subseteq \Gamma_{\bm{\mu}_N}=\mathcal{U}[\overline{\mathbf{m}}].
\end{equation*}
Consequently, the associated graded space $\gr\mathcal{F}_{\overline{\mathbf{m}}}$ of the filtration $\mathcal{F}_{\overline{\mathbf{m}}}$ on $\mathcal{U}[\overline{\mathbf{m}}]$ is given by
\begin{equation}\label{eq:3.23}
\gr\mathcal{F}_{\overline{\mathbf{m}}}
=\bigoplus_{i=1}^N\Gamma_{\bm{\mu}_i}/\Gamma_{\bm{\mu}_{i-1}}
=\bigoplus_{i=1}^N\Gamma_{\bm{\mu}_i}/\Gamma_{\bm{\mu}_i}'
=\bigoplus_{\bm{\mu}\vdash\overline{\mathbf{m}}} \Gamma_{\bm{\mu}}/\Gamma_{\bm{\mu}}'.
\end{equation}
Our next step is to understand the graded structure of the functions in $\mathcal{U}[\overline{\mathbf{m}}]$ through the associated graded space $\gr\mathcal{F}_{\overline{ \mathbf{m}}}$ of $\mathcal{F}_{\overline{\mathbf{m}}}$ on $\mathcal{U}[\overline{\mathbf{m}}]$, by relating each graded piece $\Gamma_{\bm{\mu}}/\Gamma_{\bm{\mu}}'$ of the associated graded space $\gr\mathcal{F}_{\overline{\mathbf{m}}}$ to its corresponding image under the evaluation map $\varphi_{\bm{ \mu}}$ for each multipartition $\bm{\mu}\vdash\overline{ \mathbf{m}}$. To this end, we will define $\tilde{\mathcal{H}}[\bm{\mu}]$ to be the space of rational functions~$h(\mathbf{ y}_{\bm{\mu}})$ in $\mathbf{y}_{\bm{\mu}}$ that is of the form
\begin{equation}\label{eq:3.24}
h(\mathbf{y}_{\bm{\mu}})
=\frac{\prod_{a\in I_r}\prod_{(i,r)>(j,s)}\bigl(\bigl(y_{i,r}^{(a)}\bigr)^{t_a^{\vee}}-\bigl(y_{j,s}^{(a)}\bigr) ^{t_a^{\vee}})^{2\min(i,j)}}{\prod_{\substack{a<b \\\overline{C}_{ab}<0}} \prod_{i,j,r,s}\bigl(\bigl(y_{i,r}^{(a)}\bigr) ^{k_{ab}}-\bigl(y_{j,s}^{(b)}\bigr)^{k_{ab}}\bigr)^{\min(i,j)}} h_1(\mathbf{y}_{\bm{\mu}}),
\end{equation}
where $k_{ab}=\max(t_a^{\vee},t_b^{\vee})$ for all $a,b\in I_r$ satisfying $\overline{C}_{ab}<0$, and $h_1(\mathbf{y}_{\bm{\mu}})$ is an arbitrary Laurent~polynomial in $\mathbf{y}_{\bm{\mu}}$ that is symmetric in \smash{$\bigl\{y_{i,r}^{(a)}\mid r\in[1,m_{a,i}]\bigr\}$} for all pairs $(a,i)$.

We are now ready to relate $\Gamma_{\bm{\mu}}/ \Gamma_{\bm{\mu}}'$ to $\tilde{\mathcal{H}}[\bm{\mu}]$ for each multipartition $\bm{\mu}\vdash\overline{\mathbf{m}}$.

\begin{Theorem}\label{3.4}
Let $\bm{\mu}$ be a multipartition of $\overline{\mathbf{m}}$. Then the evaluation map $\varphi_{\bm{\mu}}\colon\mathcal{U}[\overline{\mathbf{m}}]\to \mathcal{H}[\bm{\mu}]$ induces an isomorphism $\overline{\varphi}_{\bm{\mu}}\colon\Gamma_{\bm{\mu}}/ \Gamma_{\bm{\mu}}'\to\tilde{\mathcal{H}}[\bm{\mu}]$ of graded vector spaces.
\end{Theorem}

The proof of Theorem~\ref{3.4} proceeds in a similar manner as in the proof of \cite[Theorem 3.6]{AKS06} and \cite[Theorem 4.1]{AK07}, which we will describe below the fold.

Our first step is to describe the zeros and poles of the functions in $\tilde{\mathcal{H}}[\bm{\mu}]$.

\begin{Lemma}\label{3.5}
Let $g(\mathbf{x}_{\overline{\mathbf{m}}})\in \Gamma_{\bm{\mu}}$, and $a\in I_r$. Then the function $h(\mathbf{y}_{\bm{\mu}})=\varphi_{\bm{\mu}}(g(\mathbf{x}_{\overline{\mathbf{m}}}))$ has a zero of order at least $2\min(i,j)$ whenever \smash{$y_{i,r}^{(a)}=y_{j,s}^{(a)}$}. Moreover, if $t_a^{\vee}=\kappa$, then the function $h(\mathbf{y}_{\bm{\mu}})$ has a zero of order at least $2\min(i,j)$ whenever \smash{$y_{i,r}^{(a)}=\xi^m y_{j,s}^{(a)}$} for all $m\in[1,\kappa-1]$.
\end{Lemma}

\begin{proof}
The proof of the first part of Lemma \ref{3.5} follows in the same fashion as in the proof of~\cite[Lemma 3.7]{AKS06}, and shall be omitted. For the second part of Lemma \ref{3.5}, we first note that as~${t_a^{\vee}=\kappa}$, we have \smash{$g\bigl(x_{\ell}^{(a)}\bigr)=g\bigl(\xi^mx_{\ell}^{(a)}\bigr)$} for all $m\in [1,\kappa-1]$. As the map $\varphi_{\bm{\mu}}$ is degree preserving, it follows that we have \smash{$h\bigl(y_{\ell,u}^{(a)}\bigr)= h\bigl(\xi^m y_{\ell,u}^{(a)}\bigr)$} for all $m\in [1,\kappa-1]$. Together with the first part of Lemma \ref{3.5}, this proves the second part of Lemma \ref{3.5}.
\end{proof}

\begin{Lemma}\label{3.6}
Let $g(\mathbf{x}_{\overline{\mathbf{m}}})\in\mathcal{U}[\overline{\mathbf{m}}]$, $a,b\in I_r$ be root indices that satisfy $\overline{C}_{ab}<0$, and $(i,r)$ and~${(j,s)}$ be parts of $\mu^{(a)}$ and $\mu^{(b)}$, respectively. Then the function $h(\mathbf{y}_{\bm{\mu}})=\varphi_{\bm{\mu}}(g(\mathbf{x}_{\overline{\mathbf{m}}}))$ has a pole of order~at most $\min(i,j)$ whenever \smash{$y_{i,r}^{(a)}=y_{j,s}^{(b)}$}. Moreover, if $\max(t_a^{\vee},t_b^{\vee})=\kappa$, then the function~$h(\mathbf{y}_{\bm{\mu}})$ has a pole of order at most $\min(i,j)$ whenever \smash{$y_{i,r}^{(a)}=\xi^m y_{j,s}^{(b)}$} for all $m\in[1,\kappa-1]$.
\end{Lemma}

\begin{proof}
Let us write $g(\mathbf{x}_{\overline{\mathbf{m}}})$ in the form given in \eqref{eq:3.18}. As the vanishing condition \eqref{eq:3.19} satisfied by the function $g_1(\mathbf{x}_{\overline{\mathbf{m}}})$ is identical to the vanishing condition arising from the Serre relations in the simply-laced untwisted case (see \cite[equation~(3.15)]{AKS06} and \cite[equation~(4.4)]{AK07}), the first part of Lemma \ref{3.6} follows in a similar fashion as in the proofs of \cite[Lemma 3.8]{AKS06} and \cite[Lemma A.2]{AK07}, and shall be omitted. We will omit the proof of the second part of Lemma \ref{3.6} as well, as it follows from a similar argument as in the proof of the second part of Lemma \ref{3.5}.
\end{proof}

As a corollary of Lemmas \ref{3.5} and \ref{3.6}, this shows that the evaluation map $\varphi_{\bm{\mu}}\colon\Gamma_{\bm{\mu}} \to\tilde{\mathcal{H}}[\bm{\mu}]$ is well-defined.

Our next step is to show that the evaluation map $\varphi_{\bm{\mu}}\colon\Gamma_{\bm{\mu}} \to\tilde{\mathcal{H}}[\bm{\mu}]$ is surjective. Similar as before in the proofs of \cite[Lemma 3.16]{AKS06} and \cite[Theorem A.4]{AK07}, we will produce an explicit function $g(\mathbf{x}_{\overline{ \mathbf{m}}})\in\mathcal{U}[\overline{\mathbf{m}}]$ for each function $h(\mathbf{y}_{\bm{\mu}})\in \tilde{\mathcal{H}}[\bm{\mu}]$, and show that up to a nonzero scalar, the image of $g(\mathbf{x}_{\overline{\mathbf{m}}})$ under the evaluation map $\varphi_{\bm{\mu}}$ is precisely $h(\mathbf{y}_{\bm{\mu}})$.

Let us enumerate the variables in the pre-image of \smash{$y_{i,r}^{(a)}$} under the evaluation map $\varphi_{\bm{\mu}}= \varphi_{\bm{\mu},\mathbf{T}}$ as \smash{$\bigl\{x_{i,r}^{(a)}[1],\dots,x_{i,r}^{(a)}[i]\bigr\}\subseteq\bigl\{x_1^{(a)},\dots, x_{m^{(a)}}^{(a)}\bigr\}$}. We first let
\begin{gather}\label{eq:3.25}
g_0(\mathbf{x}_{\overline{\mathbf{m}}})
=\frac{\prod_{a\in I_r}\prod_{(i,r)>(i',r')} \prod_{j=1}^{i'}\bigl(x_{i,r}^{(a)}[j]^{t_a^{\vee}}-x_{i',r'}^{(a)}[j]^{t_a^{\vee}}\bigr)\!\bigl(x_{i,r}^{(a)}[j+1]^{t_a^{\vee}}-x_{i',r'}^{(a)}[j]^{t_a^{\vee}}\bigr)}{ \prod_{\substack{a<b\\\overline{C}_{ab}<0}}\prod_{i,j,r,s} \prod_{\ell=1}^{\min(i,j)}\bigl(x_{i,r}^{(a)}[\ell]^{k_{ab}}-x_{j,s}^{(b)}[\ell]^{k_{ab}}\bigr)},\!\!\!
\end{gather}
where $k_{ab}=\max(t_a^{\vee},t_b^{\vee})$ for all $a,b\in I_r$ satisfying $\overline{C}_{ab}<0$, and \smash{$x_{i,r}^{(a)}[i+1]:=x_{i,r}^{(a)}[1]$}. Next, for any Laurent polynomial $h_1(\mathbf{y}_{\bm{\mu}})$ in $\mathbf{y}_{\bm{\mu}}$ that is symmetric in \smash{$\bigl\{y_{i,r}^{(a)}\mid r\in[1,m_{a,i}]\bigr\}$} for all pairs~$(a,i)$, and any term
\[
c\cdot \bigl(y_{i_1,r_1}^{(a_1)}\bigr)^{m_1t_{a_1}^{\vee}}\cdots\bigl(y_{i_k,r_k}^{(a_k)}\bigr)^{m_kt_{a_k}^{\vee}}
\] of the Laurent polynomial $h_1(\mathbf{y}_{\bm{\mu}})$, we let
\[
g_1(\mathbf{x}_{\overline{\mathbf{m}}})
=\smash{\bigl(x_{i_1,r_1}^{(a_1)}[1]\bigr)^{m_1t_{a_1}^{\vee}}\cdots \bigl(x_{i_k,r_k}^{(a_k)}[1]\bigr)^{m_kt_{a_k}^{\vee}}},
\]
and we define
\begin{equation}\label{eq:3.27}
g(\mathbf{x}_{\overline{\mathbf{m}}})
=\Sym(g_0(\mathbf{x}_{\overline{\mathbf{m}}})g_1(\mathbf{x}_{\overline{\mathbf{m}}})),
\end{equation}
where the symmetrization is over each of the $r$ sets $\bigl\{ x_1^{(a)},\dots,x_{m^{(a)}}^{(a)}\bigr\}$ of
variables with the same label $a$.

\begin{Lemma}\label{3.7}
The function $g(\mathbf{x}_{\overline{\mathbf{m}}})$ defined by equation \eqref{eq:3.27} is an element of $\Gamma_{\bm{\mu}}$, and the image~$\varphi_{\bm{\mu}}(g(\mathbf{x}_{\overline{ \mathbf{m}}}))$ of the function $g(\mathbf{x}_{\overline{ \mathbf{m}}})$ under the evaluation map $\varphi_{\bm{\mu}}\colon \Gamma_{\bm{\mu}}\to\mathcal{H}[\bm{\mu}]$ is a nonzero scalar multiple of the function $h(\mathbf{y}_{\bm{\mu}})$ defined by \eqref{eq:3.24}.
\end{Lemma}

The proof of Lemma \ref{3.7} shall be omitted, as it follows from the proofs of \cite[Lemmas 3.11, 3.13, and 3.15]{AKS06}, with the appropriate modifications made in the definition of $g_0(\mathbf{x}_{\overline{\mathbf{m}}})$ in \eqref{eq:3.25} to account for the differences in the structure of zeros and poles between the twisted case and the untwisted case described in \cite{AK07,AKS06}.

By Lemma \ref{3.7}, this shows that the evaluation map $\varphi_{\bm{\mu}}\colon \Gamma_{\bm{\mu}}\to\mathcal{H}[\bm{\mu}]$ is surjective. As the kernel of the evaluation map $\varphi_{\bm{\mu}}\colon \Gamma_{\bm{\mu}}\to\mathcal{H}[\bm{\mu}]$ is given by $\Gamma_{\bm{\mu}}\cap\ker \varphi_{\bm{\mu}}=\Gamma_{\bm{\mu}}'$, this completes the proof of Theorem~\ref{3.4}. Thus, as a corollary of Theorem~\ref{3.4} and \eqref{eq:3.23}, we see that the associated graded space $\gr\mathcal{F}_{\overline{\mathbf{m}}}$ of the filtration $\mathcal{F}_{\overline{\mathbf{m}}}$ on $\mathcal{U}[\overline{\mathbf{m}}]$ is given by
\smash{$
\gr\mathcal{F}_{\overline{\mathbf{m}}}
=\bigoplus_{\bm{\mu}\vdash\overline{\mathbf{m}}}\mathcal{H}[\bm{\mu}]$}.

\subsection[The dual space of generating functions of matrix elements to the fusion product of twisted KR-modules]{The dual space of generating functions of matrix elements\\ to the fusion product of twisted KR-modules}\label{Section 3.4}

Having provided an explicit description of the dual space $\mathcal{U}$ of functions by describing each associated graded space $\gr\mathcal{F}_{\overline{\mathbf{m}}}$ of the filtration $\mathcal{F}_{\overline{\mathbf{m}}}$ on $\mathcal{U}[\overline{\mathbf{m}}]$ in terms of spaces $\mathcal{H}[\bm{\mu}]$ of functions associated to multipartitions $\bm{\mu}\vdash\overline{\mathbf{m}}$ for each $r$-tuple $\overline{\mathbf{m}}$ of nonnegative integers, our goal in this subsection is to give a similar explicit characterization of the subspace $\overline{\mathcal{C}}_{ \overline{\lambda},\mathbf{n}}$ of $\mathcal{U}$, consisting of generating functions of matrix elements that arises from fusion products of localized KR-modules over $\mathfrak{g}[t]^{\sigma}$, by extending the procedure described in Section~\ref{Section 3.3} to $\overline{\mathcal{C}}_{ \overline{\lambda},\mathbf{n}}$, following the approaches outlined in \cite[Section~5.5]{AKS06} and \cite[Section~5.2]{AK07}.

Let us keep the notations as in Sections \ref{Section 3.1.1} and \ref{Section 3.1.2}, and write $V_j(z_j)=\KR_{a_j,i_j} ^{\sigma}(z_j)$ for all $j\in [1,N]$. Here, we view $V_1(z_1), \dots,V_N(z_N)$ as $\mathfrak{g}[t]^{\sigma}$-modules, even though they arise via the restriction the action of localized KR-modules over $\mathfrak{g}[t]$ to $\mathfrak{g}[t]^{\sigma}$, as we are interested in computing the graded multiplicities of irreducible $\mathfrak{g}^{\sigma}$-modules in fusion products of localized twisted KR-modules.

Following \eqref{eq:3.8}, we define $\overline{\mathcal{C}}_{\overline{\lambda},\mathbf{n}}[\overline{\mathbf{m}}]$ to be the subspace of $\mathcal{U}[\overline{\mathbf{m}}]$, consisting of generating functions of matrix elements of the form
\begin{equation}
\big\langle v_{\overline{\lambda}^*}\big|\overline{f}_{b_1}\bigl(x_1^{(b_1)}\bigr)\cdots\overline{f}_{b_M}\bigl(x_{m^{(b_M)}}^{(b_M)}\bigr) \big|v_1\otimes\cdots\otimes v_N\big\rangle\label{eq:3.29}
\end{equation}
for each $\overline{\mathbf{m}}=\bigl(m^{(1)},\dots,m^{(r)}\bigr)\in \mathbb{Z}_+^r$, where $M=\sum_{b\in I_r}m^{(b)}$, and $b_1,\dots,b_M\in I_r$ are root indices that satisfy \smash{$\sum_{j=1}^M\overline{\alpha}_{b_j}=\sum_{b\in I_r}^rm^{(b)} \overline{\alpha}_b$}.

As $\overline{\mathcal{C}}_{\overline{\lambda},\mathbf{n}}[\overline{\mathbf{m}}]$ is a subspace of $\mathcal{U}[\overline{\mathbf{m}}]$, it follows that all functions $g(\mathbf{x}_{\overline{\mathbf{m}}})$ in $\overline{\mathcal{C}}_{\overline{\lambda},\mathbf{n}}[\overline{\mathbf{m}}]$ are of the form given in \eqref{eq:3.18}, and its numerator $g_1(\mathbf{x}_{\overline{ \mathbf{m}}})$ satisfies the vanishing conditions given in \eqref{eq:3.19} and \eqref{eq:3.20}. In addition, the functions $g(\mathbf{x}_{\overline{\mathbf{m}}})$ (and hence their corresponding numerators $g_1(\mathbf{x}_{\overline{ \mathbf{m}}})$) satisfy additional properties that arise mainly from the properties of the lowest weight vector~$v_{\overline{\lambda}^*}$ of the $\widehat{\mathfrak{g}}^{ \sigma}$-module \smash{$\widehat{V}_{\kappa k,\overline{\lambda}}^*$} and the highest weight vector $v_1\otimes\cdots\otimes v_N$ of $V_1(z_1)*\cdots*V_N(z_N)=\KR_{a_1,i_1}^{\sigma}(z_1)*\cdots* \KR_{a_N,i_N}^{\sigma}(z_N)$. These additional properties are given as follows:
\begin{enumerate}\itemsep=0pt
\item[$(1)$] \emph{Zero weight condition:} For the matrix element given by \eqref{eq:3.29} to be nonzero, the total weight of this matrix element with respect to $\mathfrak{h}^{\sigma} \subseteq\mathfrak{g}^{\sigma}$ must be equal to zero, that is, we have (recall that the $\mathfrak{g}^{\sigma}$-weight of $v_{\overline{\lambda}^*}$ is $-\overline{\lambda}$)
\begin{equation}\label{eq:3.30}
-\overline{\lambda}-\sum_{j=1}^M\overline{\alpha}_{b_j}+ \sum_{j=1}^Ni_j\overline{\omega}_{a_j}=0.
\end{equation}
By letting $\overline{\lambda}=\sum_{a\in I_r}\ell_a \overline{\omega}_a$, we may rewrite the left-hand side of the zero weight condition~\eqref{eq:3.30} as
\begin{align}
-\overline{\lambda}-\sum_{j=1}^M\overline{\alpha}_{b_j}+ \sum_{j=1}^Ni_j\overline{\omega}_{a_j}
&=-\sum_{a\in I_r}\ell_a\overline{\omega}_a-\sum_{b\in I_r}m^{(b)} \overline{\alpha}_b+\sum_{a\in I_r}\sum_{i\in\mathbb{N}} in_{a,i}\overline{\omega}_a\nonumber\\
&=-\sum_{a\in I_r}\ell_a\overline{\omega}_a-\sum_{a,b\in I_r}\overline{C}_{ab} m^{(b)}\overline{\omega}_a+\sum_{a\in I_r}\sum_{i\in \mathbb{N}}in_{a,i}\overline{\omega}_a.\label{eq:3.31}
\end{align}
Thus, by comparing the coefficients of $\overline{\omega}_a$ on the right-hand side of \eqref{eq:3.30} and \eqref{eq:3.31} for all $a\in I_r$, it follows that the zero weight condition is equivalent to
\begin{equation}\label{eq:3.32}
\ell_a+\sum_{b\in I_r}\overline{C}_{ab} m^{(b)}-\sum_{i\in \mathbb{N}}in_{a,i}=0
\end{equation}
for all $a\in I_r$. In particular, as the Cartan matrix $\overline{C}$ is invertible, it follows if there exists a~solution $\overline{\mathbf{m}}\in\mathbb{Z}_+^r$ to \eqref{eq:3.32}, then it is uniquely determined by the choices of $\overline{\lambda}$ and $\mathbf{n}$ (note that for fixed $\overline{\lambda}$ and $\mathbf{n}$, there always exists a unique solution $\overline{\mathbf{m}}\in\mathbb{Q}^r$ to \eqref{eq:3.32}, as \eqref{eq:3.32} is an equation with integer coefficients). Subsequently, in this subsection and beyond, we may restrict our attention to pairs $\bigl(\overline{\lambda},\mathbf{n}\bigr)$ for which there exists a unique solution $\overline{\mathbf{m}}\in\mathbb{Z}_+^r$ to \eqref{eq:3.32}, or equivalently, a unique $\overline{ \mathbf{m}}\in\mathbb{Z}_+^r$ for which $\overline{ \mathcal{C}}_{\overline{\lambda},\mathbf{n}}[\overline{ \mathbf{m}}]\neq\{0\}$.

\item[$(2)$] \emph{Lowest weight condition:} As \smash{$v_{\overline{ \lambda}^*}$} is a lowest weight vector of the $\widehat{ \mathfrak{g}}^{\sigma}$-module \smash{$\widehat{V}_{\kappa k, \overline{\lambda}}^*$}, it follows that we have $\widehat{ \mathfrak{n}}_-^{\sigma}\cdot v_{\overline{ \lambda}^*}=0$. In particular, we have $\overline{f}_a [t_a^{\vee}n]\cdot v_{ \overline{\lambda}^*}=0$ for all $a\in I_r$ and~${n\in\mathbb{ Z}_{\leq0}}$. Thus, we have
\begin{align}
\overline{f}_a\bigl(x_i^{(a)}\bigr)\cdot v_{\overline{\lambda}^*}
&=\sum_{n\in\mathbb{Z}}\overline{f}_a[t_a^{\vee}n]\bigl(x_i^{(a)}\bigr)^{-t_a^{\vee}(n+1)}\cdot v_{\overline{\lambda}^*}\nonumber\\
&=\sum_{n\in\mathbb{N}}\overline{f}_a[t_a^{\vee}n]\bigl(x_i^{(a)}\bigr)^{-t_a^{\vee}(n+1)}\cdot v_{\overline{\lambda}^*},\label{eq:3.33}
\end{align}
which in turn implies that we must have
\smash{$
\deg_{x_i^{(a)}}g(\mathbf{x}_{\overline{\mathbf{m}}})\leq -2t_a^{\vee}
$}
for all $g(\mathbf{x}_{\overline{\mathbf{m}}})\in\overline{ \mathcal{C}}_{\overline{\lambda},\mathbf{n}}[\overline{ \mathbf{m}}]$ and $a\in I_r$. Equivalently, by letting
\begin{equation}\label{eq:3.35}
\overline{f}_a^+(z)=
\sum_{n\in\mathbb{N}}\overline{f}_a[t_a^{\vee}n]z^{-t_a^{\vee}(n+1)}
\end{equation}
for all $a\in I_r$, we see from \eqref{eq:3.29} and \eqref{eq:3.33} that the space $\overline{\mathcal{C}}_{ \overline{\lambda},\mathbf{n}}[\overline{\mathbf{m}}]$ consists of generating functions of matrix elements of the form
\begin{equation}\label{eq:3.36}
\big\langle v_{\overline{\lambda}^*}\big|\overline{f}_{b_1}^+\bigl(x_1^{(b_1)}\bigr)\cdots\overline{f}_{b_M}^+\bigl(x_{m^{(b_M)}}^{(b_M)}\bigr) \big|v_1\otimes\cdots\otimes v_N\big\rangle
\end{equation}
for each $\overline{\mathbf{m}}=\bigl(m^{(1)},\dots,m^{(r)}\bigr)\in \mathbb{Z}_+^r$, where $M=\sum_{b\in I_r}m^{(b)}$, and $b_1,\dots,b_M\in I_r$ are root indices that satisfy \smash{$\sum_{j=1}^M\overline{\alpha}_{b_j}=\sum_{b\in I_r}^rm^{(b)} \overline{\alpha}_b$}.
\item[$(3)$] \emph{Highest weight conditions:} By \eqref{eq:2.6}, it follows that for each $p\in[1,N]$, we have
\begin{equation*}
\psi_{z_p}\bigl(\overline{f}_a[t_a^{\vee}n]\bigr)v_p =\delta_{a_p,a}z_p^{t_{a_p}^{\vee}n}\psi_0\bigl(\overline{f}_{a_p}[0]\bigr)v_p.
\end{equation*}
In light of \eqref{eq:3.35} and \eqref{eq:3.36}, it follows that if $a\neq a_p$, then we have $\psi_{z_p}\bigl(\overline{ f}_a^+\bigl(x_i^{(a)}\bigr)\bigr)v_p=0$. Else, we have
\begin{align}
\psi_{z_p}\bigl(\overline{f}_{a_p}^+\bigl(x_i^{(a_p)}\bigr)\bigr)v_p
&=\sum_{n\in\mathbb{N}}\psi_{z_p}\bigl(\overline{f}_{a_p}[t_{a_p}^{\vee}n]\bigl(x_i^{(a_p)}\bigr)^{-t_{a_p}^{\vee}(n+1)}\bigr)v_p\nonumber\\
&=\sum_{n\in\mathbb{N}}z_p^{t_{a_p}^{\vee}n}\bigl(x_i^{(a_p)}\bigr)^{-t_{a_p}^{\vee}(n+1)}\psi_0\bigl(\overline{f}_{a_p}[0]\bigr)v_p\nonumber\\
&=\frac{z_p^{t_{a_p}^{\vee}}}{\bigl(x_i^{(a_p)}\bigr)^{t_{a_p}^{\vee}}\bigl(\bigl(x_i^{(a_p)}\bigr)^{t_{a_p}^{\vee}}-z_p^{t_{a_p}^{\vee}}\bigr)}\psi_0\bigl( \overline{f}_{a_p}[0]\bigr)v_p.\label{eq:3.37}
\end{align}
Using the pairing $\mathcal{U}[\overline{\mathbf{m}}]\times U[\overline{\mathbf{m}}]\to\mathbb{C}$, this shows that for each $g(\mathbf{x}_{\overline{\mathbf{m}}})\in\overline{ \mathcal{C}}_{\overline{\lambda},\mathbf{n}}[\overline{ \mathbf{m}}]$ and $a\in I_r\setminus\{a_p\}$, the function $g(\mathbf{x}_{\overline{\mathbf{m}}})$ does not have a pole at \smash{$x_i^{(a)}=z_p$}, and in the case where~${t_a^{\vee}= \kappa}$, the function $g(\mathbf{x}_{\overline{\mathbf{m}}})$ does not have a pole at \smash{$x_i^{(a)}=\xi^m z_p$} for all $m\in[1,\kappa-1]$ as well. The function $g(\mathbf{x}_{ \overline{\mathbf{m}}})$ may have a simple pole at \smash{$x_j^{(a_p)}=z_p$}, and in the case where $t_{a_p}^{\vee}= \kappa$, the~function $g(\mathbf{x}_{\overline{\mathbf{ m}}})$ may have a simple pole at \smash{$x_j^{(a_p)}=\xi^m z_p$} for all $m\in[1,\kappa-1]$ as well. Thus, if we define $g(\mathbf{x}_{\overline{\mathbf{m}}})$ by \eqref{eq:3.18}, then we have
\begin{equation}\label{eq:3.38}
g_1(\mathbf{x}_{\overline{\mathbf{m}}})
=\frac{g_2(\mathbf{x}_{\overline{\mathbf{m}}})}{  \prod_{p=1}^N\prod_{i=1}^{m^{(a_p)}}\bigl(\bigl(x_i^{(a_p)}\bigr)^{t_{a_p} ^{\vee}}-z_p^{t_{a_p}^{\vee}}\bigr)},
\end{equation}
where $g_2(\mathbf{x}_{\overline{\mathbf{m}}})$ is a polynomial in $\mathbf{x}_{\overline{\mathbf{m}}}$ that satisfies the vanishing conditions~\eqref{eq:3.19} and~\eqref{eq:3.20}. Lastly, $g(\mathbf{x}_{\overline{ \mathbf{m}}})$ does not have a pole at \smash{$x_i^{(a)}=0$} for all $a\in I_r$, as there are no $\mathfrak{g}[t]^{\sigma}$-modules localized at $0$.
\item[$(4)$] \emph{Integrability condition:} For each $p\in[1,N]$, it follows from \eqref{eq:2.5c} that we have
\begin{equation}\label{eq:3.39}
\psi_{z_p}\bigl(\overline{f}_{a_p}[0]\bigr)^{k+1}v_p
=0
=\psi_0\bigl(\overline{f}_{a_p}[0]\bigr)^{k+1}v_p.
\end{equation}
On the other hand, it follows from \eqref{eq:3.37} that we have
\begin{gather*}
\psi_{z_p}\bigl(\overline{f}_{a_p}^+\bigl(x_{j_1}^{(a_p)}\bigr)\bigr) \cdots\psi_{z_p}\bigl(\overline{f}_{a_p}^+\bigl(x_{j_{i_p}} ^{(a_p)}\bigr)\bigr) v_p\\
\qquad
=\frac{z_p^{t_{a_p}^{\vee}i_p}}{\prod_{\ell=1} ^{i_p}(x_{j_{\ell}}^{(a_p)})^{t_{a_p}^{\vee}}\bigl(\bigl(x_{j_{\ell}}^{(a_p)}\bigr) ^{t_{a_p}^{\vee}}-z_p^{t_{a_p}^{\vee}}\bigr)}\psi_0\bigl(\overline{f} _{a_p}[0]\bigr)^{i_p} v_p
\end{gather*}
for all \smash{$j_1,\dots,j_{i_p}\in\bigl[1, m^{(a_p)}\bigr]$}, which implies that
\begin{gather*}
\prod_{\ell=1}^{i_p}\bigl(\bigl(x_{j_{\ell}}^{(a_p)}\bigr)^{t_{a_p}^{\vee}}-z_p^{t_{a_p}^{\vee}}\bigr)\psi_{z_p}\bigl(\overline{f}_{a_p}^+(x_{ j_1}^{(a_p)})\bigr)\cdots\psi_{z_p}\bigl(\overline{f}_{a_p}^+(x_{j_{i_p}}^{(a_p)})\bigr)v_p
\end{gather*}
consists solely of regular terms. Together with \eqref{eq:3.39}, this shows that we have
\begin{equation*}
\prod_{\ell=1}^{i_p+1}\bigl(\bigl(x_{j_{\ell}}^{(a_p)}\bigr)^{t_{a_p}^{\vee}}-z_p^{t_{a_p}^{\vee}}\bigr)\psi_{z_p}\bigl(\overline{f}_{a_p}^+\bigl(x_{ j_1}^{(a_p)}\bigr)\bigr)\cdots\psi_{z_p}\bigl(\overline{f}_{a_p}^+\bigl(x_{j_{i_p+1}}^{(a_p)}\bigr)\bigr)v_p\big|_{x_{j_1}^{(a_p)}=\cdots=x_{j_{i_p+1}}^{(a_p)}=z_p}=0
\end{equation*}
for all distinct $j_1,\dots,j_{i_p+1}\in\bigl[1,m^{(a_p)}\bigr]$. Using the pairing $\mathcal{U}[\overline{\mathbf{m}}]\times U[\overline{ \mathbf{m}}]\to\mathbb{C}$, this shows that if we define $g(\mathbf{x}_{\overline{\mathbf{m}}})\in\overline{ \mathcal{C}}_{\overline{\lambda},\mathbf{n}}[\overline{ \mathbf{m}}]$ by \eqref{eq:3.18} and \eqref{eq:3.38}, then in addition to the vanishing conditions \eqref{eq:3.19} and \eqref{eq:3.20} satisfied by $g_2(\mathbf{x}_{ \overline{\mathbf{m}}})$, the function $g_2(\mathbf{x}_{\overline{\mathbf{m}}})$ satisfies the following vanishing condition:
\begin{equation*}
g_2(\mathbf{x}_{\overline{\mathbf{m}}})|_{x_{j_1}^{(a_p)}=\cdots=x_{j_{i_p+1}}^{(a_p)}=z_p}=0
\end{equation*}
for all distinct $j_1,\dots,j_{i_p+1}\in\bigl[1,m^{(a_p)}\bigr]$, and in the case where $t_{a_p}^{\vee}=\kappa$, the function $g_2(\mathbf{x}_{\overline{\mathbf{m}}})$ also satisfies the following vanishing condition:
\begin{equation*}
g_2(\mathbf{x}_{\overline{\mathbf{m}}})|_{x_{j_1}^{(a_p)}=\cdots=x_{j_{i_p+1}}^{(a_p)}=\xi^m z_p}=0
\end{equation*}
for all distinct $j_1,\dots,j_{i_p+1}\in\bigl[1,m^{(a_p)}\bigr]$ and $m\in[1,\kappa-1]$.
\end{enumerate}

\subsection[Filtration of the space of generating functions of matrix elements]{Filtration of the space of generating functions of matrix elements $\boldsymbol{\overline{\mathcal{C}}_{ \overline{\lambda},\mathbf{n}}[\overline{\mathbf{m}}]}$}\label{Section 3.5}

Following the approaches outlined in \cite[Section~5.6]{AKS06} and \cite[Section~5.3]{AK07}, our next step is to describe the induced filtration $\tilde{\mathcal{F}}_{\overline{\mathbf{m}}}$ on $\overline{\mathcal{C}}_{\overline{\lambda},\mathbf{n}}[\overline{\mathbf{m}}]$ from the filtration $\mathcal{F}_{\overline{\mathbf{m}}}$ on $\mathcal{U}[\overline{\mathbf{m}}]$ via restriction as a subspace of $\mathcal{U}[\overline{\mathbf{m}}]$, which can be described explicitly as follows: we enumerate the set of multipartitions $\bm{\mu}\vdash\overline{\mathbf{m}}$ as in Section~\ref{Section 3.3}, and we define
\begin{equation*}
\tilde{\Gamma}_{\bm{\mu}}=\Gamma_{\bm{\mu}}\cap\overline{\mathcal{C}}_{\overline{\lambda},\mathbf{n}}[\overline{\mathbf{m}}], \qquad
\tilde{\Gamma}_{\bm{\mu}}'=\Gamma_{\bm{\mu}}'\cap\overline{\mathcal{C}}_{\overline{\lambda},\mathbf{n}}[\overline{\mathbf{m}}]
\end{equation*}
for all multipartitions $\bm{\mu}\vdash\overline{\mathbf{m}}$. As we have \smash{$\tilde{\Gamma}_{\bm{\mu}_{i-1}}=\tilde{\Gamma} _{\bm{\mu}_i}'$} for all $i\in [1,N]$, we have a filtration~$\tilde{\mathcal{F}}_{\overline{\mathbf{m}}}$ on $\overline{ \mathcal{C}}_{\overline{\lambda},\mathbf{n}}[\overline{ \mathbf{m}}]$ parameterized by all multipartitions $\bm{\mu}\vdash\overline{\mathbf{m}}$
\begin{equation*}
\{0\}=\tilde{\Gamma}_{\bm{\mu}_0}\subseteq\tilde{\Gamma}_{\bm{ \mu}_1}\subseteq\cdots\subseteq\tilde{\Gamma}_{\bm{\mu}_N}= \overline{\mathcal{C}}_{\overline{\lambda},\mathbf{n}}[\overline{\mathbf{m}}].
\end{equation*}
Thus, the associated graded space $\gr\tilde{\mathcal{F}}_{\overline{\mathbf{m}}}$ of the filtration $\tilde{\mathcal{F}}_{\overline{\mathbf{m}}}$ on $\overline{\mathcal{C}}_{\overline{\lambda},\mathbf{n}}[\overline{\mathbf{m}}]$ is given by
\begin{equation}\label{eq:3.44}
\gr\tilde{\mathcal{F}}_{\overline{\mathbf{m}}}
=\bigoplus_{i=1}^N\tilde{\Gamma}_{\bm{\mu}_i}/\tilde{\Gamma}_{\bm{\mu}_{i-1}}
=\bigoplus_{i=1}^N\tilde{\Gamma}_{\bm{\mu}_i}/\tilde{\Gamma}_{\bm{\mu}_i}'
=\bigoplus_{\bm{\mu}\vdash\overline{\mathbf{m}}} \tilde{\Gamma}_{\bm{\mu}}/\tilde{\Gamma}_{\bm{\mu}}'.
\end{equation}
Similar as before in Section \ref{Section 3.3}, we would like to describe the image $\varphi_{\bm{\mu}}(\tilde{\Gamma}_{ \bm{\mu}})$ of the functions in~$\tilde{\Gamma}_{\bm{\mu}}$ under the evaluation map $\varphi_{\bm{\mu}}$ for each multipartition $\bm{\mu}\vdash\overline{\mathbf{m}}$. To this end, we will need the following description of the poles of the functions in $\varphi_{\bm{\mu}}\bigl(\overline{\mathcal{C}}_{ \overline{\lambda},\mathbf{n}}[\overline{\mathbf{m}}]\bigr)$ that arises from the highest weight and integrability conditions described in Section \ref{Section 3.4}.

\begin{Lemma}\label{3.8}
Let $g(\mathbf{x}_{\overline{\mathbf{m}}})\in \overline{ \mathcal{C}}_{\overline{\lambda},\mathbf{n}}[\overline{ \mathbf{m}}]$, and $p\in[1,N]$. Then the function $h(\mathbf{y}_{\bm{\mu}})=\varphi_{\bm{\mu}}(g(\mathbf{x}_{ \overline{ \mathbf{m}}}))$ has a pole of order at most $\min(i,i_p)$ whenever \smash{$y_{i,r}^{(a_p)}=z_p$}. Moreover, if $t_{a_p}^{\vee}=\kappa$, then the function~$h(\mathbf{y}_{\bm{ \mu}})$ has a pole of order at most $\min(i,i_p)$ whenever \smash{$y_{i,r}^{(a_p)}=\xi^m z_p$} for all $m\in[1,\kappa-1]$.
\end{Lemma}

The proof of Lemma \ref{3.8} follows from a similar argument as in the proof of \cite[Lemma 3.9]{AKS06}, bearing in mind the presence of extra poles in the twisted case, as compared to the untwisted case described in~\cite{AK07,AKS06}.

Together with Theorem~\ref{3.4} and the lowest weight condition, we have the following characterization of the functions in $\varphi_{\bm{\mu}}(\tilde{\Gamma}_{\bm{\mu}})$ for each multipartition $\bm{\mu}\vdash\overline{\mathbf{m}}$.

\begin{Theorem}\label{3.9}
Let $\bm{\mu}$ be a multipartition of $\overline{\mathbf{m}}$. Then the image \smash{$\varphi_{\bm{\mu}}\bigl(\tilde{\Gamma}_{\bm{\mu}}\bigr)$} of \smash{$\tilde{\Gamma}_{\bm{\mu}}$} under the evaluation map $\varphi_{\bm{\mu}}\colon\mathcal{U}[\overline{\mathbf{m}}]\to \mathcal{H}[\bm{\mu}]$ is a subspace of $\tilde{\mathcal{H}}[\bm{\mu}]\subseteq\mathcal{H}[\bm{\mu}]$, where $\tilde{ \mathcal{H}}[\bm{\mu}]$ is the space of rational functions
$h(\mathbf{y}_{\bm{\mu}})$ in $\mathbf{y}_{\bm{\mu}}$ of the form
\begin{align}
h(\mathbf{y}_{\bm{\mu}})
={}&\frac{\prod_{a\in I_r}\prod_{(i,r)>(j,s)}\bigl(\bigl(y_{i,r}^{(a)}\bigr)^{t_a^{\vee}}-\bigl(y_{j,s}^{(a)}\bigr) ^{t_a^{\vee}}\bigr)^{2\min(i,j)}}{\prod_{p=1}^N \prod_{i,r}\bigl(\bigl(y_{i,r}^{(a_p)}\bigr)^{t_{a_p}^{\vee}}-z_p^{t_{a_p} ^{\vee}})^{\min(i,i_p)}\prod_{\substack{a<b\\ \overline{C}_{ab}<0}}\prod_{i, j,r,s}\bigl(\bigl(y_{i,r}^{(a)}\bigr)^{k_{ab}}-\bigl(y_{j,s}^{(b)}\bigr) ^{k_{ab}}\bigr)^{\min(i,j)}}\nonumber\\
&\times h_1(\mathbf{y}_{\bm{\mu}})\label{eq:3.45}
\end{align}
with $k_{ab}=\max(t_a^{\vee},t_b^{\vee})$ for all $a,b\in I_r$ satisfying $\overline{C}_{ab}<0$, and $h_1(\mathbf{y}_{\bm{ \mu}})$ is an arbitrary polynomial in $\mathbf{y}_{\bm{\mu}}$ that is symmetric in \smash{$\{y_{i,r}^{(a)}\mid r\in[1,m_{a,i}]\}$} for all pairs $(a,i)$, such that the total degree~\smash{$\deg_{y_{i,r}^{(a)}}h( \mathbf{y}_{\bm{\mu}})$} of the function $h(\mathbf{y}_{\bm{ \mu}})$ in the variable \smash{$y_{i,r}^{(a)}$} is less than or equal to $-2t_a^{\vee}i$ for all $a\in I_r$ and parts $(i,r)$ of \smash{$\mu^{(a)}$}.
\end{Theorem}

As the kernel of the evaluation map $\varphi_{\bm{\mu}} \colon\tilde{\Gamma}_{\bm{\mu}}\to\tilde{\mathcal{H}}[\bm{\mu}]$ is given by $\tilde{\Gamma}_{\bm{\mu}}\cap\ker\varphi_{\bm{\mu}} =\tilde{\Gamma}_{\bm{\mu}}'$, we have an induced injective map \smash{$\overline{\varphi}_{\bm{\mu}}\colon\tilde{\Gamma}_{\bm{\mu}}/\tilde{ \Gamma}_{\bm{\mu}}'\to\tilde{\mathcal{H}}[\bm{\mu}]$}. As in the untwisted case \cite{AK07,AKS06}, we are unable to show that the map \smash{$\overline{ \varphi}_{\bm{\mu}}\colon\tilde{\Gamma}_{\bm{\mu}}/\tilde{\Gamma} _{\bm{\mu}}'\to\tilde{\mathcal{H}}[\bm{\mu}]$} is surjective directly, which will only allow us to deduce Theorem~\ref{3.1} at this point. However, we will see in Section \ref{Section 4} that Theorems \ref{1.1} and \ref{1.2} imply the surjectivity of the evaluation map $\overline{\varphi}_{ \bm{\mu}}\colon\tilde{\Gamma}_{\bm{\mu}}/\tilde{\Gamma}_{\bm{\mu}}' \to\tilde{\mathcal{H}}[\bm{\mu}]$.

\subsection[Proof of Theorem 3.1]{Proof of Theorem~\ref{3.1}}\label{Section 3.6}

We are now ready to prove Theorem~\ref{3.1}. As in \cite{AK07,AKS06}, we will first relate the $q$-graded multiplicity~$\mathcal{M}_{\overline{\lambda},\mathbf{n}}(q)$ of $V\bigl(\overline{\lambda}\bigr)$ in the fusion product $\mathcal{F}_{\mathbf{n}}^*$ of localized twisted KR-modules parameterized by~$\mathbf{n}$ to the $q$-graded dimension of \smash{$\overline{\mathcal{C}}_{\overline{\lambda},\mathbf{n}}[\overline{\mathbf{m}}]$}, or equivalently, the $q$-graded dimension of \smash{$\gr\tilde{\mathcal{F}}_{\overline{ \mathbf{m}}}$}. Subsequently, we will compute the $q$-graded dimension of $\tilde{\mathcal{H}}[\bm{\mu}]$ for each multipartition $\bm{\mu}$ of~$\overline{\mathbf{m}}$ using the characterization of the functions in $\tilde{\mathcal{H}}[\bm{\mu}]$ given in Theorem~\ref{3.9}, and show that each~$\tilde{\mathcal{H}}[\bm{\mu}]$ whose $q$-graded dimension is nonzero corresponds to a (nonzero) term on the right-hand side of \eqref{eq:1.5}. Together with the injectivity of the induced evaluation map $\overline{\varphi}_{\bm{\mu}}\colon\tilde{ \Gamma}_{\bm{\mu}}/\tilde{\Gamma}_{\bm{\mu}}'\to\tilde{ \mathcal{H}}[\bm{\mu}]$ and~\eqref{eq:3.44}, this will then show that $\mathcal{M}_{\overline{\lambda},\mathbf{n}}\bigl(q^{-1}\bigr)\leq M_{\overline{\lambda},\mathbf{n}}\bigl(q^{-1}\bigr)$.

To begin, we first note that the space \smash{$\overline{ \mathcal{C}}_{\overline{\lambda},\mathbf{n}}[\overline{ \mathbf{m}}]$} has a filtration by homogeneous total degree in the variables \smash{$\bigl\{x_i^{(a)}\mid a\in I_r,\,i\in \bigl[1,m^{(a)}\bigr]\bigr\}$}, where we take the degree of the factors $\bigl(\smash{\bigl(x_i^{(a_p)}\bigr)^{t_{a_p}^{ \vee}}}\allowbreak-\smash{z_p^{t_{a_p}^{\vee}}}\bigr)$ in the denominator on the right-hand side of \eqref{eq:3.38} to be equal to $-t_{a_p}^{\vee}$ for all $p\in[1,N]$ and \smash{$i\in\bigl[1,m^{(a_p)}\bigr]$}, and is equivalent to setting $z_p=0$ for all $p\in[1,N]$ in \eqref{eq:3.38}. Let us denote the graded components by homogeneous total degree of \smash{$\overline{\mathcal{C}}_{\overline{\lambda},\mathbf{n}}[\overline{\mathbf{m}}]$} by \smash{$\overline{\mathcal{C}}_{ \overline{\lambda},\mathbf{n}}[\overline{\mathbf{m}}][n]$}. Similarly, each graded component $\tilde{\Gamma}_{\bm{\mu}}/ \tilde{\Gamma}_{\bm{\mu}}'$ of the associated graded space $\gr\tilde{\mathcal{F}}_{\overline{\mathbf{m}}}$ of the filtration $\tilde{\mathcal{F}}_{\overline{\mathbf{m}}}$ on $\overline{\mathcal{C}}_{\overline{\lambda},\mathbf{n}}[\overline{\mathbf{m}}]$ admits a filtration by homogeneous total degree in \smash{$\bigl\{x_i^{(a)}\mid a\in I_r,\,i\in \bigl[1,m^{(a)}\bigr]\bigr\}$} for all multipartitions $\bm{\mu}$ of $\overline{\mathbf{m}}$.

Now, as the degree of the coefficient \smash{$\bigl(x_i^{(a)}\bigr)^{-t_a^{\vee}(n+1)}$} of $\overline{f}_a[t_a^{\vee}n]$ in the generating function \smash{$\overline{f}_a\bigl(x_i^{(a)}\bigr)$} is equal to $-t_a^{\vee}(n+1)$, it follows that the space $\Hom_{\mathfrak{g}^{\sigma}}\bigl(\mathcal{F}_{\mathbf{n}}^*[m],V\bigl(\overline{\lambda}\bigr)\bigr)$ is dual to the space $\overline{ \mathcal{C}}_{\overline{ \lambda},\mathbf{n}}[\overline{ \mathbf{m}}][-m-C_{\overline{\mathbf{m}}}]$ for all $m\in \mathbb{Z}_+$, where \smash{$C_{\overline{\mathbf{m}}}=\sum_{a\in I_r}t_a^{\vee}m^{(a)}$}. Thus, by letting $\ch_q V=\sum_{m\in \mathbb{Z}}\dim V[m]q^m$ denote the generating function of the dimensions of the $\mathbb{Z}$-graded components~$V[m]$ of $V$ for any $\mathbb{Z}$-graded vector space $V=\bigoplus_{n\in \mathbb{Z}} V[m]$, it follows from \eqref{eq:1.8} and \eqref{eq:3.44} that we have
\begin{align*}
\mathcal{M}_{\overline{\lambda},\mathbf{n}}(q)
&=\sum_{m=0}^{\infty}\dim\Hom_{\mathfrak{g}^{\sigma}}\bigl(\mathcal{F}_{\mathbf{n}}^*[m],V\bigl(\overline{\lambda}\bigr)\bigr)q^m=\sum_{m=0}^{\infty}\dim\overline{ \mathcal{C}}_{\overline{ \lambda},\mathbf{n}}[\overline{\mathbf{m}}][-m-C_{\overline{ \mathbf{m}}}]q^m\\
&=q^{-C_{\overline{\mathbf{m}}}}\sum_{m=0}^{\infty}\dim \overline{\mathcal{C}}_{\overline{\lambda},\mathbf{n}}[\overline{\mathbf{m}}][-m-C_{\overline{\mathbf{m}}}]q^{m+C_{ \overline{\mathbf{m}}}}=q^{-C_{\overline{\mathbf{m}}}}\ch_{q^{-1}} \overline{ \mathcal{C}}_{\overline{ \lambda},\mathbf{n}}[\overline{ \mathbf{m}}]\\
&=q^{-C_{\overline{\mathbf{m}}}}\sum_{\bm{\mu}\vdash \overline{\mathbf{m}}}\ch_{q^{-1}}\tilde{\Gamma}_{\bm{\mu}}/\tilde{ \Gamma}_{\bm{\mu}}',
\end{align*}
or equivalently,
\begin{equation}\label{eq:3.46}
\mathcal{M}_{\overline{\lambda},\mathbf{n}}\bigl(q^{-1}\bigr)
=q^{C_{\overline{\mathbf{m}}}}\ch_q\overline{\mathcal{C}}_{ \overline{\lambda},\mathbf{n}}[\overline{\mathbf{m}}]
=q^{C_{\overline{\mathbf{m}}}}\sum_{\bm{\mu}\vdash\overline{ \mathbf{m}}}\ch_q\tilde{\Gamma}_{\bm{\mu}}/\tilde{\Gamma}_{ \bm{\mu}}'.
\end{equation}
Likewise, the space of functions $\tilde{\mathcal{H}}[\bm{ \mu}]$ admit a filtration by homogeneous total degree in the variables \smash{$\bigl\{y_{i,r}^{(a)}\mid a\in I_r,\,i\in\bigl[1,m^{(a)}\bigr],\,r\in[1, m_{a,i}]\bigr\}$} for each multipartition $\bm{\mu}$ of $\overline{\mathbf{m}}$. As $\overline{\varphi}_{\bm{\mu}}\colon \tilde{\Gamma}_{\bm{\mu}}/\tilde{\Gamma}_{\bm{\mu}}'\to \smash{\tilde{\mathcal{H}}[\bm{\mu}]}$ is an injective map of $\mathbb{Z}$-graded vector spaces, we have
\smash{$
\ch_q\tilde{ \Gamma}_{\bm{\mu}}/\tilde{\Gamma}_{\bm{\mu}}'
\leq\ch_q\tilde{\mathcal{H}}[\bm{\mu}]
$}
for all multipartitions $\bm{\mu}$ of $\overline{\mathbf{m}}$. Together with \eqref{eq:3.46}, we have
\begin{equation}\label{eq:3.47}
\mathcal{M}_{\overline{\lambda},\mathbf{n}}\bigl(q^{-1}\bigr)
\leq q^{C_{\overline{\mathbf{m}}}}\sum_{\bm{\mu}\vdash \overline{\mathbf{m}}}\ch_q\tilde{\mathcal{H}}[\bm{\mu}].
\end{equation}
It remains to show that the right-hand side of the inequality \eqref{eq:3.47} is equal to $M_{\overline{\lambda},\mathbf{n}}\bigl(q^{-1}\bigr)$. To this end, we will need to compute \smash{$\ch_q\tilde{ \mathcal{H}}[\bm{\mu}]$} for each multipartition $\bm{\mu}$ of $\overline{\mathbf{m}}$. We first note that the associated graded space of the filtration by homogeneous total degree in the variables $\smash{\bigl\{y_{i,r}^{(a)}}\mid a\in I_r,\, i\in\smash{\bigl[1,m^{(a)}\bigr]},\, r\in [1, m_{a,i}]\bigr\}$ on $\tilde{\mathcal{H}}[\bm{\mu}]$ is obtained by setting $z_p=0$ for all $p\in[1,N]$ in~\eqref{eq:3.45}, and by Theorem~\ref{3.9}, this associated graded space is isomorphic to the space of functions of the form~${
h(\mathbf{y}_{\bm{\mu}})
=h_0(\mathbf{y}_{\bm{\mu}})h_1(\mathbf{y}_{\bm{\mu}})}$,
where
\begin{align*}
 h_0(\mathbf{y}_{\bm{\mu}})
&=\frac{\prod_{a\in I_r}\prod_{(i,r)>(j,s)}\bigl(\bigl(y_{i,r}^{(a)}\bigr)^{t_a^{\vee}}-\bigl(y_{j,s}^{(a)}\bigr)^{t_a^{\vee}}) ^{2\min(i,j)}}{\prod_{p=1}^N\prod_{i,r}\bigl(y_{i,r} ^{(a_p)}\bigr)^{\min(i,i_p)t_{a_p}^{\vee}} \prod_{\substack{a<b\\ \overline{C}_{ab}<0}}\prod_{i,j,r,s}\bigl(\bigl(y_{i,r}^{(a)}\bigr)^{k_{ab}}-\bigl(y_{j,s}^{(b)}\bigr) ^{k_{ab}}\bigr)^{\min(i,j)}}\\
&=\frac{\prod_{a\in I_r}\prod_{(i,r)>(j,s)}\bigl(\bigl(y_{i,r}^{(a)}\bigr)^{t_a^{\vee}}-\bigl(y_{j,s}^{(a)}\bigr)^{t_a^{\vee}}\bigr) ^{2\min(i,j)}}{\prod_{b\in I_r}\prod_{j\in \mathbb{N}}\prod_{i,r}\bigl(y_{i,r}^{(b)}\bigr)^{n_{b,j}\min(i,j) t_b^{\vee}}\prod_{\substack{a<b\\ \overline{C}_{ab}<0}} \prod_{i,j,r,s}\bigl(\bigl(y_{i,r}^{(a)}\bigr)^{k_{ab}}-\bigl(y_{j,s}^{(b)}\bigr) ^{k_{ab}}\bigr)^{\min(i,j)}},
\end{align*}
with $k_{ab}=\max(t_a^{\vee},t_b^{\vee})$ for all $a,b\in I_r$ satisfying $\overline{C}_{ab}<0$,
\begin{equation}\label{eq:3.49}
\deg_{y_{i,r}^{(a)}}h(\mathbf{y}_{\bm{\mu}})
\leq-2t_a^{\vee}i
\end{equation}
for all $a\in I_r$ and parts $(i,r)$ of \smash{$\mu^{(a)}$}, and $h_1(\mathbf{y}_{\bm{\mu}})$ is an arbitrary polynomial in $\mathbf{y}_{\bm{\mu}}$ that is symmetric in \smash{$\bigl\{y_{i,r}^{(a)}\mid r\in[1,m_{a,i}]\bigr\}$} for all pairs $(a,i)$.

Now, using the fact that we have \smash{$\sum_{i\in\mathbb{N}} im_{a,i}=m^{(a)}$} for all $a\in I_r$, we may rewrite the left-hand side of~\eqref{eq:3.32} as
\begin{align*}
\ell_a+\sum_{b\in I_r}\overline{C}_{ab} m^{(b)}-\sum_{i\in \mathbb{N}}in_{a,i}
&=\ell_a+\sum_{i\in\mathbb{N}}\sum_{b\in I_r}i\overline{C}_{ ab} m_{b,i}-\sum_{i\in\mathbb{N}}in_{a,i}\\
&=\ell_a+\sum_{i\in\mathbb{N}}\sum_{b\in I_r}i\bigl(\overline{C}_{ ab} m_{b,i}-\delta_{ab}n_{b,i}\bigr)=q_{a,0},
\end{align*}
from which we see that the zero weight condition \eqref{eq:3.32} is equivalent to requiring the vector~${\mathbf{m}=(m_{a,i})_{a\in I_r,i\in\mathbb{N}}}$ satisfy $q_{a,0}=0$ for all $a\in I_r$.
Next, let us compute the total degree \smash{$\deg_{y_{i,r}^{(a)}} h_0(\mathbf{y}_{\bm{\mu}})$} of the function $h_0(\mathbf{ y}_{\bm{\mu}})$ in the variable \smash{$y_{i,r}^{(a)}$} for all $a\in I_r$ and parts $(i,r)$ of \smash{$\mu^{(a)}$}. As~we have $\overline{C}_{aa}=2$, and $k_{ab}=\max(t_a^{\vee},t_b^{\vee})=-t_a^{\vee}\overline{C}_{ab}=-t_b^{\vee}\overline{C}_{ba}$ for all $b\in I_r$ satisfying $\overline{C}_{ab}<0$, it follows that we have
\begin{align}
\deg_{y_{i,r}^{(a)}} h_0(\mathbf{y}_{\bm{\mu}})
={}&\sum_{(j,s)\neq(i,r)}2t_a^{\vee}\min(i,j)-\sum_{j\in\mathbb{N}} t_a^{\vee}\min(i,j)n_{a,j}-\sum_{\substack{b\neq a\\ \overline{C}_{ab}<0}}\sum_{(j,s)}k_{ab}\min(i,j)\nonumber\\
={}&-2t_a^{\vee}i+\sum_{j\in\mathbb{N}}2t_a^{\vee}\min(i,j)m_{a,j}-\sum_{j\in\mathbb{N}}\sum_{b\in I_r}t_a^{\vee}\min(i,j)\delta_{ab}n_{b,j}\nonumber\\
&+\sum_{j\in\mathbb{N}}\sum_{\substack{b\neq a\\ \overline{C}_{ab}<0}}t_a^{\vee}\min(i,j)\overline{C}_{ab}m_{b,j} \nonumber\\
={}&-2t_a^{\vee}i+t_a^{\vee}\sum_{j\in\mathbb{N}}\min(i,j)\overline{C}_{aa}m_{a,j}-t_a^{\vee}\sum_{j\in\mathbb{N}}\sum_{b\in I_r}\min(i,j)\delta_{ab}n_{b,j}\nonumber\\
&+t_a^{\vee}\sum_{j\in\mathbb{N}}\sum_{b\neq a}\min(i,j) \overline{C}_{ab}m_{b,j}\nonumber\\
={}&-2t_a^{\vee}i-t_a^{\vee}\sum_{j\in\mathbb{N}}\sum_{b\in I_r}\min(i,j)\delta_{ab}n_{b,j}+t_a^{\vee}\sum_{j\in\mathbb{N}}\sum_{b\in I_r}\min(i,j)\overline{C}_{ab}m_{b,j}\nonumber\\
={}&-2t_a^{\vee}i-t_a^{\vee}p_{a,i}.\label{eq:3.50}
\end{align}
Moreover, as $h_1(\mathbf{y}_{\bm{\mu}})$ is a polynomial in $\mathbf{y}_{ \bm{\mu}}$, it follows from \eqref{eq:3.49} and \eqref{eq:3.50} that we have
\begin{equation}\label{eq:3.51}
0\leq \deg_{y_{i,r}^{(a)}}h_1(\mathbf{y}_{\bm{\mu}})
\leq t_a^{\vee}p_{a,i}
\end{equation}
for all $a\in I_r$ and parts $(i,r)$ of $\mu^{(a)}$. Thus, the inequality \eqref{eq:3.51}, the zero weight condition, and the bijection \eqref{eq:3.21} together imply that $\dim\tilde{\mathcal{H}}[\bm{\mu}]>0$ if and only if the vector $\mathbf{m}=(m_{a,i})_{a\in I_r,i\in \mathbb{N}}$ satisfies $q_{a,0}=0$ and
$p_{a,i}\geq0$ for all $a\in I_r$ and $i\in\mathbb{N}$.

Now, we let $\mathcal{P}[\bm{\mu}]$ denote the space of polynomials $p(\mathbf{y}_{\bm{\mu}})$ in $\mathbf{y}_{\bm{\mu}}$ that are symmetric in $\bigl\{\smash{y_{i,r}^{(a)}}\mid r\in[1,m_{a,i}]\bigr\}$ for all pairs $(a,i)$, and satisfy the degree restrictions \eqref{eq:3.51}. Then we have
\begin{gather}
\ch_q\tilde{\mathcal{H}}[\bm{\mu}]
=q^{\deg h_0(\mathbf{y}_{\bm{\mu}})}\ch_q\mathcal{P}[\bm{\mu}],\label{eq:3.52}\\
\mathcal{P}[\bm{\mu}]
\cong\bigotimes_{a\in I_r}\bigotimes_{\substack{I\in\mathbb{N}\\m_{a,i}>0}}\mathcal{P}[\bm{\mu}]_{a,i},\label{eq:3.53}
\end{gather}
where $\mathcal{P}[\bm{\mu}]_{a,i}$ is the space of polynomials $p(\mathbf{y}_{a,i})$ in \smash{$\mathbf{y}_{a,i}=\bigl\{\bigl(y_{i,1}^{(a)}\bigr)^{ t_a^{\vee}},\dots,\bigl(y_{i,m_{a,i}}^{(a)}\bigr)^{t_a^{\vee}}\bigr\}$} that are symmetric in \smash{$\bigl\{y_{i,r}^{(a)}\mid r\in[1,m_{a,i}]\bigr\}$} for all pairs $(a,i)$, and satisfy the degree restriction \eqref{eq:3.51}. As~the degree of each variable \smash{$y_{i,r}^{(a)}$} of $p(\mathbf{y}_{a,i})$ is a multiple of $t_a^{\vee}$ for all polynomials $p(\mathbf{y}_{a,i})\in\mathcal{P}[\bm{\mu}]_{a,i}$, we have
\begin{equation*}
\ch_q\mathcal{P}[\bm{\mu}]_{a,i}=\begin{bmatrix} m_{a,i}+p_{a,i} \\m_{a,i}\end{bmatrix}_{q^{t_a^{\vee}}}=\begin{bmatrix} m_{a,i}+p_{a,i}\\m_{a,i}\end{bmatrix}_{q_a}
\end{equation*}
for all $a\in I_r$ and $i\in\mathbb{N}$ that satisfy $m_{a,i}>0$. Moreover, as
\begin{equation*}
\begin{bmatrix} m_{a,i}+p_{a,i}\\m_{a,i}\end{bmatrix}_{q_a}=1
\end{equation*}
for all $a\in I_r$ and $i\in\mathbb{N}$ that satisfy $m_{a,i}=0$, it follows from \eqref{eq:3.53} that we have
\begin{equation}\label{eq:3.54}
\ch_q\mathcal{P}[\bm{\mu}]
=\prod_{i\in\mathbb{N}}\prod_{a\in I_r}\ch_q\mathcal{P}[\bm{\mu}]_{a,i}
=\prod_{i\in\mathbb{N}}\prod_{a\in I_r}\begin{bmatrix} m_{a,i}+p_{a,i} \\m_{a,i}\end{bmatrix}_{q_a}.
\end{equation}
Finally, to complete the proof of Theorem~\ref{3.1}, we would need to show that $\deg h_0(\mathbf{y}_{\bm{\mu}})=-C_{\overline{ \mathbf{m}}}+Q(\mathbf{m},\mathbf{n})$. Indeed, we have
\begin{align*}
\deg h_0(\mathbf{y}_{\bm{\mu}})
={}&\sum_{a\in I_r}\sum_{(i,r)>(j,s)}2t_a^{\vee}\min(i,j)-\sum_{a\in I_r}\sum_{j\in\mathbb{N}}\sum_{(i,r)} t_a^{\vee}\min(i,j)n_{a,j}\\
&-\sum_{\substack{a<b\\ \overline{C}_{ab}<0}}\sum_{(i,r),(j,s)}k_{ab}\min(i,j)\\
={}&\sum_{a\in I_r}\sum_{(i,r)\neq(j,s)}t_a^{\vee}\min(i,j)-\sum_{a\in I_r}\sum_{i,j\in\mathbb{N}} t_a^{\vee}\min(i,j) m_{a,i}n_{a,j}\\
&+\sum_{\substack{a<b\\ \overline{C}_{ab}<0}}\sum_{i,j\in\mathbb{N}}t_a^{\vee}\overline{C}_{ab}\min(i,j)m_{a,i}m_{b,j}\\
={}&-\sum_{a\in I_r}t_a^{\vee}\sum_{i\in\mathbb{N}}im_{a,i}+\sum_{a\in I_r}\sum_{i,j\in\mathbb{N}}t_a^{\vee}\min(i,j)m_{a,i}m_{a,j}\\
&-\sum_{a,b\in I_r}\sum_{i,j\in\mathbb{N}} t_a^{\vee}\min(i,j)\delta_{ab}m_{a,i}n_{b,j}+\sum_{a<b}\sum_{i,j\in\mathbb{N}}t_a^{\vee}\overline{C}_{ab}\min(i,j)m_{a,i}m_{b,j}\\
={}&-\sum_{a\in I_r}t_a^{\vee}m^{(a)}+\frac{1}{2}\sum_{a\in I_r}\sum_{i,j\in\mathbb{N}}t_a^{\vee}\overline{C}_{aa}\min(i,j) m_{a,i}m_{a,j}\\
&-\sum_{a,b\in I_r}\sum_{i,j\in\mathbb{N}} t_a^{\vee}\min(i,j)\delta_{ab}m_{a,i}n_{b,j}\\
&+\frac{1}{2}\sum_{a<b}\sum_{i,j\in\mathbb{N}}\bigl(t_a^{\vee}\overline{C}_{ab}+t_b^{\vee}\overline{C}_{ba}\bigr)\min(i,j)m_{a,i}m_{b,j}\\
={}&-C_{\overline{\mathbf{m}}}+\frac{1}{2}\sum_{a\in I_r}\sum_{i,j\in\mathbb{N}}t_a^{\vee}\overline{C}_{aa}\min(i,j) m_{a,i}m_{a,j}\\
&-\sum_{a,b\in I_r}\sum_{i,j\in\mathbb{N}} t_a^{\vee}\min(i,j)\delta_{ab}m_{a,i}n_{b,j}+\frac{1}{2}\sum_{a\neq b}\sum_{i,j\in\mathbb{N}}t_a^{\vee}\overline{C}_{ab}\min(i,j)m_{a,i}m_{b,j}\\
={}&-C_{\overline{\mathbf{m}}}+\frac{1}{2}\sum_{i,j\in\mathbb{N}} \sum_{a,b\in I_r}t_a^{\vee}\min(i,j)m_{a,i}\bigl( \overline{C}_{ab}m_{b,j}-2\delta_{ab}n_{b,j}\bigr)\\
={}&-C_{\overline{\mathbf{m}}}+Q(\mathbf{m},\mathbf{n}).
\end{align*}
Together with \eqref{eq:3.47}, \eqref{eq:3.52} and \eqref{eq:3.54}, we have
\begin{align*}
\mathcal{M}_{\overline{\lambda},\mathbf{n}}\bigl(q^{-1}\bigr)
&\leq \sum_{\substack{\bm{\mu}\vdash \overline{\mathbf{m}}\\ \dim\tilde{\mathcal{H}}[\bm{\mu}]>0}} q^{Q( \mathbf{m},\mathbf{n})}\prod_{i\in\mathbb{N}}\prod_{a\in I_r} \begin{bmatrix} m_{a,i}+p_{a,i}\\m_{a,i}\end{bmatrix}_{q_a}\\
&=\sum_{\substack{\mathbf{m}\geq\mathbf{0}\\q_{a,0}=0,p_{a,i}\geq0}}q^{Q(\mathbf{m},\mathbf{n})}\prod_{i\in \mathbb{N}}\prod_{a\in I_r}\begin{bmatrix}m_{a,i}+p_{a,i}\\ m_{a,i}\end{bmatrix}_{q_a}=M_{\overline{\lambda},\mathbf{n}}\bigl(q^{-1}\bigr),
\end{align*}
and this completes the proof of Theorem~\ref{3.1}.

\section[Quantum twisted Q-systems and graded fermionic sums]{Quantum twisted $\boldsymbol{Q}$-systems and graded fermionic sums}\label{Section 4}

In this section, we will prove Theorems \ref{1.1} and \ref{1.2}. As in \cite{DFK08, DFK14, Lin21}, we will prove a slightly stronger statement, where we fix a positive integer $k$, and define $k$-restricted sums \smash{$M_{\overline{\lambda}, \mathbf{n}}^{(k)}\bigl(q^{-1}\bigr)$} and \smash{$\tilde{M}_{\overline{\lambda}, \mathbf{n}}^{(k)}\bigl(q^{-1}\bigr)$} by restricting the respective sums $M_{\overline{\lambda}, \mathbf{n}}\bigl(q^{-1}\bigr)$ and $\tilde{M}_{\overline{\lambda}, \mathbf{n}}\bigl(q^{-1}\bigr)$ to the vectors $\mathbf{m}$ that satisfy $m_{a,i}=0$ for all $a\in I_r$ and $i>k$, and show that
\[
M_{\overline{\lambda}, \mathbf{n}}^{(k)}\bigl(q^{-1}\bigr)=\tilde{M}_{\overline{\lambda}, \mathbf{n}}^{(k)}\bigl(q^{-1}\bigr)
\]
 are equal to each other. As both \smash{$M_{\overline{\lambda}, \mathbf{n}}^{(k)}\bigl(q^{-1}\bigr)$} and \smash{$\tilde{M}_{\overline{\lambda}, \mathbf{n}}^{(k)}\bigl(q^{-1}\bigr)$} are equal to \smash{$M_{\overline{\lambda}, \mathbf{n}}\bigl(q^{-1}\bigr)$} and \smash{$\tilde{M}_{\overline{\lambda}, \mathbf{n}}\bigl(q^{-1}\bigr)$}, respectively, whenever $k$ is large, this will show that Theorem~\ref{1.1} holds. Together with a similar argument as in \cite[Section~1]{Kedem11}, this will show that Theorem~\ref{1.2} holds as well.

\subsection[Quantum twisted Q-systems]{Quantum twisted $\boldsymbol{ Q}$-systems}\label{Section 4.1}

In this subsection, we will review the definition and properties of the quantum $X_m^{(\kappa)}$ $Q$-systems for all twisted affine types \smash{$X_m^{(\kappa)}\neq A_{2r}^{(2)}$}. These quantum twisted $Q$-systems arise from the quantization of the twisted $Q$-system cluster algebras defined in \cite{Williams15}. While explicit formulas for the~quantum \smash{$X_m^{(\kappa)}$} $Q$-system relations were only given for the classical twisted affine types in~\cite[equations~(4.2)--(4.4)]{DFK24}, that is, \smash{$X_m^{(\kappa)}=A_{2r-1}^{(2)}, D_{r+1}^{(2)}$}, the quantum \smash{$X_m^{(\kappa)}$} $Q$-system relations for the~exceptional twisted affine types, that is, \smash{$X_m^{(\kappa)}=E_6^{(2)}, D_4^{(3)}$}, can be derived in a similar fashion as the quantum \smash{$X_m^{(\kappa)}$} $Q$-system relations for the classical twisted affine types \smash{$\neq A_{2r}^{(2)}$}.

More generally, the quantum \smash{$X_m^{(\kappa)}$} $Q$-system relations for all affine types \smash{$\neq A_{2r}^{(2)}$} can be derived from the $Q$-system cluster algebras defined in \cite{DFK09,Kedem08,Williams15} in a uniform manner, following the procedure outlined in \cite[Remark 4.2]{DFK24}, which we will use to obtain a uniform derivation of the quantum \smash{$X_m^{(\kappa)}$} $Q$-systems for all twisted affine types \smash{$X_m^{(\kappa)}\neq A_{2r}^{(2)}$} below the fold.

Let us start by first recalling that the \smash{$X_m^{(\kappa)}$} $Q$-system cluster algebra is defined from the initial cluster data $\{Q_{a,0},Q_{a,1}\mid a\in I_r\}$, whose exchange matrix is given by\footnote{Our notation differs from \cite{Williams15}, which uses $B^T$ instead of $B$.}
$
B
=\big(\begin{smallmatrix}
0 & -\overline{C} \\
\overline{C} & 0
\end{smallmatrix}\big)$,
where~the order of both the row and column indices appearing in the exchange matrix $B$ correspond to that in the ordered initial cluster data $(Q_{1,0},\dots,Q_{r,0},Q_{1,1},\dots,Q_{r,1})$ of the \smash{$X_m^{(\kappa)}$} $Q$-system. We~let~${\delta=\det\bigl(\overline{C}\bigr)}$, $T^{\vee}=\diag(t_1^{\vee} ,\dots,t_r^{\vee})$, and \smash{$\Lambda=\delta T^{\vee}\overline{ C}^{-1}$}. Then $\Lambda$ is a symmetric matrix. Following \cite[Definition 3.1]{BZ05}, we define a skew-symmetric $\nu$-commutation matrix $\tilde{\Lambda}$ with respect to the exchange matrix $B$ by
$
\tilde{\Lambda}=\big(
\begin{smallmatrix}
0 & -\Lambda \\
\Lambda & 0
\end{smallmatrix}\big)$,
so that $\bigl(\tilde{\Lambda},B\bigr)$ forms a compatible pair. We~let~$\widehat{Q}_{a,i}$ denote the quantum cluster variable, which we call a quantum $Q$-system variable, that corresponds to the $Q$-system variable $Q_{a,i}$ in the quantum \smash{$X_m^{(\kappa)}$} $Q$-system cluster algebra associated to the compatible pair $\bigl(\tilde{\Lambda},B\bigr)$. These quantum $Q$-system variables satisfy two types of relations. The first type of relations are the $\nu$-commutation relations
\begin{equation*}
\widehat{Q}_{a,k+i}\widehat{Q}_{b,k+j}=\nu^{(i-j)\Lambda_{ab}} \widehat{Q}_{b,k+j}\widehat{Q}_{a,k+i},\qquad a,b\in I_r,\quad i,j\in\{0,1\},\quad k\in\mathbb{Z}.
\end{equation*}
The second type of relations are the quantum \smash{$X_m^{(\kappa)}$} $Q$-system relations, which are quantum deformations of the cluster transformations that correspond to the classical \smash{$X_m^{(\kappa)}$} $Q$-system relations~\eqref{eq:1.1}. Up to a renormalization of the quantum $Q$-system variables $\widehat{Q}_{a,i}$ (see, for instance, \cite[Lemma~4.4]{DFK21} and the derivation of \cite[equation~(3.20)]{Lin21}), the quantum \smash{$X_m^{(\kappa)}$} $Q$-system relations are given by
\begin{equation}\label{eq:4.4}
\nu^{-\Lambda_{aa}}\widehat{Q}_{a,k+1}\widehat{Q}_{a,k-1}= \widehat{Q}_{a,k}^2-\prod_{b\sim a}\widehat{Q}_{b,k}^{-\overline{C}_{ba}}
\end{equation}
for all $a\in I_r$ and $k\in\mathbb{Z}$.

\begin{remark*}
In the case where \smash{$X_m^{(\kappa)}=A_{2r-1}^{(2)}, D_{r+1}^{(2)}$}, our choice of the matrices $\Lambda$ and $\tilde{\Lambda}$ differ from those in \cite{DFK24} as follows: the matrix $\Lambda$ defined in \cite[Definition 4.1]{DFK24} is defined by \smash{$\Lambda=pT^{\vee}\overline{C}^{-1}$}, where $p$ is chosen so that $\Lambda_{11}=1$, and the skew-symmetric $\nu$-commutation matrix $\tilde{\Lambda}$ defined in \cite[Remark 4.2]{DFK24} differs from ours by a sign. The reader can verify that after accounting for these differences, the $\nu$-commutation relations and quantum $Q$-system relations given here and in \cite{DFK24} are equivalent to each other.
\end{remark*}

For later convenience, we will rewrite the quantum \smash{$X_m^{(\kappa)}$} $Q$-system relations \eqref{eq:4.4} as
\begin{equation}\label{eq:4.5}
\nu^{-\Lambda_{aa}}\widehat{Q}_{a,k+1}\widehat{Q}_{a,k-1}= \widehat{Q}_{a,k}^2\bigl(1-\widehat{Y}_{a,k}\bigr),
\end{equation}
where \smash{$\widehat{Y}_{a,k}= \prod_{b\in I_r}\widehat{Q}_{b,k}^{-\overline{C}_{ba}}$} for all $a\in I_r$ and $k\in\mathbb{Z}$.

Next, we will review some basic definitions and properties involving these quantum twisted $Q$-systems. We begin by first recalling that a Motzkin path of length $\ell-1$ is a vector ${\vec{m}=(m_i)_{i=1}^{\ell}}$ of integers that satisfy $|m_i-m_{i+1}|\leq 1$ for all $i\in[1,\ell-1]$. As an example, an initial data for the quantum \smash{$X_m^{(\kappa)}$} $Q$-system is given by \smash{$\bigl(\widehat{Q}_{a,0},\widehat{Q}_{a,1}\bigr)_{a\in I_r}$}. More generally, by letting $\widehat{\mathbf{Q}}_i=\bigl\{\widehat{Q}_{a,i}\mid a\in I_r\bigr\}$ for all $i\in\mathbb{Z}$, it follows that the set $\widehat{\mathbf{Q}}_k\cup\widehat{\mathbf{Q}}_{k+1}$ forms a valid set of initial data for the quantum \smash{$X_m^{(\kappa)}$} $Q$-system for all $k\in\mathbb{Z}$.

Now, each Motzkin path $\vec{m}=(m_a)_{a\in I_r}$ gives rise to a valid set $\mathbf{Y}_{\vec{m}}=\bigl(\widehat{Q}_{a,m_a}, \widehat{Q}_{a,m_a+1}\bigr)_{a\in I_r}$ of initial data for the quantum $Q$-system. As each $\mathbf{Y}_{\vec{m}}$ corresponds to a cluster consisting of quantum $Q$-system variables, it follows that the variables in $\mathbf{Y}_{\vec{m}}$ lie on the same quantum torus, and hence satisfy some $\nu$-commutation relations, which are given as follows:

\begin{Lemma}\label{4.1}
Let $\vec{m}=(m_a)_{a\in I_r}$ be a Motzkin path. Then for any pair $\widehat{Q}_{a,i}$, $\widehat{Q}_{b,j}$ of quantum $Q$-system variables in $\mathbf{Y}_{\vec{m}}$, we have
\smash{$
\widehat{Q}_{a,i}\widehat{Q}_{b,j}=\nu^{(i-j)\Lambda_{ab}} \widehat{Q}_{b,j}\widehat{Q}_{a,i}$}.
\end{Lemma}

The proof of Lemma \ref{4.1} is similar to that of \cite[Lemma~3.2]{DFK14} with minor modifications, and shall be omitted.

Moreover, as quantum cluster algebras satisfy the Laurent property \cite[Corollary~5.2]{BZ05}, we have that the solutions of the quantum $Q$-system inherit the following Laurent property as well.

\begin{Lemma}\label{4.2}
For any Motzkin path $\vec{m}=(m_a)_{a\in I_r}$, $b\in I_r$ and $i\in\mathbb{Z}$, the quantum $Q$-system variable $\widehat{Q}_{b,i}$ can be expressed as a $($non-commutative$)$ Laurent polynomial in the initial data $\mathbf{Y}_{\vec{m}}$ with coefficients in $\mathbb{Z}\bigl[\nu^{\pm1}\bigr]$.
\end{Lemma}

Finally, similar to the untwisted case, the quantum twisted $Q$-systems satisfy a translational invariance property. To state the property clearly, we will need to write the solution of the quantum twisted $Q$-system as $\widehat{Q}_{a,n}(\mathbf{Y}_{\vec{m}})$ to display its dependence on the initial data $\mathbf{Y}_{\vec{m}}$. Then we have the following.

\begin{Lemma}\label{4.3}
For all $m\in\mathbb{Z}$ and $j\in\mathbb{Z}_+$, we have
\[
\widehat{Q}_{a,m}\bigl(\widehat{\mathbf{Q}}_j\cup\widehat{ \mathbf{Q}}_{j+1}\bigr)
=\widehat{Q}_{a,m+j}\bigl(\widehat{\mathbf{Q}}_0\cup\widehat{ \mathbf{Q}}_1\bigr).
\]
\end{Lemma}

The proof of Lemma \ref{4.3} follows from a similar argument as in the proof of \cite[Lemma 4.10]{DFK14}, and shall be omitted.

\subsection{Quantum generating functions}\label{Section 4.2}

Let us fix a positive integer $k$. Throughout this subsection and the next, we will restrict our attention to vectors $\mathbf{m}=(m_{a,i})_{a\in I_r,i\in\mathbb{N}}$ that satisfy $m_{a,i}=0$ for all $a\in I_r$ and $i>k$, or equivalently, $(r\times k)$-tuples $\mathbf{m}=(m_{a,i})_{a\in I_r,i\in[1,k]}$ of nonnegative integers.

For any $(r\times k)$-tuples $\mathbf{m}=(m_{a,i})_{a\in I_r,i\in[1,k]}$, $\mathbf{n}=(n_{b,j})_{b\in I_r,j\in[1,k]}$ and any dominant $\mathfrak{g}^{\sigma}$-weight \smash{$\overline{\lambda}=\sum_{a\in I_r}\ell_a\overline{\omega}_a$}, we define the $k$-restricted total spin $q_{a,0}$ and the $k$-restricted vacancy numbers~$p_{a,i}$ by
\begin{align*}
q_{a,0}&=\ell_a+\sum_{j=1}^k\sum_{b\in I_r}j\bigl(\overline{ C}_{ab}m_{b,j}-\delta_{ab}n_{b,j}\bigr),\qquad
p_{a,i}=\sum_{j=1}^k\sum_{b\in I_r}\min(i,j)\bigl(\delta_{ ab}n_{b,j}-\overline{C}_{ab}m_{b,j}\bigr).
\end{align*}
We will also define the modified $k$-restricted vacancy numbers $q_{a,i}$ by
\begin{equation}\label{eq:4.8}
q_{a,i}
=q_{a,0}+p_{a,i}
=\ell_a+\sum_{j=i+1}^k\sum_{b\in I_r}(j-i)\bigl(\overline{C}_{ab}m_{b,j}-\delta_{ab} n_{b,j}\bigr).
\end{equation}
We see that when $q_{a,0}=0$, we have $q_{a,i}=p_{a,i}$ for all $a\in I_r$ and $i\in[1,k]$.

Next, let us define the $k$-restricted quadratic form $Q^{(k)}(\mathbf{m},\mathbf{n})$ as follows:
\begin{gather*}
Q^{(k)}(\mathbf{m},\mathbf{n})=\frac{1}{2}\sum_{i,j=1}^k \sum_{a,b\in I_r}t_a^{\vee}\min(i,j)m_{a,i}\bigl(\overline{ C}_{ab}m_{b,j} -2\delta_{ab}n_{b,j}\bigr).
\end{gather*}
The $k$-restricted $M$- and $\tilde{M}$-sums \smash{$M_{\overline{ \lambda},\mathbf{n}}^{(k)}\bigl(q^{-1}\bigr)$} and \smash{$\tilde{M}_{\overline{ \lambda},\mathbf{n}}^{(k)}\bigl(q^{-1}\bigr)$} are then defined by
\begin{align}
M_{\overline{\lambda},\mathbf{n}}^{(k)}\bigl(q^{-1}\bigr)
&=\sum_{\substack{\mathbf{m}\geq\mathbf{0}\\q_{a,0}=0,p_{a,i}\geq0}}q^{Q^{(k)}(\mathbf{m},\mathbf{n})}\prod_{i=1}^k\prod_{a \in I_r}\begin{bmatrix}m_{a,i}+p_{a,i}\\m_{a,i} \end{bmatrix}_{q_a},\nonumber\\
\tilde{M}_{\overline{\lambda},\mathbf{n}}^{(k)}\bigl(q^{-1}\bigr)
&=\sum_{\substack{\mathbf{m}\geq\mathbf{0}\\q_{a,0}=0}}q^{Q^{(k)}(\mathbf{m},\mathbf{n})}\prod_{i=1}^k\prod_{a\in I_r}
\begin{bmatrix}m_{a,i}+p_{a,i}\\m_{a,i}\end{bmatrix}_{q_a}.\label{eq:4.11}
\end{align}
Our first technical lemma involves rewriting the quadratic form $Q^{(k)}(\mathbf{m},\mathbf{n})$ in terms of the modified $k$-restricted vacancy numbers $q_{a,i}$ and $n_{a,i}$.
\begin{Lemma}\label{4.4}
Let $\mathbf{n}_i=(n_{a,i})_{a\in I_r}$, $\mathbf{q}_j=(q_{a,j})_{a\in I_r}$, and $\mathbf{m}_{\ell}=(m_{a,\ell})_{a\in I_r}$, for all $i\in\mathbb{N}$, $j\in[0,k]$, and $\ell\in[1,k]$. Then we have
\begin{equation*}
Q^{(k)}(\mathbf{m},\mathbf{n})=\frac{1}{2\delta}\sum_{j=1}^k \left((\mathbf{q}_{j-1}-\mathbf{q}_j)\cdot\Lambda(\mathbf{ q}_{j-1}-\mathbf{q}_j)-\left(\sum_{i=j}^k\mathbf{n}_i\right) \cdot\Lambda\left(\sum_{i=j}^k\mathbf{n}_i\right)\right).
\end{equation*}
\end{Lemma}

\begin{proof}
Following the strategy in the proof of \cite[Lemma 4.2]{Lin21}, but with simplified indices in the twisted case, we first observe that
\begin{align*}
Q^{(k)}(\mathbf{m},\mathbf{n})
&=\frac{1}{2}\sum_{i,j=1}^k\sum_{a,b\in I_r} t_a^{\vee}\min(i,j)m_{a,i}\bigl(\overline{C}_{a b}m_{b,j}-2\delta_{ab}n_{b,j}\bigr)\\
&=\frac{1}{2}\sum_{i,j=1}^k\sum_{\ell=1}^{\min(i,j)}\sum_{ a,b\in I_r}t_a^{\vee}m_{a,i}\bigl(\overline{C}_{ab}m_{b,j}-2\delta_{ab}n_{b,j}\bigr)\\
&=\frac{1}{2}\sum_{\ell=1}^k\sum_{i,j=\ell}^k\mathbf{m}_i\cdot T^{\vee}\bigl(\overline{C}\mathbf{m}_j-2\mathbf{n}_j\bigr).
\end{align*}
On the other hand, it follows from \eqref{eq:4.8} that we have
$
\mathbf{q}_{j-1}-\mathbf{q}_j =\sum_{i=j}^k\bigl(\overline{C}\mathbf{m}_i-\mathbf{n}_i\bigr)
$
for all~${j\in[1,k]}$. This implies that
\begin{align*}
\sum_{j=1}^k(\mathbf{q}_{j-1}-\mathbf{q}_j)\cdot \Lambda(\mathbf{q}_{j-1}-\mathbf{q}_j)
&=\sum_{j=1}^k\sum_{i,\ell=j}^k\bigl(\overline{C}\mathbf{m}_i-\mathbf{n}_i\bigr)\cdot\Lambda\bigl(\overline{C}\mathbf{m}_{\ell}-\mathbf{n}_{\ell}\bigr)\\
&=\sum_{j=1}^k\sum_{i,\ell=j}^k\bigl[\mathbf{m}_i\cdot \overline{C}^T\Lambda\bigl(\overline{C}\mathbf{m}_{\ell}-2\mathbf{n}_{\ell}\bigr)+\mathbf{n}_i\cdot\Lambda\mathbf{n}_{\ell}\bigr]\\
&=\sum_{j=1}^k\sum_{i,\ell=j}^k\bigl[\delta\mathbf{m}_i\cdot T^{\vee}\bigl(\overline{C}\mathbf{m}_{\ell}-2\mathbf{n}_{\ell}\bigr) +\mathbf{n}_i\cdot\Lambda\mathbf{n}_{\ell}\bigr],
\end{align*}
where the last statement follows from the fact that $\Lambda\overline{C}=\delta T^{\vee}$.
\end{proof}

Let us fix the quantum parameter to be $q=\nu^{\delta}$, and let \smash{$u=\nu^{1/2}=q^{\frac{1}{2\delta}}$} and $\mathbb{Z}_u= \mathbb{Z}\bigl[u,u^{-1}\bigr]$. For any ring $R$ and a set of variables $x=\{x_1,\dots,x_n\}$, we let $R((x))$ denote the ring of formal Laurent series in the variables $x_1,\dots,x_n$, with coefficients in $R$. We define the generating function for multiplicities
\smash{$
Z_{\overline{\lambda},\mathbf{n}}^{(k)}\bigl(\widehat{\mathbf{Q}}_0 ,\widehat{\mathbf{Q}}_1)\in\mathbb{Z}_u\bigl[\widehat{\mathbf{Q}}_0^{\pm1}\bigr]\bigl(\bigl(\widehat{\mathbf{Q}}_1^{-1}\bigr)\bigr)
$}
in the fundamental initial data ${\widehat{\mathbf{Q}}_0\cup \widehat{\mathbf{Q}}_1}$ for the quantum $Q$-system as follows:
\begin{equation}\label{eq:4.13}
Z_{\overline{\lambda},\mathbf{n}}^{(k)}\bigl(\widehat{\mathbf{ Q}}_0,\widehat{\mathbf{Q}}_1\bigr)
=\sum_{\mathbf{m}}q^{\overline{Q}^{(k)}(\mathbf{m},\mathbf{ n})} \prod_{a\in I_r}\prod_{i=1}^k\begin{bmatrix}m_{ a,i}+q_{a,i}\\m_{a,i}\end{bmatrix}_{q_a}\prod_{a\in I_r}\widehat{Q}_{a,1}^{-q_{a,0}} \prod_{a\in I_r}\widehat{Q}_{a,0}^{q_{a,1}}.
\end{equation}
Here, the sum is over all $(r\times k)$-tuples $\mathbf{m}=(m_{a,i})_{a\in I_r,i\in[1,k]}$, and the modified $k$-restricted quadratic form \smash{$\overline{Q}^{(k)}(\mathbf{m},\mathbf{n})$} is defined by
\begin{equation*}
\overline{Q}^{(k)}(\mathbf{m},\mathbf{n})
=\frac{1}{2\delta}\left[\mathbf{q}_1\cdot\Lambda\mathbf{q}_1+ \sum_{j=2}^k(\mathbf{q}_{j-1}-\mathbf{q}_j)\cdot\Lambda (\mathbf{q}_{j-1}-\mathbf{q}_j)\right].
\end{equation*}
Thus, when $\mathbf{q}_0=\mathbf{0}$, we have
\begin{equation}\label{eq:4.15}
\overline{Q}^{(k)}(\mathbf{m},\mathbf{n})= Q^{(k)}(\mathbf{m},\mathbf{n})+\frac{1}{2\delta} L_k(\mathbf{n}),
\end{equation}
where
\begin{equation*}
L_k(\mathbf{n})
=\sum_{j=1}^k\left(\sum_{i=j}^k\mathbf{n}_i\right)\cdot\Lambda\left(\sum_{i=j}^k\mathbf{n}_i\right)
=\sum_{\ell=1}^k\sum_{i,j=\ell}^k\mathbf{n}_i\cdot\Lambda\mathbf{n}_j
=\sum_{i,j=1}^k\min(i,j)\mathbf{n}_i\cdot\Lambda\mathbf{n}_j.
\end{equation*}
The generating function \smash{$Z_{\overline{\lambda},\mathbf{n}}^{ (k)}\bigl(\widehat{\mathbf{Q}}_0,\widehat{\mathbf{Q}}_1\bigr)$} is related to the $k$-restricted $\tilde{M}$-sum \smash{$\tilde{M}_{ \overline{\lambda},\mathbf{n}}^{(k)}\bigl(q^{-1}\bigr)$} \eqref{eq:4.11} via a constant term and an evaluation. For any \smash{$f\in \mathbb{Z}_u\bigl[\widehat{\mathbf{Q}}_0^{\pm1}\bigr]\bigl(\bigl(\widehat{ \mathbf{Q}}_1^{-1}\bigr)\bigr)$}, we denote the constant term of $f$ by \smash{$\mathrm{CT}_{\widehat{\mathbf{Q}}_1}(f)$}. In particular, if
\begin{equation*}
f=\sum_{\mathbf{r},\mathbf{s}\in\mathbb{Z}^r}f_{\mathbf{r},\mathbf{s}}\prod_{a\in I_r}\widehat{Q}_{a,1}^{r_a}\prod_{b\in I_r}\widehat{Q}_{b,0}^{s_b}
\end{equation*}
with $f_{\mathbf{r},\mathbf{s}}\in\mathbb{Z}_u$ for all $\mathbf{r},\mathbf{s}\in\mathbb{Z}^r$, then we have
\begin{equation*}
\mathrm{CT}_{\widehat{\mathbf{Q}}_1}(f)=\sum_{\mathbf{r}\in\mathbb{Z}^r}f_{\mathbf{0},\mathbf{s}}\prod_{b\in I_r}\widehat{Q}_{b,0}^{s_b}.
\end{equation*}
Similarly, we define the multiple evaluation of $f$ at $\widehat{Q}_{1,0},\dots,\widehat{Q}_{r,0}=1$ to be the following formal series with coefficients in $\mathbb{Z}_u$:
\begin{equation*}
f|_{\widehat{\mathbf{Q}}_0=1}=\sum_{\mathbf{r},\mathbf{s}\in\mathbb{Z}^r}f_{\mathbf{r},\mathbf{s}}\prod_{a\in I_r}\widehat{Q}_{a,1}^{r_a}.
\end{equation*}
We note that this is a ``right evaluation". The constant term and evaluation maps commute, and their composition gives
\begin{equation*}
\mathrm{CT}_{\widehat{\mathbf{Q}}_1}(f)|_{\widehat{\mathbf{Q}}_0=1}=\sum_{\mathbf{r}\in\mathbb{Z}^r}f_{\mathbf{r},\mathbf{0}}.
\end{equation*}

\begin{remark*}
As $\widehat{Q}_{a,0}$ and $\widehat{Q}_{b,1}$ $q$-commute for all $a,b\in I_r$, we could also define the notion of a~``left evaluation" analogously. The two different ways of evaluation would still lead to the same result, since all of the variables involved are on the same quantum torus.
\end{remark*}

To simplify our notation, we will define $\phi\colon \mathbb{Z}_u\bigl[\widehat{\mathbf{Q}}_0^{\pm1}\bigr]\bigl(\bigl(\widehat{\mathbf{Q}}_1^{-1}\bigr)\bigr)\to\mathbb{Z}_u$ by
\begin{equation*}
\phi(f)
=\mathrm{CT}_{\widehat{\mathbf{Q}}_1}(f)|_{\widehat{\mathbf{Q}}_0=1}
=\sum_{\mathbf{r}\in\mathbb{Z}^r}f_{\mathbf{r},\mathbf{0}}
\end{equation*}
for all $f\in\mathbb{Z}_u\bigl[\widehat{\mathbf{Q}}_0^{\pm1}\bigr]\bigl(\bigl(\widehat{\mathbf{Q}}_1^{-1}\bigr)\bigr)$.

By \eqref{eq:4.15}, we may express the $k$-restricted $\tilde{M}$-sum \smash{$\tilde{M}_{\overline{\lambda},\mathbf{n}}^{(k)}\bigl(q^{-1}\bigr)$} in terms of \smash{$Z_{\overline{\lambda},\mathbf{n}}^{(k)}\bigl(\widehat{\mathbf{Q}}_0,\widehat{\mathbf{Q}}_1\bigr)$} as follows
\begin{equation}
\tilde{M}_{\overline{\lambda},\mathbf{n}}^{(k)}\bigl(q^{-1}\bigr)
=q^{-\frac{1}{2\delta}L_k(\mathbf{n})}\phi\bigl(Z_{\overline{ \lambda},\mathbf{n}}^{(k)}\bigl(\widehat{\mathbf{Q}}_0,\widehat{ \mathbf{Q}}_1\bigr)\bigr).\label{eq:4.16}
\end{equation}

\subsection{Factorization properties of the quantum generating functions}\label{Section 4.3}

In this subsection, we will express the quantum generating function \smash{$Z_{\overline{\lambda},\mathbf{n}}^{(k)}\bigl(\widehat{\mathbf{Q}}_0,\widehat{\mathbf{Q}}_1\bigr)$} as a~product of the quantum twisted $Q$-system variables (and their inverses). In order to arrive at such an expression, we would need to state a few technical lemmas.

\begin{Lemma}\label{4.5}
Let \smash{$\widehat{Z}_{a,k}=\widehat{Q}_{a,k}\widehat{Q}_{a,k+1}^{-1}$} for all $a\in I_r$ and $k\in\mathbb{Z}$. Then for any distinct root indices $a,b\in I_r$ and $i\in\mathbb{Z}$, we have
\begin{enumerate}\itemsep=0pt
\item[$(1)$] $\widehat{Q}_{b,i}\widehat{Y}_{a,i+1}=\widehat{Y}_{a,i+1}\widehat{Q}_{b,i}$ and $\widehat{Q}_{a,i}\widehat{Y}_{a,i+1}=q_a\widehat{Y}_{a,i+1}\widehat{Q}_{a,i}$,
\item[$(2)$] $\widehat{Q}_{a,i}\widehat{Q}_{b,i+2}=\nu^{-2\Lambda_{ab}}\widehat{Q}_{b,i+2}\widehat{Q}_{a,i}$,
\item[$(3)$] $\widehat{Q}_{b,i+2}\bigl(\widehat{Z}_{a,i}^{-1}\widehat{Z}_{a,i+1}\bigr)=\bigl(\widehat{Z}_{a,i}^{-1}\widehat{Z}_{a,i+1}\bigr)\widehat{Q}_{b,i+2}$,
\item[$(4)$] $\widehat{Z}_{b,i}\widehat{Z}_{a,i}=\widehat{Z}_{a,i}\widehat{Z}_{b,i}$ and $\widehat{Z}_{b,i}\widehat{Z}_{a,i-1}=\widehat{Z}_{a,i-1}\widehat{Z}_{b,i}$, and
\item[$(5)$] $\widehat{Z}_{a,i}\widehat{Y}_{a,i+1}=q_a\widehat{Y}_{a,i+1}\widehat{Z}_{a,i}$.
\end{enumerate}
\end{Lemma}

\begin{proof}
We only need to prove (1), since (2) follows from (1) and Lemma \ref{4.1}, while (3), (4) and~(5) follows from (2) and the commutation relations in Lemma \ref{4.1}. We have
\begin{align*}
\widehat{Q}_{b,i}\widehat{Y}_{a,i+1}
&=\widehat{Q}_{b,i}\prod_{d\in I_r}\widehat{Q}_{d,i+1}^{-\overline{C}_{da}}=\nu^{\sum_{d\in I_r}\Lambda_{bd}\overline{C}_{da}}\biggl(\prod_{d\in I_r}\widehat{Q}_{d,i+1}^{-\overline{C}_{da}}\biggr)\widehat{Q}_{b,i}\\
&=\nu^{\delta T_{ba}^{\vee}}\widehat{Y}_{a,i+1}\widehat{Q}_{b,i}=\widehat{Y}_{a,i+1}\widehat{Q}_{b,i},
\end{align*}
where the last equality follows from the fact that $\Lambda\overline{C}=\delta T^{\vee}$. A similar argument as above would also show that $\widehat{Q}_{a,i}\widehat{Y}_{a,i+1}=\nu^{\delta t_a^{\vee}}\widehat{Y}_{a,i+1}\widehat{Q}_{a,i}=q_a\widehat{Y}_{a,i+1}\widehat{Q}_{a,i}$.
\end{proof}

\begin{Lemma}\label{4.6}
For all $a\in I_r$ and $i\in\mathbb{Z}$, we have $\widehat{Z}_{a,i}\bigl(1-\widehat{Y}_{a,i+1}\bigr)^{-1} =\widehat{Z}_{a,i+1}$.
\end{Lemma}

\begin{proof}
Similar to the proof of \cite[Lemma 4.5]{Lin21}, it follows from Lemma \ref{4.1} that we have
\begin{equation*}
1-\widehat{Y}_{a,i+1}
=\nu^{-\Lambda_{aa}}\widehat{Q}_{a, i+2}\widehat{Q}_{a,i}\widehat{Q}_{a,i+1}^{-2}
=\widehat{Z}_{a,i+1}^{-1}\widehat{Z}_{a,i},
\end{equation*}
or equivalently, $\bigl(1-\widehat{Y}_{a,i+1}\bigr)^{-1}= \widehat{Z}_{a,i}^{-1}\widehat{Z}_{a,i+1}$.
\end{proof}

\begin{Lemma}\label{4.7}
For any $a\in I_r$ and $i,b\in\mathbb{Z}$, we have
\begin{equation*}
\sum_{a\geq0}\begin{bmatrix}a+b\\a\end{bmatrix}_{q_a}\widehat{Y}_{a,i+1}^a\widehat{Z}_{a,i}^b=\widehat{Z}_{a,i}^{-1}\widehat{Z}_{a,i+1}^{b+1}.
\end{equation*}
\end{Lemma}

\begin{proof}
Similar to the proofs of \cite[Lemma 4.7]{DFK14} and \cite[Lemma 4.6]{Lin21}, we have
\begin{equation*}
\sum_{a\geq0}\begin{bmatrix}a+b\\a\end{bmatrix}_vx^ay^b=y^{-1}\bigl(y(1-x)^{-1}\bigr)^{b+1},
\end{equation*}
where we consider the right-hand side of the above equation as a formal power series in $x$. In particular, by setting $v=q_a$, $x=\widehat{Y}_{a,i+1}$ and $y=\widehat{Z}_{a,i}$, the desired statement follows from the above equation, along with Lemmas \ref{4.5}\,(5) and \ref{4.6}.
\end{proof}

\begin{Lemma}\label{4.8}
For any $a\in I_r$ and $k,b\in\mathbb{Z}$, we have
\begin{equation*}
\widehat{Z}_{a,i}^{b+1}\widehat{Q}_{a,i}^{-b}
=\nu^{\frac{b(b+1)\Lambda_{aa}}{2}}\widehat{Z}_{a,i}\widehat{Q}_{a,i+1}^{-b}.
\end{equation*}
\end{Lemma}

\begin{proof}
We use the commutation relations in Lemma \ref{4.1} to deduce that
\begin{gather*}
\widehat{Z}_{a,i}^{b+1}\widehat{Q}_{a,i}^{-b}
=\nu^{\frac{b(b+1)\Lambda_{aa}}{2}}\widehat{Q}_{a,i}^{b+1}\widehat{Q}_{a,i+1}^{-b-1}\widehat{Q}_{a,i}^{-b}
=\nu^{\frac{b(b+1)\Lambda_{aa}}{2}}\widehat{Q}_{a,i}\widehat{Q}_{a,i+1}^{-b-1}
=\nu^{\frac{b(b+1)\Lambda_{aa}}{2}}\widehat{Z}_{a,i}\widehat{Q}_{a,i+1}^{-b}.\tag*{\qed}
\end{gather*}\renewcommand{\qed}{}
\end{proof}

Our first goal is to express \smash{$Z_{\overline{\lambda},\mathbf{n}}^{(1)}\bigl(\widehat{\mathbf{Q}}_0,\widehat{\mathbf{Q}}_1\bigr)$} as a product of the quantum twisted $Q$-system variables, by explicitly summing over $m_{a,1}$ for all $a\in I_r$. In the case $k=1$, it follows from equation \eqref{eq:4.8} that we have $q_{a,0}=\ell_a+\sum_{b\in I_r}\overline{C}_{ab}m_{b,1}-n_{a,1}$ and $q_{a,1}=\ell_a$ for all $a\in I_r$. Using the commutation relations in Lemmas \ref{4.1} and \ref{4.5}\,(1), we have
\begin{align*}
\prod_{a\in I_r}\widehat{Q}_{a,1}^{-q_{a,0}}\prod_{a\in I_r}\widehat{Q}_{a,0}^{q_{a,1}}
&=\prod_{a\in I_r}\widehat{Q}_{a,1}^{n_{a,1}}\prod_{a\in I_r}\widehat{Y}_{a,1}^{m_{a,1}}\prod_{a\in I_r}\widehat{Q}_{a,1}^{-\ell_a}\prod_{a\in I_r}\widehat{Q}_{a,0}^{\ell_a}\\
&=\nu^{-\sum_{a\leq b}\Lambda_{ab}\ell_a\ell_b}\prod_{a\in I_r}\widehat{Q}_{a,1}^{n_{a,1}}\prod_{a\in I_r}\widehat{Y}_{a,1}^{m_{a,1}}\prod_{a\in I_r}\widehat{Q}_{a,0}^{\ell_a}\widehat{Q}_{a,1}^{-\ell_a}\\
&=\nu^{-\frac{1}{2}\sum_{a\neq b}\Lambda_{ab}\ell_a\ell_b-\frac{1}{2}\sum_{a\in I_r}\Lambda_{aa}(2\ell_a^2-\ell_a(\ell_a-1))}\prod_{a\in I_r}\widehat{Q}_{a,1}^{n_{a,1}}\prod_{a\in I_r}\widehat{Y}_{a,1}^{m_{a,1}}\prod_{a\in I_r}\widehat{Z}_{a,0}^{\ell_a}\\
&=\nu^{-\frac{1}{2}\sum_{a,b\in I_r}\Lambda_{ab}\ell_a\ell_b-\frac{1}{2}\sum_{a\in I_r}\Lambda_{aa}\ell_a}\prod_{a\in I_r}\widehat{Q}_{a,1}^{n_{a,1}}\prod_{a\in I_r}\widehat{Y}_{a,1}^{m_{a,1}}\widehat{Z}_{a,0}^{\ell_a}\\
&=q^{-\frac{1}{2\delta}\mathbf{q}_1\cdot\Lambda\mathbf{q}_1-\frac{1}{2\delta}\sum_{a\in I_r}\Lambda_{aa}\ell_a}\prod_{a\in I_r}\widehat{Q}_{a,1}^{n_{a,1}}\prod_{a\in I_r}\widehat{Y}_{a,1}^{m_{a,1}}\widehat{Z}_{a,0}^{\ell_a}.
\end{align*}
As we have \smash{$\overline{Q}^{(1)}(\mathbf{m},\mathbf{n})=\frac{1}{2\delta}\mathbf{q}_1\cdot\Lambda\mathbf{q}_1$}, it follows from the previous equation that we have
\begin{align}
Z_{\overline{\lambda},\mathbf{n}}^{(1)}\bigl(\widehat{\mathbf{Q}}_0,\widehat{\mathbf{Q}}_1\bigr)
&=\sum_{\mathbf{m}}q^{\overline{Q}^{(1)}(\mathbf{m},\mathbf{n})}\prod_{a\in I_r}\begin{bmatrix}m_{a,1}+\ell_a\\m_{a,1}\end{bmatrix}_{q_a}\prod_{a\in I_r}\widehat{Q}_{a,1}^{-q_{a,0}}\prod_{a\in I_r}\widehat{Q}_{a,0}^{\ell_a}\nonumber\\
&=q^{-\frac{1}{2\delta}\sum_{a\in I_r}\Lambda_{aa}\ell_a}\prod_{a\in I_r}\widehat{Q}_{a,1}^{n_{a,1}}\prod_{a\in I_r}\sum_{m_{a,1}\geq0}\begin{bmatrix}m_{a,1}+\ell_a\\m_{a,1}\end{bmatrix}_{q_a}\widehat{Y}_{a,1}^{m_{a,1}}\widehat{Z}_{a,0}^{\ell_a}\nonumber\\
&=q^{-\frac{1}{2\delta}\sum_{a\in I_r}\Lambda_{aa}\ell_a}\prod_{a\in I_r}\widehat{Q}_{a,1}^{n_{a,1}}\prod_{a\in I_r}\widehat{Z}_{a,0}^{-1}\widehat{Z}_{a,1}^{\ell_a+1}\nonumber\\
&=q^{-\frac{1}{2\delta}\sum_{a\in I_r}\Lambda_{aa}\ell_a}\prod_{a\in I_r}\widehat{Q}_{a,1}^{n_{a,1}}\prod_{a\in I_r}\widehat{Z}_{a,0}^{-1}\prod_{a\in I_r}\widehat{Z}_{a,1}^{\ell_a+1}\nonumber\\
&=q^{-\frac{1}{2\delta}\sum_{a\in I_r}\Lambda_{aa}\ell_a-\frac{1}{\delta}\sum_{a,b\in I_r}\Lambda_{ab}n_{a,1}}\prod_{a\in I_r}\widehat{Z}_{a,0}^{-1}\prod_{a\in I_r}\widehat{Q}_{a,1}^{n_{a,1}}\prod_{a\in I_r}\widehat{Z}_{a,1}^{\ell_a+1},\label{eq:4.17}
\end{align}
where we have used the commutation relations in Lemma \ref{4.1}, as well as Lemmas \ref{4.5}\,(4) and~\ref{4.7}, to simplify the equalities.

Our next goal is to write \smash{$Z_{\overline{\lambda},\mathbf{n}}^{(k)}\bigl(\widehat{\mathbf{Q}}_0,\widehat{\mathbf{Q}}_1\bigr)$} as a product of the quantum twisted $Q$-system variab\-les, as well as the generating function \smash{$Z_{\overline{\lambda},\mathbf{n}'}^{(k-1)}\bigl(\widehat{\mathbf{Q}}_1,\widehat{\mathbf{Q}}_2\bigr)$}, where $\mathbf{n}'=(n_{a,i}')_{a\in I_r,i\in\mathbb{N}}$ is \mbox{defined} by~${n_{a,i}'=n_{a,i+1}}$ for all $a\in I_r$ and $i\in\mathbb{N}$. We first observe from equation \eqref{eq:4.8} that we have
\begin{equation*}
q_{a,0}-2q_{a,1}+q_{a,2}=\sum_{b\in I_r}\overline{C}_{ab}m_{b,1}-n_{a,1}
\end{equation*}
for all $a\in I_r$. Using the commutation relations in Lemmas~\ref{4.1} and~\ref{4.5}\,(1), we have
\begin{equation}\label{eq:4.18}
\prod_{a\in I_r}\widehat{Q}_{a,1}^{-q_{a,0}}
=\prod_{a\in I_r}\widehat{Q}_{a,1}^{n_{a,1}}\prod_{a\in I_r}\widehat{Y}_{a,1}^{m_{a,1}}\prod_{a\in I_r}\widehat{Q}_{a,1}^{-2q_{a,1}+q_{a,2}},
\end{equation}
and
\begin{gather}
\prod_{a\in I_r}\widehat{Q}_{a,1}^{-2q_{a,1}+q_{a,2}}\prod_{a\in I_r}\widehat{Q}_{a,0}^{q_{a,1}}\nonumber\\
\qquad=\nu^{-2\sum_{a\leq b}\Lambda_{ab}q_{a,1}q_{b,1}+\sum_{a,b\in I_r}\Lambda_{ab}q_{a,1}q_{b,2}}\prod_{a\in I_r}\widehat{Q}_{a,0}^{q_{a,1}}\widehat{Q}_{a,1}^{-2q_{a,1}}\prod_{a\in I_r}\widehat{Q}_{a,1}^{q_{a,2}}\nonumber\\
\qquad=\nu^{-\sum_{a,b\in I_r}\Lambda_{ab}q_{a,1}(q_{b,1}-q_{b,2})-\sum_{a\in I_r}\Lambda_{aa}q_{a,1}^2+\frac{1}{2}\sum_{a\in I_r}\Lambda_{aa}q_{a,1}(q_{a,1}-1)}\prod_{a\in I_r}\widehat{Z}_{a,0}^{q_{a,1}}\widehat{Q}_{a,1}^{-q_{a,1}}\prod_{a\in I_r}\widehat{Q}_{a,1}^{q_{a,2}}\nonumber\\
\qquad=q^{-\frac{1}{\delta}\mathbf{q}_1\cdot\Lambda(\mathbf{q}_1-\mathbf{q}_2)-\frac{1}{2\delta}\sum_{a\in I_r}\Lambda_{aa}q_{a,1}(q_{a,1}+1)}\prod_{a\in I_r}\widehat{Z}_{a,0}^{q_{a,1}}\widehat{Q}_{a,1}^{-q_{a,1}}\prod_{a\in I_r}\widehat{Q}_{a,1}^{q_{a,2}}.\label{eq:4.19}
\end{gather}
As \smash{$\overline{Q}^{(k)}(\mathbf{m},\mathbf{n})-\overline{Q}^{(k-1)}(\mathbf{m}',\mathbf{n}')=\frac{1}{\delta}\mathbf{q}_1\cdot\Lambda(\mathbf{q}_1-\mathbf{q}_2)$} (where $\mathbf{m}'=(m_{a,i}')_{a\in I_r,i\in[1,k-1]}$ is defined by~${m_{a,i}'= m_{a,i+1}}$ for all $a\in I_r$ and $i\in[1,k-1]$), it follows from equations \eqref{eq:4.18} and \eqref{eq:4.19} that we have
\begin{align}
Z_{\overline{\lambda},\mathbf{n}}^{(k)}\bigl(\widehat{\mathbf{Q}}_0,\widehat{\mathbf{Q}}_1\bigr)
={}&\sum_{\mathbf{m}}q^{\overline{Q}^{(k)}(\mathbf{m},\mathbf{n})}\prod_{a\in I_r}\prod_{i=1}^k\begin{bmatrix}m_{a,i}+q_{a,i}\\m_{a,i}\end{bmatrix}_{q_a}\prod_{a\in I_r}\widehat{Q}_{a,1}^{-q_{a,0}}\prod_{a\in I_r}\widehat{Q}_{a,0}^{q_{a,1}}\nonumber\\
={}&\prod_{a\in I_r}\widehat{Q}_{a,1}^{n_{a,1}}\sum_{\mathbf{m}'\geq\mathbf{0}}q^{\overline{Q}^{(k-1)}(\mathbf{m}',\mathbf{n}')-\frac{1}{2\delta}\sum_{a\in I_r}\Lambda_{aa}q_{a,1}(q_{a,1}+1)}\prod_{a\in I_r}\prod_{i=2}^k\begin{bmatrix}m_{a,i}+q_{a,i}\\m_{a,i}\end{bmatrix}_{q_a}\nonumber\\
&\times\left(\prod_{a\in I_r}\sum_{m_{a,1}\geq0}\begin{bmatrix}m_{a,1}+q_{a,1}\\m_{a,1}\end{bmatrix}_{q_a}\widehat{Y}_{a,1}^{m_{a,1}}\widehat{Z}_{a,0}^{q_{a,1}}\widehat{Q}_{a,1}^{-q_{a,1}}\right)\prod_{a\in I_r}\widehat{Q}_{a,1}^{q_{a,2}},\label{eq:4.20}
\end{align}
where we used Lemma \ref{4.5}\,(1) in the last equality. As \smash{$\overline{Q}^{(k-1)}(\mathbf{m}',\mathbf{n}')-\frac{1}{2\delta}\sum_{a\in I_r}\Lambda_{aa}q_{a,1}(q_{a,1}+1)$} is independent of $\mathbf{m}_1$, we may sum over $m_{a,1}$ for all $a\in I_r$, and simplify the internal product of sums
\begin{equation*}
\prod_{a\in I_r}\sum_{m_{a,1}\geq0}\begin{bmatrix}m_{a,1}+q_{a,1}\\m_{a,1}\end{bmatrix}_{q_a}\widehat{Y}_{a,1}^{m_{a,1}}\widehat{Z}_{a,0}^{q_{a,1}}\widehat{Q}_{a,1}^{-q_{a,1}}.
\end{equation*}
Now, it follows from Lemmas \ref{4.5}\,(3) and (4), \ref{4.7} and \ref{4.8} that we have
\begin{gather}
\prod_{a\in I_r}\sum_{m_{a,1}\geq0}\begin{bmatrix}m_{a,1}+q_{a,1}\\m_{a,1}\end{bmatrix}_{q_a}\widehat{Y}_{a,1}^{m_{a,1}}\widehat{Z}_{a,0}^{q_{a,1}}\widehat{Q}_{a,1}^{-q_{a,1}}\nonumber\\
\qquad=q^{\frac{1}{2\delta}\sum_{a\in I_r}\Lambda_{aa}q_{a,1}(q_{a,1}+1)}\prod_{a\in I_r}\widehat{Z}_{a,0}^{-1}\widehat{Z}_{a,1}\widehat{Q}_{a,2}^{-q_{a,1}}\nonumber\\
\qquad=q^{\frac{1}{2\delta}\sum_{a\in I_r}\Lambda_{aa}q_{a,1}(q_{a,1}+1)}\prod_{a\in I_r}\widehat{Z}_{a,0}^{-1}\widehat{Z}_{a,1}\prod_{a\in I_r}\widehat{Q}_{a,2}^{-q_{a,1}}.\label{eq:4.21}
\end{gather}
By combining equations \eqref{eq:4.20} and \eqref{eq:4.21}, we have
\begin{align}
Z_{\overline{\lambda},\mathbf{n}}^{(k)}\bigl(\widehat{\mathbf{Q}}_0,\widehat{\mathbf{Q}}_1\bigr)
={}&\prod_{a\in I_r}\widehat{Q}_{a,1}^{n_{a,1}}\sum_{\mathbf{m}'\geq\mathbf{0}}q^{\overline{Q}^{(k-1)}(\mathbf{m}',\mathbf{n}')-\frac{1}{2\delta}\sum_{a\in I_r}\Lambda_{aa}q_{a,1}(q_{a,1}+1)}\prod_{a\in I_r}\prod_{i=2}^k\begin{bmatrix}m_{a,i}+q_{a,i}\\m_{a,i}\end{bmatrix}_{q_a}\nonumber\\
& \times\left(\prod_{a\in I_r}\sum_{m_{a,1}\geq0}\begin{bmatrix}m_{a,1}+q_{a,1}\\m_{a,1}\end{bmatrix}_{q_a}\widehat{Y}_{a,1}^{m_{a,1}}\widehat{Z}_{a,0}^{q_{a,1}}\widehat{Q}_{a,1}^{-q_{a,1}}\right)\prod_{a\in I_r}\widehat{Q}_{a,1}^{q_{a,2}}\nonumber\\
={}&\prod_{a\in I_r}\widehat{Q}_{a,1}^{n_{a,1}}\sum_{\mathbf{m}'\geq\mathbf{0}}q^{\overline{Q}^{(k-1)}(\mathbf{m}',\mathbf{n}')}\prod_{a\in I_r}\prod_{i=2}^k\begin{bmatrix}m_{a,i}+q_{a,i}\\m_{a,i}\end{bmatrix}_{q_a}\nonumber\\
& \times\prod_{a\in I_r}\widehat{Z}_{a,0}^{-1}\widehat{Z}_{a,1}\prod_{a\in I_r}\widehat{Q}_{a,2}^{-q_{a,1}}\prod_{a\in I_r}\widehat{Q}_{a,1}^{q_{a,2}}\nonumber\\
={}&q^{-\frac{1}{\delta}\sum_{a,b\in I_r}\Lambda_{ab}n_{a,1}}\prod_{a\in I_r}\widehat{Z}_{a,0}^{-1}\prod_{a\in I_r}\widehat{Q}_{a,1}^{n_{a,1}}\prod_{a\in I_r}\widehat{Z}_{a,1}\nonumber\\
& \times\sum_{\mathbf{m}'\geq\mathbf{0}}q^{\overline{Q}^{(k-1)}(\mathbf{m}',\mathbf{n}')}\prod_{a\in I_r}\prod_{i=2}^k\begin{bmatrix}m_{a,i}+q_{a,i}\\m_{a,i}\end{bmatrix}_{q_a}\prod_{a\in I_r}\widehat{Q}_{a,2}^{-q_{a,1}}\prod_{a\in I_r}\widehat{Q}_{a,1}^{q_{a,2}}\nonumber\\
={}&q^{-\frac{1}{\delta}\sum_{a,b\in I_r}\Lambda_{ab}n_{a,1}}\prod_{a\in I_r}\widehat{Z}_{a,0}^{-1}\prod_{a\in I_r}\widehat{Q}_{a,1}^{n_{a,1}}\prod_{a\in I_r}\widehat{Z}_{a,1}Z_{\overline{\lambda},\mathbf{n}'}^{(k-1)}\bigl(\widehat{\mathbf{Q}}_1,\widehat{\mathbf{Q}}_2\bigr)\label{eq:4.22},
\end{align}
where we have used the commutation relations in Lemma \ref{4.1} in the second last equality.

By invoking Lemma \ref{4.3}, along with equations \eqref{eq:4.17} and \eqref{eq:4.22}, we get
\begin{align}
Z_{\overline{\lambda},\mathbf{n}}^{(k)}\bigl(\widehat{\mathbf{Q}}_0,\widehat{\mathbf{Q}}_1\bigr)
={}&q^{-\frac{1}{2\delta}\sum_{a\in I_r}\Lambda_{aa}\ell_a-\frac{1}{\delta}\sum_{a,b\in I_r}\sum_{i=1}^k\Lambda_{ab}n_{a,i}}\nonumber\\
&\times\prod_{a\in I_r}\widehat{Z}_{a,0}^{-1}\left(\prod_{i=1}^k\prod_{a\in I_r}\widehat{Q}_{a,i}^{n_{a,i}}\right)\prod_{a\in I_r}\widehat{Z}_{a,k+1}^{\ell_a+1}.\label{eq:4.23}
\end{align}
As an immediate consequence of equation \eqref{eq:4.23}, we deduce that
\begin{equation}\label{eq:4.24}
Z_{\overline{\lambda},\mathbf{n}}^{(k)}\bigl(\widehat{\mathbf{Q}}_0,\widehat{\mathbf{Q}}_1\bigr)
=Z_{0,\overline{\mathbf{n}}^{(j)}}^{(j)}\bigl(\widehat{\mathbf{Q}}_0,\widehat{\mathbf{Q}}_1\bigr)Z_{\overline{\lambda},\mathbf{n}^{(k-j)}}^{(k-j)}
\bigl(\widehat{\mathbf{Q}}_j,\widehat{\mathbf{Q}}_{j+1}\bigr)
\end{equation}
for all $j\in[1,k]$, where $\overline{\mathbf{n}}^{(j)}$ is the $(r\times j)$-tuple of non-negative integers \smash{$\overline{ \mathbf{n}}^{(j)}=(n_{a,i})_{a\in I_r,i\in[1,j]}$}, and \smash{$\mathbf{n}^{(j)} \!=\! \bigl(n_{a,i}^{(j)}\bigr)_{a\in I_r,i\in\mathbb{N}}$} is defined by \smash{$n_{a,i}^{(j)}=n_{a,i+j}$} for all $a\in I_r$ and $i\in\mathbb{N}$.

\subsection[Proof of Theorems 1.1 and 1.2]{Proof of Theorems \ref{1.1} and \ref{1.2}}\label{Section 4.4}

We are now ready to show that the two sums \smash{$M_{\overline{\lambda},\mathbf{n}}^{(k)}\bigl(q^{-1}\bigr)$} and \smash{$\tilde{M}_{\overline{\lambda},\mathbf{n}}^{(k)}\bigl(q^{-1}\bigr)$} are equal to each other. To this end, we need a few auxiliary lemmas from \cite[Section 5.5]{DFK14}. We will let $A$ denote the ring~${\mathbb{Z}_{\nu}[\widehat{\mathbf{Q}}_1,\widehat{\mathbf{Q}}_{-1},\widehat{\mathbf{Q}}_0^{\pm1}]}$, and $A_a$ denote the ring $\mathbb{Z}_{\nu}\bigl[\bigl\{\widehat{Q}_{b,\pm1}\bigr\}_{b\neq a},\widehat{\mathbf{Q}}_0^{\pm1}\bigr]$.

\begin{Lemma}[{\cite[Lemma 5.9]{DFK14}}]\label{4.9}
Let $a_1,\dots,a_n\in I_r$, $i_1,\dots,i_n\in\mathbb{Z}$ and $m_1,\dots,m_n\in\mathbb{Z}_+$. Then \smash{$\prod_{j=1}^n\widehat{Q}_{a_j,i_j}^{m_j}\in A$}.
\end{Lemma}

The proof of Lemma \ref{4.9} follows from that of \cite[Lemma 5.9]{DFK14}, bearing in mind that $\widehat{Q}_{a,1}$ and~\smash{$\widehat{Q}_{b,-1}$} $\nu$-commute for all distinct $a,b\in I_r$ by Lemma \ref{4.5}\,(1).

\begin{Lemma}[{\cite[Lemma 5.12]{DFK14}}]\label{4.10}
For any $b\in I_r$, we have
\begin{equation*}
\biggl(\prod_{a\in I_r}\widehat{Z}_{a,0}^{-1}\biggr)\widehat{Q}_{b,-1}\biggl|_{\widehat{\mathbf{Q}}_0=1}=0.
\end{equation*}
\end{Lemma}

\begin{Lemma}[{\cite[Lemma 5.14]{DFK14}}]\label{4.11}
For all $a\in I_r$ and $n\in\mathbb{Z}_+$, we have \smash{$\widehat{Q}_{a,n}^{-1}\in A_a\bigl(\bigl(\widehat{Q}_{a,1}^{-1}\bigr)\bigr)$}.
\end{Lemma}

\begin{proof}[Proof of Theorem~\ref{1.1}]
By equation \eqref{eq:4.16}, it suffices to show that if a term $S$ on the right-hand side of~\eqref{eq:4.13} corresponds to a vector $\mathbf{m}$ that satisfies $q_{a,i}<0$ for some $a\in I_r$ and $i\in[1,k]$, then~${\phi(S)=0}$. To this end, we will prove by induction on $j=k,\dots,1$, and show that if $S$ on the right-hand side of \eqref{eq:4.13} corresponds to a vector $\mathbf{m}$ that satisfies $q_{a,i}<0$ for some $a\in I_r$ and~${i\in[j,k]}$, then $\phi(S)=0$. The base case $j=k$ holds since we have $q_{a,k}=\ell_a\geq0$ for all $a\in I_r$, which implies that there is no such term $S$ that satisfies $q_{a,k}<0$ for some $a\in I_r$. Next, let us assume that the statement holds for $j+1$, where $j\geq1$. By equations \eqref{eq:4.23} and \eqref{eq:4.24}, we have
\begin{align}\label{eq:4.25}
Z_{\overline{\lambda},\mathbf{n}}^{(k)}\bigl(\widehat{\mathbf{Q}}_0,\widehat{\mathbf{Q}}_1\bigr)
={}&q^{-\frac{1}{2\delta}\sum_{a\in I_r}\Lambda_{aa}\ell_a-\frac{1}{\delta}\sum_{a,b\in I_r}\sum_{i=1}^j\Lambda_{ab}n_{a,i}}\prod_{a\in I_r}\widehat{Z}_{a,0}^{-1}\left(\prod_{i=1}^j\prod_{a\in I_r}\widehat{Q}_{a,i}^{n_{a,i}}\right)\prod_{a\in I_r}\widehat{Z}_{a,j+1}\nonumber\\
&\times\sum_{\mathbf{m}^{(j)}\geq0}q^{\overline{Q}^{(k-j)}(\mathbf{m}^{(j)},\mathbf{n}^{(j)})}\nonumber\\
&\times\prod_{a\in I_r}\prod_{i=j+1}^k\begin{bmatrix}m_{a,i}+q_{a,i}\\m_{a,i}\end{bmatrix}_{q_a}\prod_{a\in I_r}\widehat{Q}_{a,j+1}^{-q_{a,j}}\prod_{a\in I_r}\widehat{Q}_{a,j}^{q_{a,j+1}}.
\end{align}
By induction hypothesis, we have $q_{a,i}\geq0$ for all $a\in I_r$ and $i\geq j+1$. We note that a term in the right-hand side of \eqref{eq:4.25} has the form
\begin{equation*}
S
=\prod_{a\in I_r}\widehat{Z}_{a,0}^{-1}\left(\prod_{i=1}^j\prod_{a\in I_r}\widehat{Q}_{a,i}^{n_{a,i}}\right)\prod_{a\in I_r}\widehat{Z}_{a,j+1}\prod_{a\in I_r}\widehat{Q}_{a,j+1}^{-q_{a,j}}\prod_{a\in I_r}\widehat{Q}_{a,j}^{q_{a,j+1}}.
\end{equation*}
When $q_{b,j}<0$, it follows from Lemmas \ref{4.9} and \ref{4.11} that $S\in A_b$. By Lemma \ref{4.10}, it follows that
\[
S|_{\widehat{\mathbf{Q}}_0=1}\in\mathbb{Z}_t\bigl[\widehat{\mathbf{Q}}_1\bigr]\bigl[\bigl[\bigl\{\widehat{Q}_{a,1}^{-1}\bigr\}_{a\neq b}\bigr]\bigr].
\]
 Due to the prefactor \smash{$\prod_{a\in I_r}\widehat{Z}_{a,0}^{-1}$}, it follows that the exponent of~$\widehat{Q}_{a,1}$ in all terms of \smash{$S|_{\widehat{\mathbf{Q}}_0=1}$} are positive. Consequently, we have $\phi(S)=0$. So this shows that the sum in equation \eqref{eq:4.13} is unchanged if we impose the restriction $q_{b,j}\geq0$. As $b\in I_r$ is arbitrary, the statement is proved for $j$, and we are done.
\end{proof}

Having proved Theorem~\ref{1.1}, it remains to show that Theorem~\ref{1.2} holds. As in \cite[Section~1]{Kedem11}, Theorem~\ref{1.2} will follow from a series of (in-)equalities involving tensor product multiplicities involving KR-modules over twisted quantum affine and current algebras, as well as fermionic sums, which we will describe in detail as follows.

\begin{proof}[Proof of Theorem~\ref{1.2}]
We first recall that similar to the setting of twisted current algebras, the KR-modules over the twisted quantum affine algebra $U_q(\widehat{\mathfrak{g}}^{\sigma})$ are parameterized by $a\in I_r$, $m\in\mathbb{Z}_+$, and $z\in\mathbb{C}^*$ as well, and we denote them by $W_{a,m}^{\sigma}(z)$. By \cite[Theorem 8.5]{Hernandez10}, the multiplicity of the irreducible $U_q(\mathfrak{g}^{ \sigma})$-module $V^q\bigl(\overline{\lambda}\bigr)$ in the tensor product \smash{$\bigotimes_{\alpha\in I_r,\,i\in\mathbb{Z}_+} (W_{\alpha,i}^{\sigma})^{\otimes n_{\alpha,i}}$} of KR-modules over $U_q(\widehat{\mathfrak{g}}^{\sigma})$ is given by $\tilde{M}_{\overline{\lambda},\mathbf{n}}(1)$, that is, we have
\begin{equation}
\dim\Hom_{U_q(\mathfrak{g}^{\sigma})}\biggl(\bigotimes_{ \alpha\in I_r,\,i\in\mathbb{Z}_+}(W_{\alpha,i}^{\sigma})^{ \otimes n_{\alpha,i}} ,V^q\bigl(\overline{\lambda}\bigr)\biggr)
=\tilde{M}_{\overline{\lambda},\mathbf{n}}(1).\label{eq:4.26}
\end{equation}
Next, we note that the tensor product \smash{$\bigotimes_{\alpha\in I_r,\,i\in\mathbb{Z}_+} (\KR_{\alpha,i}^{\sigma})^{\otimes n_{\alpha,i}}$} of KR-modules over $\mathfrak{g}[t]^{\sigma}$ arises as the $q\to 1$ limit of the tensor product $\bigotimes_{\alpha\in I_r,\, i\in\mathbb{Z}_+} (W_{\alpha,i}^{\sigma})^{\otimes n_{\alpha,i}}$ of KR-modules over $U_q(\widehat{\mathfrak{g}}^{\sigma})$. As both tensor products are defined as quotients by some ideal, and the ideal in the $q\to 1$ limit may be smaller than that for generic values of $q$, it follows from general deformation arguments that we have
\begin{gather}
\dim\Hom_{U_q(\mathfrak{g}^{\sigma})}\biggl(\bigotimes_{ \alpha\in I_r,i\in\mathbb{Z}_+}(W_{\alpha,i}^{\sigma})^{ \otimes n_{\alpha,i}} ,V^q\bigl(\overline{\lambda}\bigr)\biggr)\nonumber\\
\qquad
\leq\dim\Hom_{\mathfrak{g}^{\sigma}}\biggl(\bigotimes_{ \alpha\in I_r,
\, i\in\mathbb{Z}_+}(\KR_{\alpha,i}^{\sigma})^{ \otimes n_{\alpha,i}},V\bigl(\overline{\lambda}\bigr)\biggr).\label{eq:4.27}
\end{gather}
Likewise, by the definition of $\mathcal{F}_{\mathbf{n}}^*$ and by general deformation arguments, we have
\begin{equation}\label{eq:4.28}
\dim\Hom_{\mathfrak{g}^{\sigma}}\left(\bigotimes_{\alpha\in I_r,\,i\in\mathbb{Z}_+}(\KR_{\alpha,i}^{\sigma})^{\otimes n_{\alpha,i}},V\bigl(\overline{\lambda}\bigr)\right)
\leq\dim\Hom_{\mathfrak{g}^{\sigma}}(\mathcal{F}_{ \mathbf{n}}^*,V\bigl(\overline{\lambda}\bigr))=\mathcal{M}_{\overline{\lambda},\mathbf{n}}(1).
\end{equation}
Finally, by Theorem~\ref{3.1}, we have $\mathcal{M}_{ \overline{\lambda},\mathbf{n}}(1)\leq M_{\overline{\lambda}, \mathbf{n}}(1)$. Together with Theorem~\ref{1.1} and \eqref{eq:4.26}--\eqref{eq:4.28}, we have the following twisted analogue of pentagon of inequalities and equalities:
\begin{center}
\begin{tikzpicture}[scale=0.75, transform shape]
\node (A) {$\mathcal{M}_{\overline{\lambda},\mathbf{n}}(1)$};
\node (B) [right=2.5cm of A] {$M_{\overline{\lambda},\mathbf{ n}}(1)$};
\node (C) [right=2.5cm of B] {$\tilde{M}_{\overline{\lambda}, \mathbf{n}}(1)$};
\node (D) [above=2.5cm of C] {$\dim\Hom_{U_q(\mathfrak{g}^{ \sigma})} \bigl(\bigotimes_{\alpha\in I_r,\, i\in\mathbb{Z}_+}(W_{\alpha,i}^{\sigma})^{\otimes n_{\alpha,i}},V^q\bigl(\overline{ \lambda}\bigr)\bigr)$};
\node (E) [left=4cm of D] {$\dim\Hom_{\mathfrak{g}^{\sigma}} \bigl(\bigotimes_{ \alpha\in I_r,\,i\in\mathbb{Z}_+}(\KR_{ \alpha,i}^{\sigma})^{\otimes n_{\alpha,i}},V\bigl(\overline{ \lambda}\bigr)\bigr)$};
\path[sloped, anchor=center] (A) to node {$\leq$} (B);
\path[sloped, anchor=center] (B) to node {$=$} (C);
\path[sloped, anchor=center] (C) to node {$=$} (D);
\path[sloped, anchor=center] (D) to node {$\geq$} (E);
\path[sloped, anchor=center] (E) to node {$\leq$} (A);
\end{tikzpicture}
\end{center}
In particular, equality must hold throughout, and hence we must have \smash{$\mathcal{M}_{\overline{\lambda},\mathbf{n}}(1)
=M_{\overline{\lambda},\mathbf{n}}(1)$}. As each of the coefficients in the $M$-sum $M_{\overline{ \lambda},\mathbf{n}}\bigl(q^{-1}\bigr)$ are manifestly nonnegative as well, we must have \smash{$\mathcal{M}_{\overline{\lambda},\mathbf{n}}\bigl(q^{-1}\bigr)=M_{ \overline{\lambda},\mathbf{n}}\bigl(q^{-1}\bigr)$} as required.
\end{proof}

In particular, Theorem~\ref{1.2} implies that the fusion product $\mathcal{F}_{\mathbf{n}}^*$ of twisted KR-modules is independent of the choice of localization parameters, and the evaluation map $\smash{\overline{\varphi}_{ \bm{\mu}}\colon \tilde{\Gamma}_{\bm{ \mu}}/\tilde{\Gamma}_{\bm{\mu}}' }\to\tilde{\mathcal{H}}[\bm{\mu}]$ defined in Section \ref{Section 3.5} is an isomorphism of graded vector spaces. Moreover, equations \eqref{eq:4.16} and~\eqref{eq:4.23} allow us to express the graded multiplicity $\mathcal{M}_{\overline{\lambda},\mathbf{n}}\bigl(q^{-1}\bigr)$ as a~constant term evaluation of a~product of quantum twisted $Q$-system variables as follows:
\begin{align}
\mathcal{M}_{\overline{\lambda},\mathbf{n}}\bigl(q^{-1}\bigr)
={}&q^{-\frac{1}{2\delta}(L(\mathbf{n})+\sum_{a\in I_r}\Lambda_{aa}\ell_a+2\sum_{a,b\in I_r}\sum_{i\in\mathbb{N}}\Lambda_{ab}n_{a,i})}\nonumber\\
&\times\phi\biggl(\prod_{a\in I_r}\widehat{Z}_{a,0}^{-1}\biggl(\prod_{i\in\mathbb{N}}\prod_{a\in I_r}\widehat{Q}_{a,i} ^{n_{a,i}}\biggr)\prod_{a\in I_r}\widehat{Z}_a^{\ell_a+1}\biggr),\label{eq:4.29}
\end{align}
where
\begin{gather*}
L(\mathbf{n})
=\lim_{k\to\infty}L_k(\mathbf{n})
=\sum_{i,j\in\mathbb{N}}\min(i,j)\mathbf{n}_i\cdot\Lambda\mathbf{n}_j
=\sum_{i,j\in\mathbb{N}}\sum_{a,b\in I_r}\min(i,j)\Lambda_{ab} n_{a,i}n_{b,j},\\
\widehat{Z}_b
=\lim_{k\to\infty}\widehat{Z}_{b,k}.
\end{gather*}
Here, we note that a similar reasoning as in the proof of \cite[Theorem 5.17]{DFK14} shows that $\widehat{Z}_b$ is well-defined as a formal power series of $\widehat{Q}_{b,1}^{-1}$ with coefficients that are Laurent polynomials in the remaining initial data of \smash{$\widehat{\mathbf{Q}}_0\cup\widehat{\mathbf{Q}}_1$}.

\section[An identity satisfied by the graded characters of twisted KR-modules]{An identity satisfied by the graded characters\\ of twisted $\boldsymbol{ \KR}$-modules}\label{Section 5}

In this section, we will prove Theorem~\ref{1.3}, where our approach will follow that taken in the proof of \cite[Section 5]{Lin21}. As our previous calculations only involved nontrivial twisted KR-modules, whereas the identity of graded $\mathfrak{g}^{ \sigma}$-characters of the fusion products of twisted KR-modules stated in Theorem~\ref{1.3} involves trivial twisted KR-modules, we would need a generalized form of equation \eqref{eq:4.29} that includes trivial twisted KR-modules. To arrive at this generalized form, we will need to consider vectors $\widehat{\mathbf{n}}=(n_{a,i})_{a\in I_r,i\in\mathbb{Z}_+}$ that parameterizes a finite set of (possibly trivial) KR-modules over $\mathfrak{g}[t]^{\sigma}$. Let us also make the following definitions:
\begin{gather*}
\widehat{L}(\widehat{\mathbf{n}})
=\sum_{i,j\in\mathbb{Z}_+}\sum_{a,b\in I_r}\min(i,j)\Lambda_{ab} n_{a,i}n_{b,j},\\
\mathcal{M}_{\overline{\lambda},\widehat{\mathbf{n}}}\bigl(q^{-1}\bigr)
=q^{-\frac{1}{2\delta}(\sum_{a\in I_r}\Lambda_{aa}\ell_a+2\widehat{F}(\widehat{\mathbf{n}})+\widehat{L}(\widehat{\mathbf{n}}))}\phi\biggl(\prod_{a\in I_r}\widehat{Z}_{a,0}^{-1}\biggl(\prod_{i\in\mathbb{Z}_+}\prod_{a\in I_r}\widehat{Q}_{a,i}^{n_{a,i}}\biggr)\prod_{a\in I_r}\widehat{Z}_a^{\ell_a+1}\biggr),
\end{gather*}
where
$
\widehat{F}(\widehat{\mathbf{n}})=
\sum_{a,b\in I_r}\sum_{i\in\mathbb{Z}_+}\Lambda_{ab}n_{a,i}$.
We claim that $\mathcal{M}_{\overline{\lambda},\widehat{\mathbf{n}}}\bigl(q^{-1}\bigr)=\mathcal{M}_{\overline{\lambda},\mathbf{n}}\bigl(q^{-1}\bigr)$. Indeed, we may regard
\begin{equation*}
\prod_{a\in I_r}\widehat{Z}_{a,0}^{-1}\biggl(\prod_{i\in\mathbb{Z}_+}\prod_{a\in I_r}\widehat{Q}_{a,i}^{n_{a,i}}\biggr)\prod_{a\in I_r}\widehat{Z}_a^{\ell_a+1}
\end{equation*}
as an element of $\mathbb{Z}_u\bigl[\widehat{\mathbf{Q}}_0^{\pm1}\bigr]\bigl(\bigl(\widehat{\mathbf{Q}}_1^{-1}\bigr)\bigr)$ by \eqref{eq:4.23}. Thus, it follows from the definition of the function~$\phi$ that we have
\begin{align*}
\mathcal{M}_{\overline{\lambda},\mathbf{n}}\bigl(q^{-1}\bigr)
={}&q^{-\frac{1}{2\delta}(\sum_{a\in I_r}\Lambda_{aa}\ell_a+2\sum_{a,b\in I_r}\sum_{i\in\mathbb{N}}\Lambda_{ab}n_{a,i}+L( \mathbf{n}))}\\
&\times\phi\biggl(\prod_{a\in I_r}\widehat{Z}_{ a,0} ^{-1}\biggl(\prod_{i\in\mathbb{N}}\prod_{a\in I_r}\widehat{Q}_{ a,i}^{n_{a,i}}\biggr)\prod_{a\in I_r}\widehat{Z}_a^{\ell_a+1}\biggr)\\
={}&q^{-\frac{1}{2\delta}(\sum_{a\in I_r}\Lambda_{aa}\ell_a+2\sum_{a,b\in I_r}\sum_{i\in\mathbb{N}}\Lambda_{ab}n_{a,i}+\widehat{L}(\widehat{\mathbf{n}}))}\\
&\times\phi\biggl(\prod_{a\in I_r}\widehat{Q}_{a,0}^{n_{a,0}}\prod_{a\in I_r}\widehat{Z}_{ a,0} ^{-1}\biggl(\prod_{i\in\mathbb{N}}\prod_{a\in I_r} \widehat{Q}_{a,i}^{n_{a,i}}\biggr)\prod_{a\in I_r}\widehat{ Z}_a^{\ell_a+1}\biggr)\\
={}&q^{-\frac{1}{2\delta}(\sum_{a\in I_r}\Lambda_{aa}\ell_a+2\sum_{a,b\in I_r}\sum_{i\in\mathbb{Z}_+}\Lambda_{ab}n_{a,i}+\widehat{L}(\widehat{\mathbf{n}}))}\\
&\times\phi\biggl(\prod_{a\in I_r}\widehat{Z}_{a,0}^{-1}\biggl(\prod_{i\in\mathbb{Z}_+}\prod_{a\in I_r} \widehat{Q}_{a,i}^{n_{a,i}}\biggr)\prod_{a\in I_r}\widehat{ Z}_a^{\ell_a+1}\biggr)
=\mathcal{M}_{\overline{\lambda},\widehat{\mathbf{n}}}\bigl(q^{-1}\bigr),
\end{align*}
where the third equality follows from Lemma \ref{4.1}. This allows us to regard \smash{$\mathcal{M}_{\overline{\lambda},\widehat{\mathbf{n}}}\bigl(q^{-1}\bigr)$} as the graded multiplicity of the irreducible component $V\bigl(\overline{\lambda}\bigr)$ in $\mathcal{F}_{\widehat{\mathbf{n}}}^*$, where $\mathcal{F}_{\widehat{\mathbf{n}}}^*$ is the fusion product of twisted KR-modules parameterized by $\widehat{\mathbf{n}}$. More precisely, we have
\begin{equation*}
\mathcal{M}_{\overline{\lambda},\widehat{\mathbf{n}}}(q)
=\sum_{m=0}^{\infty}\dim\Hom_{\mathfrak{g}^{\sigma}}\bigl(\mathcal{ F}_{\widehat{\mathbf{n}}}^*[m],V\bigl(\overline{\lambda}\bigr)\bigr)q^m.
\end{equation*}
Next, let us use Lemma \ref{4.1} to rewrite the quantum $X_m^{(\kappa)}$ $Q$-system relations \eqref{eq:4.5} as
\begin{equation}\label{eq:5.4}
q^{\frac{1}{\delta}\Lambda_{aa}}\widehat{Q}_{a,k-1} \widehat{Q}_{a,k+1}= \widehat{Q}_{a,k}^2\bigl(1-q^{t_a^{\vee}}\widehat{Y}_{a,k}\bigr)
\end{equation}
for all $a\in I_r$ and $k\in\mathbb{Z}$.

Next, we let $\widehat{\mathbf{d}}=(d_{b,i})_{b\in I_r,\,i\in\mathbb{Z}_+}$, $\widehat{\mathbf{s}}=(s_{b,i})_{b\in I_r,\, i\in\mathbb{Z}_+}$, and $\widehat{\mathbf{k}}=(k_{b,i})_{b\in I_r,\,i\in\mathbb{Z}_+}$ be vectors that correspond to the terms
\[
\widehat{Q}_{a,m-1}\widehat{Q}_{a,m+1} , \qquad \widehat{Q}_{a,m}^2, \qquad \text{and}\qquad
\widehat{Q}_{a,m}^2\widehat{Y}_{a,m}=\prod_{b\sim a}\widehat{Q}_{b,m}^{-\overline{C}_{ba}},
\]
 respectively. Specifically, we define
\[
d_{b,i}
=\delta_{ab}(\delta_{i,m-1}+\delta_{i,m+1}), \qquad
s_{b,i}
=2\delta_{ab}\delta_{i,m} , \qquad
k_{b,i}
=-\delta_{b\sim a}\delta_{i,m}\overline{C}_{ba}
\]
for all $b\in I_r$ and~${i\in\mathbb{Z}_+}$, where the function $\delta_{b\sim a}$ is equal to $1$ if $b\sim a$, and $0$ otherwise.

By applying the map $\phi$ to \eqref{eq:5.4}, we have
\begin{gather*}
q^{\frac{1}{\delta}\Lambda_{aa}}\phi\biggl(\biggl(\prod_{b\in I_r}\widehat{Z}_{b,0}^{-1}\biggr)\widehat{Q}_{a,m-1}\widehat{Q}_{a,m+1}\biggl(\prod_{b\in I_r}\widehat{Z}_b^{\ell_b+1}\biggr)\biggr)\\
\qquad=\phi\biggl(\biggl(\prod_{b\in I_r}\widehat{Z}_{b,0}^{-1}\biggr)\widehat{Q}_{a,m}^2\biggl(\prod_{b\in I_r}\widehat{Z}_b^{\ell_b+1}\biggr)\biggr)
-q^{t_a^{\vee}}\phi\biggl(\prod_{b\in I_r}\widehat{Z}_{b,0}^{-1}\prod_{b\sim a}\widehat{Q}_{b,m}^{-\overline{C}_{ba}}\prod_{b\in I_r}\widehat{Z}_b^{\ell_b+1}\biggr),
\end{gather*}
or equivalently,
\begin{gather}
q^{\frac{1}{2\delta}(2\widehat{F}(\widehat{\mathbf{d}})+\widehat{L}(\widehat{\mathbf{d}})+2\Lambda_{aa})}
\mathcal{M}_{\overline{\lambda},\widehat{\mathbf{d}}}\bigl(q^{-1}\bigr)\nonumber\\
\qquad
=q^{\frac{1}{2\delta}(2\widehat{F}(\widehat{\mathbf{s}})+\widehat{L}(\widehat{\mathbf{s}}))}\mathcal{M}_{\overline{\lambda},\widehat{\mathbf{s}}}\bigl(q^{-1}\bigr)
+q^{\frac{1}{2\delta}(2\widehat{F}(\widehat{\mathbf{k}})+\widehat{L}(\widehat{\mathbf{k}})+2\delta t_a^{\vee})}\mathcal{M}_{\overline{\lambda},\widehat{\mathbf{k}}}\bigl(q^{-1}\bigr).\label{eq:5.8}
\end{gather}
We are now ready to prove Theorem~\ref{1.3}.

\begin{proof}[Proof of Theorem~\ref{1.3}]
Firstly, we have
\begin{gather}
\widehat{F}(\widehat{\mathbf{d}})
=\sum_{c,b\in I_r}\sum_{i\in\mathbb{Z}_+}\Lambda_{cb}d_{c,i}
=2\sum_{b\in I_r}\Lambda_{ab},\label{eq:5.9}\\
\widehat{F}(\widehat{\mathbf{s}})
=\sum_{c,b\in I_r}\sum_{i\in\mathbb{Z}_+}\Lambda_{cb}s_{c,i}
=2\sum_{b\in I_r}\Lambda_{ab},\\
\widehat{F}(\widehat{\mathbf{k}})
=\sum_{c,b\in I_r}\sum_{i\in\mathbb{Z}_+}\Lambda_{cb}k_{c,i}
=-\sum_{b\in I_r,c\sim a}\Lambda_{cb}\overline{C}_{ca}
=-\delta t_a^{\vee}+2\sum_{b\in I_r}\Lambda_{ab},\label{eq:5.11}
\end{gather}
where the last equality in \eqref{eq:5.11} follows from the fact that $\Lambda$ is symmetric, along with the equality~${\Lambda\overline{C}=\delta T^{\vee}}$.

Next, we observe that
\begin{gather}
\widehat{L}(\widehat{\mathbf{d}})
=\sum_{i,j\in\mathbb{Z}_+}\sum_{c,b\in I_r}\min(i,j)\Lambda_{cb}d_{c,i}d_{b,j}\nonumber\\
\phantom{\widehat{L}(\widehat{\mathbf{d}})}{}
=\Lambda_{aa}[\min(m-1,m-1)+2\min(m-1,m+1)+\min(m+1,m+1)]\nonumber\\
\phantom{\widehat{L}(\widehat{\mathbf{d}})}{}=(4m-2)\Lambda_{aa},\\
\widehat{L}(\widehat{\mathbf{s}})
=\sum_{i,j\in\mathbb{Z}_+}\sum_{c,b\in I_r}\min(i,j)\Lambda_{cb}s_{c,i}s_{b,j}
=2^2\Lambda_{aa}\min(m,m)
=4m\Lambda_{aa},\\
\widehat{L}(\widehat{\mathbf{k}})
=\sum_{i,j\in\mathbb{Z}_+}\sum_{c,b\in I_r}\min(i,j)\Lambda_{cb}k_{c,i}k_{b,j}
=\min(m,m)\sum_{c,b\sim a}\Lambda_{cb}\overline{C}_{ba}\overline{C}_{ca}\nonumber\\
\phantom{\widehat{L}(\widehat{\mathbf{k}})}{}
=m\sum_{c,b\neq a}\Lambda_{cb}\overline{C}_{ba}\overline{C}_{ca}
=m\sum_{b\in I_r,c\neq a}\Lambda_{cb}\overline{C}_{ba}\overline{C}_{ca}-2m\sum_{c\neq a}\Lambda_{ca}\overline{C}_{ca}\nonumber\\
\phantom{\widehat{L}(\widehat{\mathbf{k}})}{}=m\sum_{b\in I_r,c\neq a}\delta_{ca}t_a^{\vee}\overline{C}_{ca}-2m\sum_{c\in I_r}\Lambda_{ac}\overline{C}_{ca}+4m\Lambda_{aa}\nonumber\\
\phantom{\widehat{L}(\widehat{\mathbf{k}})}{}=-2m\delta t_a^{\vee}+4m\Lambda_{aa}.\label{eq:5.14}
\end{gather}
Again, the last two equalities in \eqref{eq:5.14} follows from the fact that $\Lambda$ is symmetric, along with the equality $\Lambda\overline{C}=\delta T^{\vee}$. Thus, by \eqref{eq:5.9}--\eqref{eq:5.14}, we have
\begin{equation*}
2\widehat{F}(\widehat{\mathbf{d}})+\widehat{L}(\widehat{\mathbf{d}})+2\Lambda_{aa}
=2\widehat{F}(\widehat{\mathbf{s}})+\widehat{L}(\widehat{\mathbf{s}})
=2\widehat{F}(\widehat{\mathbf{k}})+\widehat{L}(\widehat{\mathbf{k}})+2\delta t_a^{\vee}+2m\delta t_a^{\vee},
\end{equation*}
which implies that \eqref{eq:5.8} reduces to
\begin{equation*}
\mathcal{M}_{\overline{\lambda},\widehat{\mathbf{d}}}\bigl(q^{-1}\bigr)
=\mathcal{M}_{\overline{\lambda},\widehat{\mathbf{s}}}\bigl(q^{-1}\bigr)
-q^{-t_a^{\vee}m}\mathcal{M}_{\overline{\lambda},\widehat{\mathbf{k}}}\bigl(q^{-1}\bigr),
\end{equation*}
or equivalently,
\[
\mathcal{M}_{\overline{\lambda},\widehat{\mathbf{d}}}(q)
=\mathcal{M}_{\overline{\lambda},\widehat{\mathbf{s}}}(q)
-q^{t_a^{\vee}m}\mathcal{M}_{\overline{\lambda},\widehat{\mathbf{k}}}(q).
\]
As \smash{$\mathcal{M}_{\overline{\lambda},\widehat{\mathbf{n}}}(q)$} is the graded multiplicity of~$V\bigl(\overline{\lambda}\bigr)$ for all dominant $\mathfrak{g}^{\sigma}$-weights $\overline{\lambda}$, we have
\begin{equation}\label{eq:5.16}
\ch_q\mathcal{F}_{\widehat{\mathbf{d}}}^*
=\ch_q\mathcal{F}_{\widehat{\mathbf{s}}}^*
-q^{t_a^{\vee}m}\ch_q\mathcal{F}_{\widehat{\mathbf{k}}}^*.
\end{equation}
Theorem~\ref{1.3} now follows from \eqref{eq:5.16}.
\end{proof}

\begin{remark*}
Kus and Venkatesh \cite[Proposition 7.3]{KV16} obtained a short exact sequence of fusion product of KR-modules over $\mathfrak{g}[t]^{\sigma}$ that generalizes the \smash{$X_m^{(\kappa)}$} $Q$-system relations \eqref{eq:1.1} as follows:
\begin{equation*}
0\longrightarrow(K_{a,m}^{\sigma})^*\longrightarrow\KR_{a,m}^{\sigma}*\KR_{a,m}^{\sigma}\longrightarrow\KR_{a,m+1}^{\sigma}*\KR_{a,m-1}^{\sigma}\longrightarrow0.
\end{equation*}
By applying the character map to the above exact sequence, we see that the resulting identity of characters coincides with the identity stated in Theorem~\ref{1.3} when $q=1$.
\end{remark*}

\subsection*{Acknowledgments}

The author would like to thank Phillipe Di Francesco and Rinat Kedem for their helpful clarifications and illuminating discussions throughout the project. The author would also like to thank the anonymous referees for their careful reading of the manuscript and helpful suggestions to improve the exposition of this paper. Part of this work was completed while the author was a graduate student at the University of Illinois Urbana-Champaign, where the author was supported by a graduate fellowship from A*STAR (Agency for Science, Technology and Research, Singapore), and this work was also supported in part by the US National Science Foundation (DMS-1802044).

\pdfbookmark[1]{References}{ref}
\LastPageEnding

\end{document}